\documentclass[11pt,a4paper]{amsart}
\usepackage[utf8]{inputenc}
\usepackage[english]{babel}
\usepackage{amsmath,amssymb,amsthm,amscd,graphicx,mathrsfs,url,dsfont,enumitem}
\usepackage[usenames,dvipsnames]{color}
\usepackage[colorlinks=true,linkcolor=Blue,citecolor=Green,urlcolor=Blue]{hyperref}
\usepackage{dsfont}
\usepackage[super]{nth}
\usepackage[open, openlevel=2, depth=3, atend]{bookmark}
\hypersetup{pdfstartview=XYZ}
\usepackage{epigraph}
\usepackage{mathtools}
\usepackage{enumitem}
\usepackage{mathrsfs}
\usepackage{comment}

\let\oldtocsection=\tocsection
\let\oldtocsubsection=\tocsubsection
\let\oldtocsubsubsection=\tocsubsubsection
\renewcommand{\tocsection}[2]{\hspace{0em}\oldtocsection{#1}{#2}}
\renewcommand{\tocsubsection}[2]{\hspace{1em}\oldtocsubsection{#1}{#2}}
\renewcommand{\tocsubsubsection}[2]{\hspace{2em}\oldtocsubsubsection{#1}{#2}}

\def\?[#1]{\textbf{[#1]}\marginpar{\Large{\textbf{??}}}}

\DeclareMathSymbol{\eset}{\mathalpha}{AMSb}{"3F}     
\renewcommand{\d}{\mathrm{d}}             

\numberwithin{equation}{section}

\newtheorem{theorem}{Theorem}[section]
\newtheorem{lemma}[theorem]{Lemma}
\newtheorem{proposition}[theorem]{Proposition}
\newtheorem{corollary}[theorem]{Corollary}

\theoremstyle{remark}

\theoremstyle{definition}
\newtheorem{definition}[theorem]{Definition}

\newcommand{\C}{\mathbb{C}}
\newcommand{\D}{\mathbb{D}}
\newcommand{\R}{\mathbb{R}}
\newcommand{\Z}{\mathbb{Z}}
\renewcommand{\H}{\mathbb{H}}
\newcommand{\N}{\mathbb{N}}
\newcommand{\Q}{\mathbb{Q}}

\newcommand{\E}{\mathbb{E}}
\renewcommand{\P}{\mathbb{P}}
\renewcommand{\S}{\mathbb{S}}
\newcommand{\T}{\mathbb{T}}
\renewcommand{\Im}{\mathrm{Im}}
\renewcommand{\Re}{\mathrm{Re}}

\newcommand{\id}{\mathrm{Id}}

\newcommand{\cT}{\mathcal{T}}

\newcommand{\cE}{\mathcal{E}}
\newcommand{\cC}{\mathcal{C}}

\newcommand{\cF}{\mathcal{F}}

\newcommand{\cN}{\mathcal{N}}
\newcommand{\cA}{\mathcal{A}}
\newcommand{\cV}{\mathcal{V}}

\newcommand{\ind}{\mathds{1}}

\newcommand{\cD}{\mathcal{D}}

\newcommand{\cL}{\mathcal{L}}

\newcommand{\laweq}{\overset{\text{law}}{=}}

\newcommand{\cS}{\mathcal{S}}

\newcommand{\rL}{\mathrm{L}}

\renewcommand{\hat}{\widehat}

\newcommand{\del}{\partial}
\newcommand{\delbar}{\bar{\partial}}

\newcommand{\bJ}{\mathbf{J}}

\newcommand{\bS}{\mathbf{S}}

\newcommand{\bD}{\mathbf{D}}

\newcommand{\diff}{\mathrm{Diff}}
\newcommand{\norm}[1]{\left\Vert #1\right\Vert}

\newcommand{\cU}{\mathcal{U}}

\newcommand{\bT}{\mathbf{T}}

\newcommand{\bA}{\mathbf{A}}
\newcommand{\bL}{\mathbf{L}}

\newcommand{\bP}{\mathbf{P}}

\newcommand{\bz}{\mathbf{z}}

\newcommand{\sL}{\mathscr{L}}

\newcommand{\cJ}{\mathcal{J}}

\newcommand{\rM}{\mathrm{m}}

\renewcommand{\det}{\mathrm{det}}
\newcommand{\rv}{\mathrm{v}}

\newcommand{\rmint}{\mathrm{int}}

\newcommand{\rmext}{\mathrm{ext}}

\newcommand{\bk}{\mathbf{k}}
\renewcommand{\exp}{\mathrm{exp}}

\newcommand{\ru}{\mathrm{u}}

\newcommand{\ra}{\mathrm{a}}
\newcommand{\rb}{\mathrm{b}}
\newcommand{\kl}{\mathfrak{l}}
\newcommand{\sR}{\mathscr{R}}
\newcommand{\bK}{\mathbf{K}}

\newcommand{\eps}{\epsilon}

\DeclareMathOperator{\RN}{RN}

\newcommand{\indic}[1]{\mathbf{1}_{#1}}
\DeclareMathOperator{\dist}{dist}

\author{Guillaume Baverez}
\address{Aix--Marseille Universit\'e, I2M, Marseille, France}
\email{guillaume.baverez@univ-amu.fr}
\author{Antoine Jego}
\address{CNRS \& CEREMADE, Université Paris-Dauphine, PSL University, France}
\email{antoine.jego@dauphine.psl.eu}

\title[Conformal welding and matter--Liouville--ghost]{Conformal welding and the\\ matter--Liouville--ghost factorisation}

\date{}


\begin{document}

\begin{abstract}
We study the action of local conformal transformations on several measures related to the Gaussian free field and Schramm--Loewner evolutions. The main novelty of our work is a Cameron--Martin-type formula for the welding homeomorphism of the SLE loop measure ($\kappa\in(0,4]$); its proof relies on a rigorous interpretation (and computation) of the ``Jacobian" of the conformal welding map, which we relate to the ``$bc$-ghost system" from bosonic string theory. We also give an intrinsic definition of the trace of the GFF on SLE, and prove a characterisation of the free boundary GFF in $\D$. As an application, we introduce a new and intrinsic approach to the conformal welding of quantum surfaces.
\end{abstract}

\maketitle

\setcounter{tocdepth}{1}
\tableofcontents

\section{Introduction}

This article stems from a recent line of research pioneered by Yilin Wang (see \cite{MR4767165} for a review) exploring the rich links relating random conformal geometry, Teichm\"uller theory and Schramm--Loewner evolutions. The crux of these connections is the equality between the universal Liouville action (a K\"ahler potential for the Weil--Petersson metric on the universal Teichm\"uller space \cite{TakhtajanTeo06}), and the Loewner energy (the Dirichlet energy of the SLE driving function) \cite{Wang19}.
The Loewner energy can also be seen as a dissipation term for the Dirichlet energies during cutting or welding operations \cite{ViklundWang19}. This particular identity corresponds precisely to the semiclassical limit $\kappa\to0$ of Sheffield's celebrated quantum zipper \cite{sheffield2016}.
However, Viklund--Wang's identity does not explain Sheffield's result for a fixed $\kappa \in (0,4]$, in particular because of apparently mismatching constants in the action functional (in \eqref{E:VW20}, naively one would expect to have $6Q^2-25$ instead of $6Q^2$).
One of our main motivations is to uncover the reason for this anomaly, by computing what amounts formally to the Jacobian determinant of the welding map.
The rigorous formulation of this computation takes the form of an integration by parts formula for the welding homeomorphism of SLE. As an application, we give a new proof that the conformal welding of quantum discs produces a quantum sphere decorated by an independent SLE (originally due to \cite{AHS20}, building on techniques of Sheffield's original paper).
While Sheffield's proof is probabilistic, our approach is more algebraic (borrowing from the representation theory of the Virasoro algebra) and makes explicit connections with conformal field theory. In particular, we can relate the Jacobian of the welding map to the $bc$-ghost system from bosonic string theory. See Section \ref{par:bc} for a detailed discussion.

    \subsection{Background}\label{SS:background}

We start by giving some background on the main characters of this paper.  Throughout the paper, the following parameters are fixed once and for all:
\begin{equation}\label{E:constants}
\begin{aligned}
    \gamma=\sqrt{\kappa}\in(0,2];\qquad & \qquad Q=\frac{\gamma}{2}+\frac{2}{\gamma}\ge 2;\\
    c_\rM:=1-6\Big(\frac{2}{\sqrt{\kappa}}-\frac{\sqrt{\kappa}}{2}\Big)^2\le 1;\qquad & \qquad c_\rL:=1+6Q^2=26-c_\rM\ge 25.
\end{aligned}
\end{equation}

\textbf{Space of Jordan curves, homeomorphisms of $\S^1$ and SLE.}
A Jordan curve in the Riemann sphere $\hat{\C}$ is a continuous simple loop (viewed up to parametrisation). We will denote by $\cJ_{0,\infty}$ the space of Jordan curves separating $0$ from $\infty$, equipped with the Carathéodory topology. For a curve $\eta \in \cJ_{0,\infty}$ separating $0$ from $\infty$, we will denote by $\rmint(\eta)$ and $\rmext(\eta)$ the ``interior'' and ``exterior'' of $\eta$ which are the connected components of $\hat{\C} \setminus \eta$ containing $0$ and $\infty$ respectively.
We will often consider conformal maps $f : \D \to \rmint(\eta)$ and $g:\D^* \to \rmext(\eta)$, where $\D^* = \hat\C \setminus \overline\D$, fixing $0$ and $\infty$ respectively.
The maps $f$ and $g$ are uniquely defined from $\eta$ up to pre-composition by a rotation.
By Carathéodory's theorem, these maps extend continuously to homeomorphisms between the closures $\overline\D \to \overline{\rmint(\eta)}$ and $\overline{\D^*} \to \overline{\rmext(\eta)}$.
The homeomorphism $h = g^{-1} \circ f : \S^1 \to \S^1$ is called a \emph{welding homeomorphism} of $\eta$.

Conversely, if $h \in \text{Homeo}(\S^1)$ is a homeomorphism of $\S^1$, one can seek for a curve $\eta$ and conformal maps $f$ and $g$ uniformising the interior and exterior of $\eta$ and such that $h=g^{-1}\circ f$. This reverse direction is known as conformal welding. Depending on the homeomorphism $h$,
neither existence nor uniqueness of such a welding necessarily holds. Moreover, by post-composing $f$ and $g$ by a given Möbius map of $\hat\C$, we see that uniqueness can only hold up to Möbius transformations.
We will denote by
\begin{equation}\label{E:Homeo^star}
    \text{Homeo}^\star(\S^1) := \{ h \in \text{Homeo}(\S^1) \text{ welding a unique (up to Möbius maps) Jordan curve} \}.
\end{equation}

Schramm--Loewner Evolutions (SLE) are random fractal curves of the plane, initially introduced by Schramm \cite{Schramm2000}, which describe the scaling limit of interfaces of critical models of statistical mechanics. In this article, we will consider the SLE loop measure, as introduced by \cite{Werner08_loop,BenoistDubedat16,kemppainen2016nested,Zhan21_SLEloop}, focusing on the case $\kappa \le 4$ where the curve is simple. Its restriction $\nu$ to the space of loops separating $0$ from $\infty$ is a Borel measure on $\cJ_{0,\infty}$ \cite[Proposition~B.1]{BJ24}. The measure $\nu$ is infinite due to small and large loops. More precisely, by scale invariance, the marginal of the logarithm of the conformal radius of the complement of the curve seen from the origin is a constant multiple of Lebesgue measure $\d c$ on $\R$. Fixing this constant to be $\pi^{-1/2}$, the measure $\nu$ then decomposes as
\begin{equation}
    \label{E:shape_measure}
\nu = \pi^{-1/2} \d c \otimes \nu^\#,
\end{equation}
where $\nu^\#$ is the \textit{SLE shape measure}: a probability measure on the space of curves $\eta$ in $\cJ_{0,\infty}$ with unit conformal radius at the origin. Equivalently (and as done in \cite{BJ24}), we can view $\nu^\#$ as a measure on a space of conformal maps of the disc $\cE$, defined in \eqref{eq:def_E}. Samples of $\nu^\#$ are fractal curves of almost sure Hausdorff dimension $1+\kappa/8$ \cite{Beffara08}. Moreover, for $\nu^\#$-almost every curve $\eta$, the $(1+\kappa/8)$-dimensional Minkowski content $\mathfrak{m}_\eta$ of $\eta$, as constructed in \cite{LawlerRezaei15_Minkowski}, is a nondegenerate measure supported on the curve $\eta$ defined by
\begin{equation}
    \label{E:Minkowksi_SLE}
    \mathfrak{m}_\eta = \lim_{r \to 0} r^{-2+(1+\frac{\kappa}{8})} \ind_{ \mathrm{dist}(z,\eta) < r} |\d z|^2.
\end{equation}

\textbf{Gaussian free field.}
The Gaussian free field is a random centred Gaussian field. Several variants are possible including:
\begin{itemize}[leftmargin=*]
\item The GFF on a connected compact surface $(\Sigma,g)$ with some background Riemannian metric $g$. Its covariance is given by the Green function of the Laplace--Beltrami operator in $(\Sigma,g)$. In this article, we will focus on the Riemann sphere $\hat\C$. For concreteness,
\[\E[X(z)X(\zeta)]=-\log|z-\zeta|+\log|z|_++\log|\zeta|_+, \quad z, \zeta \in \hat\C,\]
corresponding to the metric $(1\vee|z|)^{-4}|\d z|^2$ \cite[Section 2.1]{KRV_DOZZ}.
\item The GFF with Neumann (or free) boundary conditions (resp. Dirichlet boundary conditions) in a proper, simply connected domain of $\C$. Its covariance is the Green function with Neumann (resp. Dirichlet) boundary conditions. For instance, for the unit disc $\D$, it is given by
\[ 
(z,\zeta)\mapsto -\log\big|(z-\zeta)(1-z\bar\zeta)\big|,
\qquad \text{resp.} \qquad (z,\zeta)\mapsto\log\Big|\frac{1-z\bar\zeta}{z-\zeta}\Big|.
\]
\end{itemize}
In all the above cases, the GFF takes values in a suitable Sobolev space of negative index. That is to say, it is a random generalised function but not a function with well-defined pointwise values. In addition, on compact surfaces or in a domain with Neumann boundary conditions, the associated Laplace operator vanishes on constant functions. In theses cases, the GFF is then only defined up to a global additive constant that we will fix in some specific way depending on our needs.
We refer to the book \cite{berestycki2024gaussian} for more details on the Gaussian free field.

\textbf{Random surfaces.} 
The theory of Gaussian multiplicative chaos (GMC), initiated by Kahane \cite{Kahane85} and rediscovered by Duplantier and Sheffield \cite{DuplantierSheffield11} (see \cite{berestycki2024gaussian, rhodes2014_gmcReview, Powell20_review} for introductions to this topic), defines random measures of the form
\begin{equation}
    \label{E:GMC}
e^{\gamma X} v = \lim_{\eps \to 0} \eps^{\gamma^2/2} e^{\gamma X_\eps(z)} v,
\end{equation}
where $\gamma \in (0,2)$ is a parameter, $X$ is some variant of the Gaussian free field and $v$ is the volume form of the underlying surface/domain. In the above limiting procedure, $X_\eps$ is an $\eps$-smoothing of the field $X$ (e.g. using convolution approximations), and the limit is in probability for the topology of weak convergence of measures. This leads to a notion of ``random surface'' where the volume form is replaced by the measure \eqref{E:GMC}. One can also construct the associated random metric \cite{MR4179836, MR4199443}, fully making sense of a random metric measure space.
As understood by Polyakov in his seminal paper \cite{Polyakov81},
these ``LQG surfaces'' are natural random perturbations of the background surface. They are conjectured to in some rare cases, are proved to) describe the scaling limit of large random planar maps and have deep connections with many mathematical objects. We refer to the expository articles \cite{gwynne2020random, ding2022introduction, GHS23, sheffield2022random} for more background, or again to the book \cite{berestycki2024gaussian}.

Due to the renormalisation procedure needed to define these GMC measures, they obey the following distorted change-of-variables formula. Let $X$ be a GFF with Dirichlet boundary conditions, say, in some proper, simply connected domain $D$ and let $\Phi : D \to D'$ be a conformal map. Then, the pullback of $e^{\gamma X} |\d z|^2$ by $\Phi$ is the GMC measure associated to the field $X \circ \Phi + Q \log |\Phi'|$ where, as in \eqref{E:constants}, $Q = \gamma/2+2/\gamma$. See e.g. \cite[Theorem 2.8]{berestycki2024gaussian}.
This motivates the introduction of the following actions:
\begin{itemize}[leftmargin=*]
    \item
The space of maps $\Phi:\hat\C \to \hat\C$ conformal in some domain $D$ acts on the space of generalised functions (``fields'') $\varphi$ in $D$ by
\begin{equation}
    \label{E:dot}
    \varphi \cdot \Phi := \varphi \circ \Phi + Q\log |\Phi'|.
\end{equation}
    \item
The space $\diff^\omega(\S^1)$ of analytic diffeomorphisms $h$ of $\S^1$ acts on the space of fields $\varphi$ on $\S^1$ by
\begin{equation}
    \label{E:doth}
    \varphi \cdot h := \varphi \circ h + Q\log \frac{zh'}{h}.
\end{equation}
\end{itemize}
These two actions are equivalent in the sense that, if the restriction $h = \Phi\vert_{\S^1}$ of a map $\Phi$ conformal in a neighbourhood of $\S^1$ is a diffeomorphism of $\S^1$, then for all $z \in \S^1$, $|\Phi'(z)| = zh'(z)/h(z)$.
We thus use the same notation for both actions. The chain rule shows that these actions are indeed group actions in the sense that, for all maps $\Phi_1$, $\Phi_2$ conformal in suitable domains and field $\varphi$,
\begin{equation}\label{E:cdot_group}
    (\varphi \cdot \Phi_1) \cdot \Phi_2 = \varphi \cdot (\Phi_1 \circ \Phi_2).
\end{equation}

    \subsection{Welding homeomorphism of SLE}\label{SS:intro_homeo}

The next definition introduces the law of the welding homeomorphism of the SLE$_\kappa$-loop, which we denote $\tilde{\nu}_\kappa$ or simply $\tilde{\nu}$. Here and in the sequel, we will identify a point $e^{i\alpha}\in\S^1$ with the rotation $z\mapsto e^{i\alpha}z$ in $\mathrm{Diff}(\S^1)$.
\begin{definition}\label{def:tilde_nu}
    Let $\kappa\in(0,4]$. Let $\eta\sim\nu_\kappa^\#$ and $\alpha$ an independent uniform random variable in $[0,2\pi)$. Let $f:\D\to\mathrm{int}(\eta)$, $g:\D^*\to\mathrm{ext}(\eta)$ be the normalised Riemann mappings of $\eta$ ($f(0)=0,f'(0)>0,g(\infty)=\infty,g'(\infty)>0$). The probability measure $\tilde{\nu}_\kappa$ on $\mathrm{Homeo}(\S^1)$ is the law of $h:=e^{-i\alpha}\circ g^{-1}\circ f|_{\S^1}$.
\end{definition}
By conformal removability of SLE \cite{RohdeSchramm05,kavvadias2022conformalremovabilitysle4}, samples of $\tilde\nu_\kappa$ are actually the welding homeomorphism of a unique curve, up to Möbius transformations. With the notation \eqref{E:Homeo^star}, this means that $\tilde{\nu}_\kappa$ gives full mass to $\mathrm{Hoemo}^\star(\S^1)$. With the definition above, we can view $\tilde{\nu}$ equivalently as the measure $\frac{\d\alpha}{2\pi}\otimes\nu^\#$ on $\S^1\times\cE$.

By definition, $\tilde\nu$ is rotationally left-invariant. On the other hand, since $\nu^\#$ is rotationally invariant, $\tilde{\nu}$ is invariant under conjugation by any rotation (i.e. $e^{-i\beta}\circ h\circ e^{i\beta}\laweq h$ for any $e^{i\beta}\in\S^1$). It follows that $\tilde{\nu}$ is rotationally right-invariant as well. However, $\tilde{\nu}$ is neither left or right invariant under the action of the whole M\"obius group $\mathrm{PSL}_2(\R)$ (for otherwise it would be infinite), see Lemma \ref{L:mobius} for a convenient description of the lift of $\tilde{\nu}$ to the invariant measure. This section states the main properties of $\tilde{\nu}$, with proofs postponed to Section \ref{S:ghost}.

As a motivating example, consider the case $\kappa=\frac{8}{3}$ ($c_\rM=0$). Since the central charge vanishes, we can view $\nu_{8/3}$ as the ``uniform measure on Jordan curves", or more precisely the invariant measure under local conformal transformations. One could naively expect $\tilde{\nu}_{8/3}$ to be an invariant measure under left/right composition, i.e. to be a Haar measure on $\mathrm{Homeo}(\S^1)$. This is of course not possible since $\mathrm{Homeo}(\S^1)$ is not locally compact. Intuitively, we are missing the Jacobian of the change of coordinates between the invariant coordinate on $\cJ_{0,\infty}$ (under local conformal deformations) and the invariant coordinates on $\mathrm{Homeo}(\S^1)$. The main contribution of this section is to give a precise meaning to this idea, and this will take the form of an integration by parts formula for $\tilde{\nu}_\kappa$. It allows us to assign the value $\exp(-\frac{13}{12\pi}\bS_1-2\bK)$ to this determinant, where $\bS_1$ is the \emph{universal Liouville action} \cite[Eq. (0.1)]{TakhtajanTeo06}, and $\bK$ is the \emph{logarithmic capacity} \cite{KirillovYurev87}: for a homeomorphism $h=g^{-1}\circ f$ in $\mathrm{Homeo}^\star(\S^1)$
with $f(0)=0$ and $g(\infty)=\infty$, these quantities are defined as
\begin{equation}\label{E:def_K_S1}
    \bK(h)=\log\left|\frac{g'(\infty)}{f'(0)}\right|\qquad\text{and}\qquad\bS_1(h)=\int_\D\left|\frac{f''}{f'}\right|^2|\d z|^2+\int_{\D^*}\left|\frac{g''}{g'}\right||\d z|^2-4\pi\log\left|\frac{g'(\infty)}{f'(0)}\right|.
\end{equation}
The maps $f$ and $g$ are not uniquely determined from $h$: post-composing $f$ and $g$ by a Möbius fixing 0 and $\infty$ yields the same homeomorphism. However, the quantities $\bK(h)$ and $\bS_1(h)$ above do not depend on this undetermined Möbius.
Both $\bK$ and $\bS_1$ are very natural quantities in the theory of the universal Teichm\"uller space: the first is a potential for the Velling--Kirillov metric \cite[Chapter 1, Theorem 5.3]{TakhtajanTeo06}, while the second is a potential for the Weil--Petersson metric \cite[Chapter 2, Theorem 3.8]{TakhtajanTeo06}.
Moreover, $\bS_1(h)$ is finite if and only if $h$ belongs to the Weil--Petersson Teichmüller space.
Thus, one can think heuristically of $\tilde{\nu}_{8/3}$ as the formal path integral
\[
``\;\d\tilde{\nu}_{8/3}(h)=\exp\left(-\frac{13}{12\pi}\bS_1-2\bK\right)Dh\;",
\]
where $Dh$ is the non-existent Haar measure on $\mathrm{Homeo}(\S^1)$. For a generic value of $\kappa \in (0,4]$, $\tilde{\nu}_\kappa$ is related to $\tilde{\nu}_{8/3}$ by the path integral
$
\d\tilde{\nu}_\kappa = \exp\left(\frac{c_\rM}{24\pi}\bS_1\right) \d \tilde\nu_{8/3}
$ \cite{CarfagniniWang, BJ24, GQW25} (this is only formal since the two measures are mutually singular for $\kappa\neq\frac{8}{3}$),
so we can more generally think of $\tilde{\nu}_\kappa$ via the path integral
\begin{equation}\label{eq:path_integral}
``\;\d\tilde{\nu}_\kappa(h)=\exp\left(\frac{c_\rM-26}{24\pi}\bS_1-2\bK\right)Dh\;".
\end{equation}

The formula for the determinant can be related to the ``ghost determinant" from bosonic string theory, which is supposed to contribute $-26$ to the central charge and $-1$ to the conformal weight in order to get an integration measure on Teichm\"uller space \cite[Section 2.2]{VirasoroMinimalString}. In our formula, these contributions are given respectively by the universal Liouville action and the Velling--Kirillov potential. In the probabilistic literature, this has been used implicitly and in a remarkable way in \cite{AngCaiSunWu_1,AngCaiSunWu_2} in the proof of the structure constant for the conformal loop ensemble (CLE). As explained to us by Xin Sun, the factor $e^{-2\bK}$ can be related to the Haar measure on the M\"obius group used in these papers. Theorem \ref{T:ghost} below can be understood as the hidden structure explaining why the coupling used in these papers is so relevant.

Let us now close this heuristic discussion and state the main result. In the statement, the operators $\sL_\mu$ and $\bar\sL_\mu$ (resp. $\sR_\mu$ and $\bar\sR_\mu$) correspond to the Lie derivative along the left (resp. right) composition by infinitesimal homeomorphisms generated by a Beltrami differential $\mu$. These operators are defined precisely in \eqref{E:def_sL_sR}. Before stating the result, we borrow some notation from \cite{BJ24}: 
given a welding homeomorphism $h = g^{-1} \circ f \in \text{Homeo}^{\star}(\S^1)$ (whose welding curve is unique), and given a Beltrami differential $\mu$ compactly supported in $\D$ (resp. $\D^*$), we define
\begin{align*}
&\vartheta_h(\mu):=-\frac{1}{\pi}\int_\D\cS f(z)\mu(z)|\d z|^2;&&\tilde{\vartheta}_h(\mu):=-\frac{1}{\pi}\int_{\D^*}\cS g(z)\mu(z)|\d z|^2;\\
&\varpi_h(\mu):=-\frac{1}{\pi}\int_\D\left(\frac{f'(z)^2}{f(z)^2}-\frac{1}{z^2}\right)\mu(z)|\d z|^2;&&\tilde{\varpi}_h(\mu):=-\frac{1}{\pi}\int_{\D^*}\left(\frac{g'(z)^2}{g(z)^2}-\frac{1}{z^2}\right)\mu(z)|\d z|^2.
\end{align*}
Here, $\cS f$ denotes the \emph{Schwarzian derivative} of $f$:
\begin{equation}
    \label{E:Schwarzian}
\cS f = \Big(\frac{f''}{f'}\Big)' - \frac{1}{2} \Big( \frac{f''}{f'} \Big)^2.
\end{equation}
Note that the above quantities depend only on $h$ and not on the choice of representatives $f$ and $g$.

\begin{theorem}\label{T:ghost}
For all Beltrami differentials $\mu$ compactly supported in $\D^*$ (resp. $\D$), 
the $L^2(\tilde{\nu}_\kappa)$-adjoint $\sL_\mu^*$ of $\sL_\mu$ (resp. the $L^2(\tilde{\nu}_\kappa)$-adjoint $\sR_\mu^*$ of $\sR_\mu$) is densely defined on $L^2(\tilde{\nu}_\kappa)$ and given by:
\begin{equation}\label{E:T_ghost}
\sL_\mu^*=-\bar\sL_\mu+\Big(\frac{c_\rL}{12}\overline{\tilde\vartheta(\mu)}+\overline{\tilde\varpi(\mu)}\Big) \id \qquad\text{resp.}\qquad\sR_\mu^*=-\bar\sR_\mu-\Big(\frac{c_\rL}{12}\overline{\vartheta(\mu)}+\overline{\varpi(\mu)}\Big)\id.
\end{equation}
\end{theorem}

Owing to the fact that (see \cite[Chapter 1, Theorem 5.3 and Chapter 2, Theorem 3.8]{TakhtajanTeo06})
\begin{align}\label{E:TT06}
\frac{c_\rL}{12}\vartheta(\mu)+\varpi(\mu)=-\sR_\mu\left(\frac{c_\rL}{24\pi}\bS_1+2\bK\right)\qquad\text{and}\qquad\frac{c_\rL}{12}\tilde\vartheta(\mu)+\tilde\varpi(\mu)=\sL_\mu\left(\frac{c_\rL}{24\pi}\bS_1+2\bK\right),
\end{align}
these relations are exactly the ones we would obtain if we could make sense literally of \eqref{eq:path_integral}, and give a mathematically rigorous interpretation of this path integral. For the reader's convenience, this formal computation is written in Appendix \ref{appendix:formal}. In fact, it is then straightforward to integrate the identity of Theorem \ref{T:ghost} to get a variational formula for the welding homeomorphism of SLE under left/right composition by analytic diffeomorphisms. In the next statement, we define 
\begin{equation}
    \label{E:Omega}
\Omega(h_2,h_1):=\frac{c_\rL}{24\pi}(\bS_1(h_2)-\bS_1(h_1))+2(\bK(h_2)-\bK(h_1))
\end{equation}
if $h_1,h_2\in\mathrm{Diff}^\omega(\S^1)$. The variation of $\bS_1$ is equal to the variation of the normalised mass of Brownian loops intersecting the welding curve of $h_1$ and the welding curve of $h_2$ \cite{FieldLawler12_Reversed, AngParkPfefferSheffield}.
If $h$ is a sample from $\tilde{\nu}_\kappa$ and $\tilde{\Phi}\in\mathrm{Diff}^\omega(\S^1)$ is fixed, both $\bS_1(\tilde{\Phi}\circ h)$ and $\bS_1(h)$ are infinite. However, the variation of the mass of Brownian loop is finite, which means that $\Omega(\tilde{\Phi}\circ h,h)$ is well defined and finite for $\tilde{\nu}_\kappa$-a.e. $h$ and all $\tilde{\Phi}\in\mathrm{Diff}^\omega(\S^1)$.
\begin{corollary}\label{C:RN_welding}
    Let $\tilde{\Phi}\in\mathrm{Diff}^\omega(\S^1)$ and $h$ sampled from $\tilde{\nu}_\kappa$. Then, the laws of $\tilde{\Phi}\circ h$ and $h\circ\tilde{\Phi}$ are respectively given by
    \[e^{-\Omega(\tilde{\Phi}\circ h,h)}\d\tilde{\nu}_\kappa(h)\qquad\text{and}\qquad e^{-\Omega(h\circ\tilde{\Phi},h)}\d\tilde{\nu}_\kappa(h).\]
\end{corollary}
We will prove this corollary in Section \ref{S:ghost}.
As communicated to us by the authors, \cite[Theorem 1.2]{FanSung25} shows that the laws of $\tilde{\Phi}\circ h$ and $h\circ\tilde{\Phi}$ are absolutely continuous with respect to $\tilde{\nu}$ for less regular $\tilde{\Phi}$ (namely, in the Weil--Petersson class), without exhibiting the expression for the Radon--Nikodym derivative. Their study was developed independently, and their strategy is in fact orthogonal to ours. They first establish an analogous result for the restriction of the Neumann GFF on $\S^1$. Then, they ``lift'' this property to the welding homeomorphism of SLE using the conformal welding result of \cite{AngCaiSunWu_1}. Using a limiting procedure, it should be possible to extend Corollary \ref{C:RN_welding} to the class of homeomorphisms exhibited in the main result of \cite{FanSung25}.

In \cite[Section 2]{BJ24}, we considered a representation of the Virasoro algebra with central charge $c_\rM$ and used it to write $L^2(\nu)$ as a direct integral of highest-weight representations of weight in $\lambda\in i\R$. Doing a modification of the operators $\sL_\mu$, $\sR_\mu$ similar to the one appearing in \cite[Section 2.4]{BJ24}, it should be possible to obtain a representation of the Virasoro algebra with central charge $c_\rL=26-c_\rM$, and write $L^2(e^{2\bK}\tilde{\nu})$ as a direct integral of highest-weight representations with weight $\lambda\in 1-i\R$. This is an instance of the duality mentioned in \cite[Remark 2.4]{FeiginFuchs_verma}, relating the Verma module of charge $c$ and weight $\lambda$ to the Verma module of charge $26-c$ and weight $1-\lambda$. In particular, for the degenerate conformal weight $\lambda_{r,s}(\kappa)=\frac{\kappa}{16}(1-r^2)+\frac{1}{\kappa}(1-s^2)+\frac{1}{2}(1-rs)$, we have $1-\lambda_{r,s}(-\kappa)=\lambda_{r,s}(\kappa)$.

  \subsection{The GFF on SLE and its variation}\label{SS:intro_gff}
Let $X$ be the GFF in $\hat\C$ equipped with the metric $\frac{|\d z|^2}{(1\vee|z|)^4}$ (see Section \ref{SS:background}). Let $\eta$ be a (deterministic) analytic Jordan curve. It is well known that one can define the trace of $X$ on $\eta$: this is also a $\log$-correlated field whose covariance is the restriction of Green's function to $\eta$. More intrinsically, its covariance is the resolvent of the jump operator across the curve.

Here, we want to generalise this construction to the case where $\eta$ is a sample from the SLE loop measure $\nu$. Samples of $\nu$ are fractal curves, so it is not clear \textit{a priori} how to define a jump operator across an SLE loop. In Section \ref{SS:jump}, we address this question by constructing a Dirichlet form $(\cE_\eta,\cF_\eta)$ on $L^2(\eta,\mathfrak{m}_\eta)$, where $\mathfrak{m}_\eta$ is the Minkowski content of the curve $\eta$ \eqref{E:Minkowksi_SLE}.
Taking the Friedrichs extension of this quadratic form allows us to define a self-adjoint operator $\bD_\eta$ on $L^2(\eta,\mathfrak{m}_\eta)$ corresponding to the jump operator across the curve. This operator allows us to define a family of ``Sobolev spaces" $H^s(\eta,\mathfrak{m}_\eta)=(\mathrm{Id}+\bD_\eta)^{-s}L^2(\eta,\mathfrak{m}_\eta)$, as well as a Gaussian measure $\P_\eta$ associated to the inner-product induced by $\bD_\eta$. Samples of $\P_\eta$ can be viewed intrinsically as elements of $H^{s}(\eta,\mathfrak{m}_\eta)$ for $s<-1/2$, or as random harmonic functions in $\C\setminus\eta$. Because $\ker(\bD_\eta)=\R$, the field is defined only up to constant. It will be convenient for us to fix the constant by imposing that the harmonic extension at $\infty$ vanishes, i.e. $\bP_\eta\varphi(\infty)=0$. For reasons that will become clear in the sequel, it is natural to ``sample" the zero mode from the measure $e^{-2Qa}\d a$, i.e. we consider the tensor product $e^{-2Qa}\d a\d\P_\eta$ whose samples are fields $\varphi+a$ with $\bP_\eta\varphi(\infty)=0$.

Now, let $\Phi$ be a conformal map defined in a neighbourhood of $\eta$, and let 
\begin{equation}\label{E:def_P_eta_Phi}
    \P_{\eta,\Phi} \text{ be the law of } \varphi\cdot\Phi=\varphi\circ\Phi+Q\log|\Phi'| \text{ when } \varphi \text{ is sampled from } \P_{\Phi(\eta)}.
\end{equation}
Our main result is a formula for the infinitesimal variation of the family of measures $\Phi\mapsto\P_{\eta,\Phi}$. In words, Theorem~\ref{T:GFF_SLE} below shows that this differential is the sum of the stress--energy tensor $\bT_\eta\varphi$ (see Section \ref{SS:stress}) and the Schwarzian derivative $\cS f^{-1}$ \eqref{E:Schwarzian} of the inverse of the uniformising map $f: \D \to \rmint(\eta)$ with $f(0)=0$, $f'(0)>0$.

\begin{theorem}\label{T:GFF_SLE}
For $\nu$-almost every $\eta$, the following holds. 
Let $\mu$ be a Beltrami differential compactly supported in $\mathrm{int}(\eta)$, and for $t \in \C$ small, let $\Phi_t$ be the unique solution to the Beltrami equation with coefficient $t\mu$ with $\Phi_t(0)=0$, $\Phi_t(\infty)=\infty$ and $\Phi_t'(\infty)=1$ (see Section~\ref{SS:qc_maps} for some background on quasiconformal maps). Then, for all $s<-1/2$ and for all $F\in\cC^0(H^{s}(\eta,\mathfrak{m}_\eta))$, we have
\begin{align*}
&\int_\R\E_{\eta,\Phi_t}[F(\varphi+a)]e^{-2Qa}\d a-\int_\R\E_\eta[F(\varphi+a)]e^{-2Qa}\d a\\
&\qquad=-2\Re\Big(t\int_\R\E_\eta\Big[\Big(\frac{1}{12}\cS f^{-1}+\bT_\eta\varphi,\mu\Big)F(\varphi+a)\Big]e^{-2Qa}\d a\Big)+o(t).
\end{align*}
If $\mu$ is compactly supported in $\mathrm{ext}(\eta)$, normalise $\Phi_t$ by $\Phi_t(0)=0, \Phi_t'(0)=1$ and $\Phi_t(\infty) = \infty$ instead. The above result then also holds in this case when replacing $\frac{1}{12}\cS f^{-1}+\bT_\eta\varphi$ by $\frac{1}{12}\cS g^{-1}+\bT_\eta\varphi+2Q\bJ_\eta\varphi$ (with $\bJ_\eta\varphi(z)=\frac{1}{z}\del_z\bP_\eta\varphi(z)$ \eqref{E:def_stress}).
\end{theorem}

	\subsection{The boundary Neumann GFF}\label{SS:intro_neumann}

Axiomatic characterisations of the law of the GFF has been at the heart of some recent research \cite{berestycki2020characterisation, 10.1214/20-EJP566, aru2022characterisation, aru2024martingale}.
In this section, we describe a novel characterisation of the law of the GFF on $\S^1$ based on the description of the infinitesimal perturbation of the law by the action of $\mathrm{Diff}^\omega(\S^1)$. Contrary to the aforementioned literature, we are able to treat a GFF with Neumann boundary conditions solving \cite[Open Problem 6.4]{berestycki2020characterisation} in the case of the unit circle.

The trace on the unit circle of a Neumann GFF in $\D$ (see Section \ref{SS:background}) is the centred Gaussian field with covariance $(z,\zeta)\mapsto-2\log|z-\zeta|$.
We will denote its law by $\dot\P$, and its samples by $\phi$. We normalise the field to have zero average on the unit circle, so that $\dot\P$ is a probability measure on $\dot{H}^{s}(\S^1)=\{\phi\in H^{s}(\S^1)|\,\int_0^{2\pi}\phi(e^{i\theta})\d\theta=0\}$, for all $s<0$. Writing an expansion in Fourier modes $\phi(z)=2\Re(\sum_{m=1}^\infty\phi_mz^m)$, the coordinates $(\phi_m)_{m\geq1}$ are independent complex Gaussians $\cN_\C(0,\frac{1}{m})$ under $\dot\P$. We may choose the zero mode $c$ from the measure $e^{-Qc}\d c$, and we will denote the sample $\phi+c$. 

Recall from \eqref{E:doth} that we consider the action of $\mathrm{Diff}^\omega(\S^1)$ on the space of distributions on $\S^1$ by 
\[\phi\cdot h:=\phi\circ h+Q\log(zh'/h),\qquad\forall h\in\mathrm{Diff}^\omega(\S^1),\,\phi\in\cC^\infty(\S^1)'.\]
Taking Lie derivatives along the infinitesimal generators of this action, we get a representation $(\cD_n)_{n\in\Z}$ of the Witt algebra, acting as (unbounded) operators on $L^2(e^{-Qc}\d c\otimes\dot\P)$. See Section \ref{SS:witt} for the details of this construction, and Lemma \ref{lem:witt} for the commutation relations. By ``Laplace transforming" in the variable $c$, we obtain for each $\alpha\in\C$ a Witt representation $(\cD_{n,\alpha})_{n\in\Z}$ on $L^2(\dot\P)$. In Lemma \ref{L:adjoint_D_alpha}, we compute the adjoints of these operators on $L^2(\dot\P)$. The operators $\cD_{n,\alpha}$ and their adjoints are first order differential operators with polynomial coefficients in $\phi_m,\bar\phi_m$. Theorem \ref{T:Neumann} below shows that these relations characterise $\dot\P$ in the strongest possible sense. In order to state the result, we introduce the \emph{Kac tables} $kac^+,kac^-\subset\C$: 
\begin{equation}
    \label{E:kac}
kac^\pm:=\left\lbrace(1\pm r)\frac\gamma2+(1\pm s)\frac2\gamma:\,r,s\in\N^*\right\rbrace.
\end{equation}

\begin{theorem}\label{T:Neumann}
Let $\Q$ be a Borel probability measure on $\C^{\N^*}$ (endowed with the cylinder topology). Let $\alpha\in\C\setminus(kac^+\cup kac^-)$, and suppose that the adjoint relations of Lemma~\ref{L:adjoint_D_alpha} hold on $L^2(\Q)$ for this $\alpha$.
Then, for $\Q$-a.e. $(\phi_m)_{m\geq1}$, the series $z\mapsto2\Re(\sum_{m=1}^\infty\phi_mz^m)$ converges in $\dot H^{s}(\S^1)$ for all $s<0$, and $\Q=\dot\P$ as Borel probability measures on $\dot H^{s}(\S^1)$.
\end{theorem}

We will prove this result in a somewhat indirect way by first proving an analogous statement for Feigin--Fuchs modules, see Section \ref{SS:FF_modules}. The two setups are related, but the treatment of Feigin--Fuchs modules is easier.


    \subsection{Coupling SLE with the GFF}\label{ss:zipper}
The three families of results described in Sections~\ref{SS:intro_homeo}, \ref{SS:intro_gff} and \ref{SS:intro_neumann} come together to introduce a new approach to Sheffield's quantum zipper \cite{sheffield2016}. Intuitively, this result states that when we cut the GFF on $\hat{\C}$ with an independent SLE and uniformise, we obtain two independent GFF in the unit disc.
Recall from \eqref{E:dot} the notation $\varphi \cdot \Phi = \varphi \circ \Phi + Q\log |\Phi'|$.

Let $\eta\sim\nu^\#$ with welding maps $f,g$, and conditionally on $\eta$, sample $\varphi\sim\P_\eta$. We can view $\varphi$ as a random harmonic function in $\hat\C\setminus\eta$, and $\varphi\cdot f$, $\varphi\cdot g$ are random harmonic functions in $\D$ and $\D^*$ respectively. The next theorem shows that they admit boundary values in $H^{-1}(\S^1)$, and characterises their distribution.

\begin{theorem}\label{T:zipper}
    There exists a universal constant $C_\diamond>0$ such that for all $F\in\cC^0(H^{-1}(\S^1)^2)$ bounded,
    \begin{align*}
    \int_{\cJ_{0,\infty}}e^{2\bK(\eta)} \d\nu^\#(\eta)
    \int_{H^{-1}(\eta)} \d\P_\eta(\varphi) \int_{\R^2} e^{-2Qa} \d a\, \d b \int_{0}^{2\pi}\frac{\d\alpha}{2\pi} F(\varphi\cdot f+a+b,\varphi\cdot (g \circ e^{i\alpha})+a-b)\\
    \qquad=C_\diamond\int_{H^{-1}(\S^1)^2} e^{-Qc}\d c\,\d\dot\P(\phi)e^{-Qc^*}\d c^*\d\dot\P(\phi^*) F(\phi+c,\phi^*+c^*).
    \end{align*}
In the left-hand-side, $f$ and $g$ are the normalised welding maps of $\mathrm{int}(\eta)$ and $\mathrm{ext}(\eta)$.
\end{theorem}

We emphasise that adding and subtracting the zero mode $b$ completely decouples the two fields $\phi+c$ and $\phi^*+c^*$ (see the elementary Lemma \ref{L:zero_modes}). This simple trick is new and allows to remove the conditioning that the quantum lengths must match.

 Our formulation differs from Sheffield's in the sense that Theorem \ref{T:zipper} describes the laws of the \emph{boundary fields}, while Sheffield works with \emph{quantum surfaces}, where the law of the field inside the surface is prescribed. Similarly, \cite{AHS20} consider quantum discs and spheres. See Appendix \ref{Appendix:comparison} for a detailed comparison with the result of \cite{AHS20}.

Theorem \ref{T:zipper} describes the joint law of the boundary fields when we uniformise the GFF on SLE to the unit disc. Combining with important results on the quantum length recalled in the proof of Corollary \ref{C:welding} below, it is standard to go the other way around, i.e. get a conformal welding result.
As already alluded to, conformal welding of random surfaces is a subject that has been pioneered by Sheffield \cite{sheffield2016}. We refer to the book \cite{berestycki2024gaussian} for a detailed exposition of Sheffield's approach. Extensions of his result to other settings have been instrumental in a spectacular programme on the derivation of structure constants for boundary Liouville CFT, the conformal loop ensemble, and critical exponents in statistical physics models \cite{AngRemySun_FZZ,AngRemySunZhu_correls,AngCaiSunWu_1,AngCaiSunWu_2,NQSZ_backbone, AHS_SLE}.  
In parallel, there have been some attempts for a ``classical" analytic approach to the quantum zipper \cite{astala2011,Binder23_welding,KupiainenSaksman23_welding}, but a complete solution is still out of reach due to the very rough nature of the objects involved. More specifically, this analytic approach does not treat the whole range of values for $\gamma = \sqrt{\kappa}$ and does not identify the law of the welding curve. However, it can deal with different values of $\gamma$ for either sides of the curve.

To state our welding result, we need to recall that if $\phi$ is sampled according to the measure $\dot\P$, then the theory of Gaussian multiplicative chaos allows to makes sense of $\d\mathfrak{l}^\gamma_\phi(\theta) = e^{\frac{\gamma}{2}\phi(e^{i\theta})}\d\theta$ as a random Borel measure on $\S^1$. Concretely, when $\gamma <2$, it is defined as the weak limit in probability
\begin{equation}
    \label{E:GMC_S1}
\d\mathfrak{l}_{\phi}^\gamma(e^{i\theta})=\lim_{\epsilon\to0}\,\epsilon^\frac{\gamma^2}{4}e^{\frac{\gamma}{2}\bP_\D\phi(e^{-\epsilon+i\theta})}\d\theta.
\end{equation}
When $\gamma=2$, the above limit degenerates to 0 corresponding to the critical case of GMC (see \cite{Powell20_review} for a review of the critical theory). One needs an extra renormalisation in this case. For instance, one can define the critical measure as a limit of the subcritical measures by 
\[ 
\mathfrak{l}_{\phi}^{2} = \lim_{\gamma \to 2^-} \frac{1}{2(2-\gamma)} \mathfrak{l}_{\phi}^\gamma.
\]

\begin{corollary}\label{C:welding}
Let $\gamma \in (0,2]$ and $\kappa=\gamma^2$.
Let $(\phi,\phi^*)\sim\dot\P^{\otimes2}$ and $\alpha$ an independent uniform random variable in $[0,2\pi)$. Let $h$ be the unique element of $\mathrm{Homeo}(\S^1)$ sending $1$ to $e^{-i\alpha}$ and such that the following two probability measures on $\S^1$ agree
\[\frac{1}{\mathfrak{l}^\gamma_\phi(\S^1)}h^*\mathfrak{l}^\gamma_\phi=\frac{1}{\mathfrak{l}^\gamma_{\phi^*}(\S^1)}\mathfrak{l}^\gamma_{\phi^*}.\]
Then, the law of $h$ is $\tilde\nu_\kappa$, and there exists almost surely a unique (up to Möbius transformations of $\hat\C$) Jordan curve $\eta$, such that $h$ is the welding homeomorphism of $\eta$. 

%
\end{corollary}

By Theorem \ref{T:zipper}, the curve $\eta$ in Corollary \ref{C:welding} follows the law $e^{2\bK} \nu_\kappa^\#$ where $\kappa = \gamma^2$.

\begin{proof}[Proof of Corollary \ref{C:welding}, assuming Theorem \ref{T:zipper}]

    Sample $(c,c^*) \sim C_\diamond e^{-Qc} \d c \otimes e^{-Qc^*}\d c^*$ independently of $(\phi,\phi^*)$ from $\dot\P^{\otimes2}$. By Theorem \ref{T:zipper}, $(\phi + c,\phi^*+c^*)$ has the same law as $(\varphi\cdot f + a + b,$ $\varphi \cdot (g\circ e^{i\alpha}) + a -b)$ where $(\varphi,a,b,\eta)$ is as in the statement of Theorem~\ref{T:zipper}.
    
    By \cite{sheffield2016}, both $\varphi\cdot f$ and $\varphi \cdot (g\circ e^{i\alpha})$ admit a $\frac{\gamma}{2}$-GMC measure $\mathfrak{l}_{\varphi\cdot f,\S^1}^\gamma$, $\mathfrak{l}_{\varphi\cdot (g\circ e^{i\alpha}),\S^1}^\gamma$ on the boundary $\S^1$, in the same sense as \eqref{E:GMC_S1}.
    Moreover, the measures $f^* \mathfrak{l}^\gamma_{\varphi\cdot f,\S^1}$ and $(g\circ e^{i\alpha})^*\mathfrak{l}_{\varphi\cdot (g\circ e^{i\alpha}),\S^1}$ are in fact equal and correspond to a notion of ``quantum length'' of $\eta$, with respect to $\varphi$. Sheffield's result concerns the subcritical case $\kappa<4$ but was extended to the critical case $\kappa=4$ in \cite{MR4291446}.
    By \cite{MR3870446} ($\kappa<4$) and \cite{PowellSepulveda24} ($\kappa \le 4$), the quantum length of $\eta$ is nothing else but a multiple of the $\frac{\gamma}{2}$-GMC measure with respect to the Minkowski content $\mathfrak{m}_\eta$ \eqref{E:Minkowksi_SLE} of $\eta$.
Let us mention that the recent article \cite{PowellSepulveda24} gives an elementary treatment of these questions which does not rely on Sheffield's quantum zipper. In particular, we have $h^*\mathfrak{l}^\gamma_{\varphi\cdot f,\S^1} = \mathfrak{l}^\gamma_{\varphi\cdot (g\circ e^{i\alpha}),\S^1}$    
    as Borel measures on $\S^1$, with $h=(g\circ e^{i\alpha})^{-1}\circ f|_{\S^1}\in\mathrm{Homeo}(\S^1)$ the welding homeomorphism of $\eta$.
    
    Going back to $(\phi,\phi^*)$, we have $h^*\mathfrak{l}^\gamma_{\phi,\S^1}=e^{\gamma(b +\frac{c^*-c}{2})} \mathfrak{l}_{\phi^*,\S^1}$. Dividing by the total mass removes this multiplicative constant. Moreover, by independence of $\alpha$ and $(f,g)$, the image of $1$ by $h$ is still uniformly distributed on $\S^1$. Thus, $h$ has the same law as the one prescribed in the statement.
    This shows the existence part of the statement.
    Uniqueness is known by conformal removability of SLE. When $\kappa < 4$ this follows from the fact that the two components of $\hat\C\setminus \eta$ are almost surely Hölder domains \cite{RohdeSchramm05}, a condition that is known to imply conformal removability \cite{JonesSmirnov00}. The delicate case $\kappa=4$ was only treated recently in \cite{kavvadias2022conformalremovabilitysle4}.
\end{proof}

We conclude with a last consequence of Theorem \ref{T:zipper}, which is equivalent to the main result of \cite{AHS20} (see Appendix \ref{Appendix:comparison} for the comparison). The main difference with Theorem \ref{T:zipper} is that we do not add the zero-mode $b$ and keep track of the quantum length.
To state this result, we need to introduce a new measure, denoted below by $\dot\E_\ell$.
It is known that $\dot\E[\kl_\phi^\gamma(\S^1)^\gamma)^p]<\infty$ for $p\in(-\infty,\frac{4}{\gamma^2})$, and an analytic expression for these moments is known \cite[Theorem 1.1]{Remy20}. We can disintegrate the law $e^{-Qc}\d c\otimes\dot\P$ with respect to the value of $\kl_\phi^\gamma$, namely using the change of variable $c=\frac{2}{\gamma}\log\ell$, we get for all $F\in\cC^0(H^{-1}(\S^1))$
\begin{equation}\label{eq:condition_ell}
\int_\R\dot\E[F(\phi+c)]\d\dot\P(\phi)e^{-Qc}\d c=\int_{\R_+}\dot\E\left[\kl_\phi^\gamma(\S^1)^\frac{2Q}{\gamma}F\left(\phi-\frac{2}{\gamma}\log\frac{\kl_\phi^\gamma(\S^1)}\ell\right)\right]\ell^{-\frac{2Q}{\gamma}}\frac{\d\ell}{\ell}=:\int_{\R_+}\dot\E_\ell[F]\ell^{-\frac{2Q}{\gamma}}\frac{\d\ell}{\ell},
\end{equation}
where the last equality defines the expectation $\dot\E_\ell$. Note that the underlying measure is infinite since $\frac{2Q}{\gamma}=1+\frac{4}{\gamma^2}>\frac{4}{\gamma^2}$.

\begin{corollary}\label{cor:conditioned}
    With $C_\diamond>0$ being the constant from Theorem \ref{T:zipper}, for all $F\in\cC^0(H^{-1}(\S^1)^2)$,
    \[\int F(\varphi\cdot f+a,\varphi\cdot (g\circ e^{i\alpha})+a)e^{-2Qa}\d a\,\d\P_\eta(\varphi)e^{2\bK(\eta)}\d\nu^\#(\eta)\frac{\d\alpha}{2\pi}=\frac{2C_\diamond}{\gamma}\int_{\R_+}\dot\E_\ell^{\otimes2}[F(\phi,\phi^*)]\ell^{-\frac{4Q}{\gamma}}\frac{\d\ell}{\ell},\]
    where
    \[\dot\E_\ell^{\otimes2}[F(\phi,\phi^*)]:=\int(\kl_\phi^\gamma(\S^1)\kl_{\phi^*}^\gamma(\S^1))^\frac{2Q}{\gamma}F\left(\phi-\frac{2}{\gamma}\log\frac{\kl_\phi^\gamma(\S^1)}{\ell},\phi^*-\frac{2}{\gamma}\log\frac{\kl_{\phi^*}^\gamma(\S^1)}{\ell}\right)\d\dot\P(\phi)\d\dot\P(\phi^*).\]
\end{corollary}

\begin{proof}[Proof of Corollary \ref{cor:conditioned}, assuming Theorem \ref{T:zipper}]
    For each $\eps>0$, insert the function $\eps^{-1}\ind_{\{|\log(\kl_\phi^\gamma/\kl_{\phi_*}^\gamma)|<\eps\}}$ in both sides of Theorem \ref{T:zipper}. Doing the same change of variable leading to \eqref{eq:condition_ell}, we see that the right-hand-side of Theorem \ref{T:zipper} converges to the right-hand-side of Corollary \ref{cor:conditioned} as $\eps\to0$. As for the left-hand-side, observe that $|\log(\kl_\phi^\gamma/\kl_{\phi^*}^\gamma)|=\gamma b$, so the left-hand-side of Theorem \ref{T:zipper} localises to $\{b=0\}$ as $\eps\to0$. 
\end{proof}

    \subsection{Notations}
    
Here is a list of notations commonly used in this paper.

\noindent
$\nu^\# = \nu_\kappa^\#$: SLE$_\kappa$ shape measure \eqref{E:shape_measure}. Samples usually denoted by $\eta$.

\noindent
$\mathfrak{m}_\eta$: $(1+\frac{\kappa}{8})$-Minkowski content on $\eta$ (natural parametrisation) \eqref{E:Minkowksi_SLE}. Sometimes abbreviated $\mathfrak{m}$.

\noindent
$\tilde{\nu}_\kappa$: Law of the welding homeomorphism of the SLE$_\kappa$ loop.

\noindent
$\P_\eta$: Law of the trace of the GFF on $\eta$. Samples usually denoted by $\varphi$. See Section \ref{SS:jump}.

\noindent
$\bD_\eta$: Jump operator \eqref{E:Deta} across $\eta$, acting on $L^2(\eta,\mathfrak{m}_\eta)$. Associated quadratic form is $(\cE_\eta,\cF_\eta)$~\eqref{E:def_Dirichlet_energy}.

\noindent
$\dot\P$: Boundary Neumann GFF. Samples usually denoted by $\phi$. See Section \ref{SS:intro_neumann}.

\noindent
$\bS_1$: Universal Liouville action \eqref{E:def_K_S1}.

\noindent
$\bK$: Logarithmic capacity, a.k.a. Velling--Kirillov potential, a.k.a. electrical thickness \eqref{E:def_K_S1}.

\noindent
$\bT_D \varphi$: Stress--energy tensor in $D$ induced by a function $\varphi$ on $\del D$ \eqref{E:def_stress}.

\noindent
$\bJ_D \varphi$: Heisenberg tensor induced by a function $\varphi$ on $\del D$ \eqref{E:def_stress}.

\noindent
$(\cD_n)$: Witt representation on $L^2(e^{-Qc}\d c\otimes\dot\P)$ \eqref{E:def_cDn}.

\noindent
$(\bL_{n,\alpha}^\mathrm{FF})$: Feigin--Fuchs representation on $L^2(\P_{\S^1})$ \eqref{eq:def_ff}.

\noindent
$\sL_\mu,\sR_\mu$: Lie derivatives induced by left/right composition along a Beltrami differential $\mu$ \eqref{E:def_sL_sR}.

\noindent
$\D,\D^*$: The unit disc and its complement $\hat\C\setminus\bar\D$.

\noindent
$\mathrm{Diff}^\omega(\S^1)$: Analytic diffeomorphisms of $\S^1$.

\noindent
$\iota:$ The inversion map $\hat\C\to\hat\C,\,z\mapsto\frac{1}{\bar z}$.

\noindent
$|\d z|^2=\frac{i}{2}\d z\wedge\d\bar z$: the Euclidean volume form.

\noindent
$\rv_n = -z^{n+1}\del_z, n \in \Z$: basis of Laurent polynomial vector fields $\C(z)\del_z$.

\smallskip

If $f:\C \to \C$ is a differentiable function, we will denote by $\partial_z f$ and $\partial_{\bar z} f$ the partial derivatives of $f$ with respect to $z$ and $\bar z$, while $\del f$, $\bar\del f$ will denote the 1-forms $\partial_z f(z) \d z$ and $\partial_{\bar z} f(z) \d \bar z$, respectively.

A Beltrami differential is a $(-1,1)$-tensor (locally of the form $\mu(z)\frac{\d\bar z}{\d z}$). A vector field is a $(-1,0)$-tensor (locally of the form $v(z)\del_z$). A quadratic differential is a $(2,0)$-tensor (locally of the form $q(z)\d z^2$). The product of a Beltrami differential (resp. vector field) with a quadratic differential is a $(1,1)$-form (resp. $(1,0)$-form) which can be integrated over a domain (resp. on a contour). The natural pairings are
\begin{equation}
    \label{E:pairings}
    (q,\mu) := \frac{1}{\pi} \int q(z) \mu(z) |\d z|^2;
    \qquad
    (q,v) := \frac{1}{2\pi i} \oint q(z) v(z) \d z,
\end{equation}
where the integral is over a common domain (resp. contour) of definition of $q$ and $\mu$ (resp. $v$). Such a domain will always be clear from the context or specified explicitly.

\subsection*{Acknowledgements} We are grateful to Xin Sun and Baojun Wu for helpful discussions on quantum surfaces and their relations to Liouville CFT and to Nathanaël Berestycki for comments on a preliminary version of this article. We thank Shuo Fan and Jinwoo Sung for communicating their results to us. G.B. is supported by ANR-21-CE40-0003 ``Confica".

\section{Setup and preliminary results}

  \subsection{Trace of the GFF on SLE}\label{SS:jump}
The goal of this subsection is to make sense of the trace of the Gaussian free field (and the Liouville field) on an SLE loop. 

To achieve our goal, we need to define the jump operator $\bD_\eta$ across an SLE loop $\eta$, as a densely--defined, non--negative self--adjoint operator $\bD_\eta$ on $L^2(\eta,\mathfrak{m}_\eta)$, where $\mathfrak{m}_\eta$ is the Minkowksi content of $\eta$ \eqref{E:Minkowksi_SLE}. The first step is to construct the quadratic form of $\bD_\eta$, using some (minimal) input from the theory of Dirichlet forms. Especially, we will apply the general setup of \cite[Section 6]{Fukushima10} in the special case of planar Brownian motion (see also \cite[Section 1]{GRV14_dirichlet} for a concise summary). We note that $\mathfrak{m}_\eta$ is a Revuz measure in the sense of \cite{Fukushima10}, as it is a positive Radon measure charging no set of zero logarithmic capacity (i.e. polar sets for planar Brownian motion). 

The domain of our quadratic form is defined to be
\[\cF_\eta:=\left\lbrace u\in L^2(\eta,\mathfrak{m}_\eta)|\,\exists F\in H^1(\C)\text{ s.t. }F(z)=u(z)\text{ for }\mathfrak{m}_\eta\text{-a.e. }z\right\rbrace.\]
This definition makes sense since any function in $H^1(\C)$ has a modification defined away from a set of zero logarithmic capacity (e.g. the quasicontinuous modification of \cite[Theorem~2.1.3]{Fukushima10}), and the Minkowski content $\mathfrak{m}_\eta$ does not charge such sets. By Lemma~\ref{L:local_time}, we can define a local time $I_T$ of Brownian motion on $\eta$ (measured with respect to $\mathfrak{m}_\eta$), which is the positive additive continuous functional (PCAF) associated to $\mathfrak{m}_\eta$ by the Revuz correspondence. Following the terminology of \cite{Fukushima10} (see also \cite{GRV14_dirichlet} for a review), Lemma \ref{L:local_time} shows that the local time is a PCAF in the strict sense, and that the support of this PCAF is $\eta$ (see \cite[(5.1.21)]{Fukushima10} for the definition). 

Every $u\in\cF_\eta$ has a modification $\tilde{u}$ such that $\tilde{u}$ is defined away from a polar set, and $\tilde{u}=u$ on a set of full $\mathfrak{m}_\eta$-measure. For such a modification, we can then define (as in \cite[(6.2.2)]{Fukushima10}) $\bP_\eta u(z):=\E_z[\tilde{u}(B_{\tau})]$ for all $z\in\C$, where $(B_t)_{t\geq0}$ is a Brownian motion started from $z$ under $\P_z$, and $\tau$ is the hitting time $\tilde{\eta}=\eta$. By virtue of \cite[Lemma 6.2.1]{Fukushima10}, $\bP_\eta u$ is independent of the choice of modification $\tilde{u}$. Thus, the harmonic extension of functions in $\cF_\eta$ is defined unambiguously. We can now define the Dirichlet form by
\begin{equation}
    \label{E:def_Dirichlet_energy}
    \cE_\eta(u,v):=\frac{1}{2\pi}\int_{\C\setminus\eta}\nabla\bP_\eta u\cdot\nabla\bP_\eta v|\d z|^2,\qquad\forall u,v\in\cF_\eta.
\end{equation}
By \cite[Theorem 6.2.1, Item (iii)]{Fukushima10}, the pair $(\cE_\eta,\cF_\eta)$ is a regular Dirichlet form on $L^2(\eta,\mathfrak{m}_\eta)$. By definition, this means that $(\cE_\eta,\cF_\eta)$ is a closed quadratic form, and $\cC^0(\eta)\cap\cF_\eta$ is dense in both $\cF_\eta$ (with the norm $\norm{\cdot}_{L^2(\eta,\mathfrak{m}_\eta)}^2+\cE_\eta(\cdot,\cdot)$) and $\cC^0(\eta)$ (with the uniform norm).

By the procedure of Friedrichs extension \cite[Theorem VIII.15]{ReedSimon1}, the quadratic form $(\cE_\eta,\cF_\eta)$ induces a (unique) densely defined, self--adjoint operator $(\bD_\eta,\mathrm{Dom}(\bD_\eta))$ on $L^2(\eta,\mathfrak{m}_\eta)$. The domain is 
\begin{equation}
    \label{E:Deta}
\mathrm{Dom}(\bD_\eta)=\left\lbrace u\in\cF_\eta|\,\exists C>0,\,\forall v\in\cF_\eta,\,|\cE_\eta(u,v)|\leq C\norm{v}_{L^2(\eta,\mathfrak{m}_\eta)}\right\rbrace,
\end{equation}
and for each $u\in\mathrm{Dom}(\bD_\eta)$, $\bD_\eta u$ is the unique element of $L^2(\eta,\mathfrak{m}_\eta)$ such that $\cE_\eta(u,v)=\langle\bD_\eta u,v\rangle_{L^2(\eta,\mathfrak{m}_\eta)}$.
\begin{definition}
The operator $(\bD_\eta,\mathrm{Dom}(\bD_\eta))$ is the \emph{jump operator across $\eta$} (with respect to the natural parametrisation). 
\end{definition}

Of course, we have $\ker(\bD_\eta)=\R$ (the constants) and $\norm{\ind}_{L^2(\eta,\mathfrak{m}_\eta)}=\sqrt{\mathfrak{m}_\eta(\eta)}$. Thus, $\mathrm{Id}+\bD_\eta$ is a positive, self--adjoint operator on $L^2(\eta,\mathfrak{m}_\eta)$. For $s\in\R$, we may then define the Hilbert space
\[H^s(\eta,\mathfrak{m}_\eta):=(\mathrm{Id}+\bD_\eta)^{-s}L^2(\eta,\mathfrak{m}_\eta).\]
Note that $\cF_\eta=H^{1/2}(\eta,\mathfrak{m}_\eta)$. We recall that for $s>0$, we have $(\mathrm{Id}+\bD_\eta)^{-s}=\frac{1}{\Gamma(s)}\int_0^\infty e^{-t\bD_\eta}e^{-t}t^s\frac{\d t}{t}$, and the semigroup $e^{-t\bD_\eta}$ is nothing but Brownian motion parametrised by its local time on $\eta$ (measured with respect to the Minkowski content), which can be understood as the natural Cauchy process on $\eta$.

Now, we want to identify $H^s(\eta,\mathfrak{m}_\eta)$ with a space of harmonic functions. To do so, we simply observe that the map 
\[(\mathrm{Id}-\Delta)^\frac{1-s}{2}\circ\bP_\eta\circ(\mathrm{Id}+\bD_\eta)^{s-\frac{1}{2}}:\,H^s(\eta,\mathfrak{m}_\eta)\to\{F\in H^{s+\frac{1}{2}}(\C)|\,F\text{ harmonic in }\C\setminus\eta\}\]
is a topological isomorphism. In particular, we get an isometric embedding of $H^{s}(\eta,\mathfrak{m}_\eta)$ into $H^{s+\frac{1}{2}}(\C)$.

The space $\cF_\eta = H^{1/2}(\eta,\mathfrak{m}_\eta)$ allows us to define a Gaussian process $\varphi$, with law $\P_\eta$, corresponding to the trace of the Gaussian free field on $\eta$.
By definition, this means that $\varphi$ is the centred Gaussian process $(\varphi_u)_{u\in\cF_\eta}$ such that $\varphi_u\sim\cN(0,\cE_\eta(u,u))$ for all $u\in\cF_\eta$. Now, let $X$ be the Gaussian free field in $\hat\C$ equipped with the metric $(1\vee|z|)^{-4}|\d z|^2$; we have $\langle X,\bP_\eta u\rangle_{H^1(\C)}\sim\varphi_u$ for all $u\in\cF_\eta$. Hence, we can identify $\varphi$ with a random function in $H^{s}(\C)$ ($s<0$) harmonic in $\C\setminus \eta$ and $\varphi\in H^{s}(\eta)$ almost surely for all $s<-\frac{1}{2}$. Since $\ker\bD_\eta=\R$, samples of $\P_\eta$ are only defined up to constant \textit{a priori}. It will be convenient for us to fix the normalisation such that the harmonic extension at $\infty$ vanishes, i.e. $\P_\eta$ is now a probability measure on $\{\varphi\in H^s(\eta)|\,\bP_\eta\varphi(\infty)=0\}$. Finally, we ``sample" the harmonic extension at $\infty$ independently from $e^{-2Qa}\d a$ on $\R$, describing a field $a+\varphi\in H^s(\eta)\simeq\R\times\{\varphi\in H^s(\eta)|\,\bP_\eta\varphi(\infty)=0\}$. The choice of the measure $e^{-2Qa}\d a$ for the zero mode is due to the expression for the Liouville action \eqref{eq:liouville_action}.

To be consistent with the literature (see Appendix \ref{Appendix:comparison}), we call the measure $e^{-2Qa}\d a\,\d\P_\eta(\varphi)$ on $H^s(\eta)$ the \emph{trace of the Liouville field} on $\eta$, although we will not use this terminology much.

\subsection{Background on quasiconformal maps}\label{SS:qc_maps}

We recall some basic facts about quasiconformal maps following \cite{Ahlfors66}.

A homeomorphism $\Phi : \hat\C \to \hat\C$ is said to be quasiconformal, with given dilation $\mu$, if it is differentiable almost everywhere and solves the \textit{Beltrami equation}: for almost every $z \in \hat \C$,
\[
\partial_{\bar z} \Phi(z) = \mu(z) \partial_z \Phi(z),
\]
where $\mu$ is a complex-valued measurable function with $\|\mu\|_\infty < 1$. Solutions to this equation are known to exist \cite[Chapter~V, Theorem~3]{Ahlfors66} and are in fact quite regular. Indeed, their partial derivatives $\partial_z \Phi$, $\partial_{\bar z} \Phi$ are locally $L^p$ for some $p>2$ depending on $\|\mu\|_\infty$ \cite[Chapter~V, Theorem~1]{Ahlfors66} and $\Phi$ is locally Hölder continuous \cite[Chapter V, Equation (10)]{Ahlfors66}.
By Weyl's Lemma (\cite[Chapter~II, Corollary~2]{Ahlfors66}), if $\mu$ vanishes in some open set $U$, then $\Phi$ is actually conformal in $U$.
Moreover, if $\Phi$ is a solution to the above equation, then $f \circ \Phi$ is also a solution for any Möbius map of $\hat\C$. The solution is unique if one fixes these 6 degrees of freedom. Here are two natural ways to do so:
\begin{itemize}[leftmargin=*]
    \item One can require that $\Phi(0)=0$, $\Phi(1)=1$ and $\Phi(\infty)=\infty$ corresponding to what is sometimes called the ``normalised'' solution.
    \item If $\mu$ vanishes in a neighbourhood of $\infty$ (resp. $0$), one can require that $\Phi(0)=0$, $\Phi(\infty) = \infty$ and $\Phi'(\infty) = 1$ (resp. $\Phi'(0)=1$), corresponding to what is sometimes called the ``normal'' solution.
\end{itemize}

The inversion map $\iota : z \mapsto 1/\bar z$ acts on Beltrami differentials $\mu \in L^\infty(\hat \C)$ by
\begin{equation}
    \label{E:iota}
    \iota^*\mu(z) = (z/\bar z)^2 \overline{\mu(1/\bar z)}, \quad z \in \C.
\end{equation}
This definition is motivated by the fact that, if $\tilde\Phi$ is a quasiconformal homeomorphism with Beltrami $\mu + \iota^*\mu$ fixing $0$ and $\infty$, then $\tilde\Phi$ preserves the unit circle $\S^1$. This simply follows from a chain rule computation which shows that $\iota\circ\tilde\Phi\circ\iota$ is solution to the same Beltrami equation. Since it also fixes $0$ and $\infty$, by uniqueness of solutions, there exists a rotation parameter $\theta \in \R$ such that $\iota\circ\tilde\Phi\circ\iota = e^{i\theta} \tilde\Phi$.

\begin{lemma}\label{L:qc}
Let $\mu$ be a measurable map with $\|\mu\|_\infty < \infty$ and let
\begin{gather}\label{E:vmu}
    w_\mu(z) = -\frac1{i\pi} \int_{\C} \frac{(z-1) \mu(\zeta)}{\zeta(\zeta-1)(\zeta-z)}|\d \zeta|^2,
    \quad z \in \C.
\end{gather}
For $t \in \C$ small, let $\Phi_t$ and $\tilde\Phi_t$ be the unique quasiconformal homeomorphisms fixing $0,1$ and $\infty$ with Beltrami $t\mu$ and $t \mu + \bar t \iota^*\mu$, respectively. Then
\begin{equation}\label{E:Ahlfors1}
    \Phi_t(z) = z + itzw_\mu(z) + o(t), \quad z \in \C, \qquad \text{and} \qquad
    \tilde \Phi_t(z) = z + 2iz \Re (tw_\mu(z))+ o(t), \quad z \in \S^1,
\end{equation}
where $o(t) \to 0$ as $t \to 0$ uniformly on compact sets of $\C$.
In addition, if $\mu \in L^p(\C)$ for some $p>2$, then
\begin{equation}\label{E:wmu_del}
    \partial_{\bar z} (iz w_\mu) = \mu.
\end{equation}
\end{lemma}

\begin{proof}
    The estimate \eqref{E:Ahlfors1} concerning $\Phi_t$ is the content of \cite[Chapter V, Theorem 5]{Ahlfors66}.
    The estimate \eqref{E:Ahlfors1} concerning $\tilde\Phi_t$ follows from the same result and the fact that for all $z \in \S^1$,
    \[ 
    t w_\mu(z) + \bar t w_{\iota^*\mu}(z) = 2 \Re(t w_\mu(z)).
    \]
    This in turn follows from the fact that $\tilde\Phi_t(\S^1)\subset \S^1$ for all $t \in \C$ (or is simply a direct computation).
    Now, assume that $\mu \in L^p(\C)$ for some $p>2$. Expanding the integrand defining $w_\mu$, we have
    \[ 
    iz w_\mu(z) = -\frac1\pi \int_\C \Big(\frac{1}{\zeta - z} - \frac1\zeta \Big) \mu(\zeta)|\d \zeta|^2 - \frac{z}{\pi} \int_\C \Big( \frac{1}{\zeta} - \frac{1}{\zeta-1} \Big) \mu(\zeta) |\d \zeta|^2.
    \]
    Both of these integrals converge by Hölder inequality and the assumption that $\mu \in L^p(\C)$. The partial derivative in $\bar z$ only involves the first integral and we conclude from \cite[Chapter~V, Lemma~3]{Ahlfors66} that $\partial_{\bar z} (izw_\mu) = \mu$.
\end{proof}

We will sometimes have to consider different normalisations than the one in Lemma \ref{L:qc} above. Consider for instance the case where $\mu \in L^\infty(\hat\C)$ is supported in $\D^*$ and $\Phi_t$ is solution to the Beltrami equation with coefficient $t\mu$ and such that $\Phi_t(0)=0$, $\Phi_t'(0)=1$ and $\Phi_t(\infty) = \infty.$ Then, by composing with a well chosen Möbius map, we can recover a solution fixing $0$, $1$ and $\infty$ to show that
\begin{equation}\label{E:qc_vendredi}
    \Phi_t(z) = z + t (v_\mu(z) - v_\mu(0) - zv_\mu'(0)) + o(t), \qquad \text{where} \qquad v_\mu(z) = izw_\mu(z).
\end{equation}

We now state a lemma that will allow us to transfer the expression of some derivatives at $t=0$ to derivates at any $t=t_0$.

\begin{lemma}\label{L:Beltrami_mus}
    Let $\mu$ be a measurable map with $\|\mu\|_\infty < \infty$. For $t \in \C$ with $t < 1/\|\mu\|_\infty$, let $\Phi_t$ be a solution to the Beltrami equation with coefficient $t \mu$. Let $t_0 \in \C$ with $t < 1/\|\mu\|_\infty$. Then, for all $t \in \C$ small, $\Phi_{t+t_0} \circ \Phi_{t_0}^{-1}$ solves the Beltrami equation with coefficient $t \mu_{t_0} + o(t)$ where
    \begin{equation}\label{E:Beltrami_mus}
        \mu_{t_0} = \overline{\del_z \Phi_{t_0}^{-1}} (\mu \del_z \Phi_{t_0})\circ \Phi_{t_0}^{-1}.
    \end{equation}
\end{lemma}

\begin{proof}
    This is a simple computation using the chain rule.
\end{proof}

Expanding on \eqref{E:iota}, we will furthermore consider for all conformal map $f$ and Beltrami differential $\mu(z) \frac{\d \bar z}{\d z}$, the pullback $f^*\mu$ and pushforward $f_*\mu$ defined by
\begin{equation}\label{E:f*mu}
    f^*\mu = \mu \circ f(z) \frac{\overline{f'(z)} \d \bar z}{f'(z) \d z}
    \qquad \text{and} \qquad f_*\mu = \mu \circ f^{-1}(z) \frac{\overline{(f^{-1})'(z)} \d \bar z}{(f^{-1})'(z) \d z}.
\end{equation}
We record for ease of reference the following elementary fact. If $f$ is a map which is conformal in the support of $\mu$, then
\begin{equation}
    \label{E:schwarzian_mu}
    (\cS f^{-1}, f_*\mu) = -(\cS f, \mu).
\end{equation}
Indeed, by the chain rule for the Schwarzian derivative, $\cS f^{-1} = - ((f^{-1})')^2 (\cS f)\circ f^{-1}$. So by definition \eqref{E:f*mu} of  $f_*\mu$, and then by a change of variable, we have
\[ 
(\cS f^{-1},f_*\mu) = - \frac1\pi \int (\cS f)\circ f^{-1}(z) \mu \circ f^{-1}(z) |(f^{-1})'(z)|^2 |\d z^2| = -(\cS f,\mu).
\]

\section{Variational formula for the GFF on SLE: proof of Theorem \ref{T:GFF_SLE}}\label{S:liouville}

In this section, we prove our variational formula stated in Theorem \ref{T:GFF_SLE}. To do so, we first introduce the Liouville action and study its infinitesimal variation under the action of the quasiconformal group. The main result is that the differential of the Liouville action is the stress--energy tensor. We introduce this tensor in Section \ref{SS:stress}, compute the variation of the Liouville action in $\D$ in Section \ref{SS:liouville_disc}, and deduce the variation of the Dirichlet energy $\cE_\eta$ in Section~\ref{SS:dirichlet}. Theorem \ref{T:GFF_SLE} is then derived in Section \ref{S:gff_sle}.

   \subsection{The stress-energy tensor}\label{SS:stress}
Let $D\subset\C$ be a simply connected domain different from $\C$ and $\varphi : \partial D \to \R$ be a distribution whose harmonic extension $\bP_D \varphi$ in $D$ is well defined (see Section~\ref{SS:jump}).
The stress-energy tensor and the Heisenberg tensor are respectively defined to be the functions
\begin{equation}\label{E:def_stress}
\bT_D \varphi(z):=-(\del_z
\bP_D \varphi(z))^2+Q\del^2_{zz}\bP_D\varphi(z)
\quad \text{and} \quad
\bJ_D \varphi(z):=\frac{1}{z}\del_z\bP_D\varphi(z), \quad z \in \hat\C.
\end{equation}
Because a harmonic function is the real part of a holomorphic function, $\bT_D\varphi$ is holomorphic and $\bJ_D\varphi$ is meromorphic. If $D$ is the interior or exterior of a curve $\eta$, we will often write $\bT_\eta \varphi$ and $\bJ_\eta \varphi$ instead of $\bT_D \varphi$ and $\bJ_D \varphi$.

The following elementary lemma gives a chain rule for the stress--energy tensor, showing that in transforms as a \emph{$(6Q^2)$-projective connection} in the terminology of \cite{FriedanShenker87}.

\begin{lemma}\label{L:chain}
Let $f:D\to\tilde{D}$ be a conformal equivalence, and $\varphi : \partial \tilde D \to \R$ be a distribution whose harmonic extension $\bP_{\tilde D} \varphi$ is well defined.
Recalling that $\varphi \cdot f = \varphi\circ f+Q\log|f'|$,
we have
\[
\bT_D (\varphi\cdot f)=(f')^2\bT_{\tilde D} \varphi\circ f+\frac{Q^2}{2}\cS f.\]
In particular, for all $\mu \in L^\infty(\tilde D)$, recalling that $f^*\mu = \overline{f'}/f' \mu\circ f$, we have
\begin{equation}
    \label{E:chain_rule_stress_mu}
(\bT_D (\varphi\cdot f), f^*\mu)=(\bT_{\tilde D} \varphi,\mu)+\frac{Q^2}{2}(\cS f,f^*\mu).
\end{equation}
\end{lemma}
\begin{proof}
We record the elementary formulas (recall $\cA f=\frac{f''}{f'}$ is the pre-Schwarzian derivative):
\begin{align*}
&\del_z(\varphi\circ f+Q\log|f'|)=f'(\del_z\varphi)\circ f+\frac{Q}{2}\cA f;\\
&(\del_z(\varphi\circ f+Q\log|f'|))^2=(f')^2(\del_z\varphi)^2\circ f+Qf''(\del_z\varphi)\circ f+\frac{Q^2}{4}(\cA f)^2;\\
&\del_{zz}^2(\varphi\circ f+Q\log|f'|)=(f')^2(\del_{zz}^2\varphi)\circ f+f''(\del_z\varphi)\circ f+\frac{Q}{2}(\cA f)'.
\end{align*}
From this, we deduce that $\bT_D (\varphi\cdot f)=-(f')^2(\del_z\varphi)^2\circ f-\frac{Q^2}{4}(\cA f)^2+\frac{Q^2}{2}(\cA f)'=(f')^2\bT_{\tilde D} \varphi\circ f+\frac{Q^2}{2}\cS f$. The identity \eqref{E:chain_rule_stress_mu} then follows by a change of variable.
\end{proof}

We also note the following important reflection formula.
\begin{lemma}\label{L:stress_iota}
Let $\varphi\in\cC^\infty(\S^1)$. For all $z\in\D^*$, we have
\begin{align*}
    \frac1{z^4} \overline{\bT_{\D} \varphi}(1/z) = \bT_{\D^*} \varphi(z) + 2Q \bJ_{\D^*} \varphi(z).
\end{align*}
In particular, for all $\mu \in L^\infty(\C)$,
$\overline{(\bT_{\D}\varphi,\iota^*\mu)} = (\bT_{\D^*}\varphi + 2Q \bJ_{\D^*}\varphi,\mu).$
\end{lemma}

\begin{proof}
This follows from the chain rule and the fact that $\bP_{\D^*} \varphi(z) = (\bP_\D \varphi) \circ \iota(z)$ for all $z\in\D^*$.
\end{proof}
    
    \subsection{Variation of the Liouville action in \texorpdfstring{$\D$}{D}}\label{SS:liouville_disc}

We define the \emph{Liouville actions} of a field $\phi\in H^{1/2}(\S^1)$ to be
\begin{align}\label{E:def_LiouvilleS1}
&\bS_\D(\phi):=\frac{i}{\pi}\int_\D\del\bP_\D\phi\wedge\delbar\bP_\D\phi+2Q\bP_\D\phi(0)=\frac{1}{i\pi}\oint_{\S^1}\phi\del\bP_\D\phi+\frac{Q}{i\pi}\oint_{\S^1}\phi(z)\frac{\d z}{z};\\
\label{E:def_LiouvilleS2}
&\bS_{\D^*}(\phi):=\frac{i}{\pi}\int_{\D^*}\del\bP_{\D^*}\phi\wedge\delbar\bP_{\D^*}\phi-2Q\bP_{\D^*}\phi(\infty)=-\frac{1}{i\pi}\oint_{\S^1}\phi\del\bP_{\D^*}\phi-\frac{Q}{i\pi}\oint_{\S^1}\phi(z)\frac{\d z}{z}.
\end{align}
Since $\bP_{\D^*}\phi(z)=\bP_\D\phi(1/\bar z)$, we have $\bS_\D(\phi)=\bS_{\D^*}(\phi)$. 

In Proposition \ref{P:var_liouville} below, we study the infinitesimal variations of $t \mapsto \bS_\D(\phi\cdot h_t^{-1})$ and $t\mapsto\bS_{\D^*}(\phi\cdot h_t^{-1})$ where $(h_t)_t$ is a family of analytic diffeomorphisms of $\S^1$. These diffeomorphisms will be defined as the restriction to $\S^1$ of symmetric quasiconformal maps (see Section \ref{SS:qc_maps}). We recall from \eqref{E:doth} that
$\phi \cdot h = \phi \circ h + Q\log (zh'/h).$

\begin{proposition}\label{P:var_liouville}
Let $\mu$ be a Beltrami differential compactly supported in $\D$ (resp. $\D^*$). For small $t \in \C$, let $h_t$ be a solution to the Beltrami equation with coefficient $t\mu+\bar t\iota^*\mu$, normalised to fix 0 and $\infty$. 
For all $\phi\in\cC^\infty(\S^1)$, the map $t\mapsto\bS_\D(\phi\cdot h_t^{-1})$ (resp. $t\mapsto\bS_{\D^*}(\phi\cdot h_t^{-1})$) is continuously differentiable in a complex neighbourhood of $t=0$ and, as $t\to0$, we have
\begin{align}\label{E:P_var_Liouville_disc1}
\bS_\D(\phi\cdot h_t^{-1})-\bS_\D(\phi)&=4\Re\big(t(\bT_\D\phi,\mu)\big)+o(t),\\
\label{E:P_var_Liouville_disc2}
\text{resp.}\qquad \bS_{\D^*}(\phi\cdot h_t^{-1})-\bS_{\D^*}(\phi)&=4\Re\big(t(\bT_{\D^*}\varphi + 2Q \bJ_{\D^*}\varphi,\mu)\big)+o(t).
\end{align}
\end{proposition}

\begin{proof}[Proof of Proposition \ref{P:var_liouville}]
    By invariance under rotations centred at the origin, we can assume without loss of generality that $h_t$ fixes $0$, $1$ and $\infty$.
    We start by dealing with the case where $\mu$ is compactly supported in $\D$ and will deduce the other case from the first one. 
    By definition \eqref{E:doth} of $\phi\cdot h_t^{-1}$ and definition \eqref{E:def_LiouvilleS1} of the Liouville action, we have
    \begin{align}\label{E:pf_P_varS1}
        \bS_\D(\phi\cdot h^{-1}_t)-\bS_\D(\phi)
&=\frac{1}{i\pi}\oint\phi\circ h^{-1}_t\del\bP_\D(\phi\circ h_t^{-1})-\frac{1}{i\pi}\oint\phi\del\bP_\D\phi\\
\label{E:pf_P_varS2}
&\quad+2Q\Re\left(\frac{1}{i\pi}\oint\log\left(\frac{z(h_t^{-1})'}{h_t^{-1}}\right)\del\bP_\D(\phi\circ h_t^{-1})\right)\\
\label{E:pf_P_varS3}
&\quad+2Q\bP_\D(\phi\cdot h_t^{-1})(0)-2Q\bP_\D\phi(0) + o(t),
\end{align}
where the $o(t)$ is the Dirichlet energy of $Q\log (z(h_t^{-1})'/h_t^{-1})$.
We will then break down the proof into several lemmas, each dealing with one of the right hand side terms in the above display.
We will heavily use the expansion $h_t(z) = z + 2iz \Re(tw_\mu(z)) + o(t)$ \eqref{E:Ahlfors1} where $o(t) \to 0$ uniformly on $\S^1$. The function $w_\mu$ is defined in \eqref{E:vmu} and satisfies $\del_{\bar z}(izw_\mu(z)) = \mu(z)$ \eqref{E:wmu_del}. In particular, since we are considering the case where the support of $\mu$ is included in $\D$, $izw_\mu$ is holomorphic in a neighbourhood of $\overline{\D^*}$.

We start with the variation of $\bP_\D\phi(0)$, i.e. considers the term in \eqref{E:pf_P_varS3}.
\begin{lemma}\label{L:var_0}
We have
\[\bP_\D(\phi\cdot h_t^{-1})(0)-\bP_\D\phi(0)=-\frac{1}{\pi}\Re\left(t\oint_{\S^1}w_\mu\del\bP_\D\phi\right)+o(t).\]
\end{lemma}
\begin{proof}
By definition, $\phi\cdot h_t^{-1} = \phi\circ h_t^{-1} + Q \log (z(h_t^{-1})'/h_t)$, so we need to control both terms.
By the circle average principle and a change of variable, we have
\begin{align*}
& \bP_\D(\phi\circ h_t^{-1})(0)
=\oint_{\S^1}\phi(h_t^{-1}(z))\frac{\d z}{2i\pi z}
=\oint_{\S^1}\phi(z)\frac{zh_t'(z)}{h_t(z)}\frac{\d z}{2i\pi z}\\
&=\bP_\D\phi(0)+\frac{1}{\pi}\Re\left(t\oint_{\S^1}\phi(z)w_\mu'(z)\d z\right)+o(t)
=\bP_\D\phi(0)-\frac{1}{\pi}\Re\left(t\oint_{\S^1}\del_z\bP_\D\phi(z)w_\mu(z)\d z\right)+o(t).
\end{align*}
Moreover, $\oint_{\S^1}\log\left(\frac{zh_t'(z)}{h_t(z)}\right)\frac{\d z}{2i\pi z}=\frac{1}{\pi}\Re(t\oint_{\S^1}\d w_\mu)+o(t)=o(t)$, concluding the proof.
\end{proof}

We continue with the variation of the Dirichlet energy under $\phi\mapsto\phi\circ h^{-1}$, i.e. considers the right hand side term in \eqref{E:pf_P_varS1}.
\begin{lemma}\label{L:diff_energies}
Let $h\in\mathrm{Diff}^\omega(\S^1)$. We have for all $\phi\in\cC^\infty(\S^1)$,
\[\frac{1}{i\pi}\oint_{\S^1}\phi\circ h^{-1}\del\bP_\D(\phi\circ h^{-1})-\frac{1}{i\pi}\oint\phi\del\bP_\D\phi=-\frac{1}{2\pi^2}\oint\oint\log\left(\frac{h(z)-h(\zeta)}{z-\zeta}\right)\d\phi(z)\d\phi(\zeta).\]
\end{lemma}
\begin{proof}
By Cauchy's integral formula and then an integration by parts, we have for all $z\in\D$,
\begin{align*}
\del_z\bP_\D(\phi\circ h^{-1})(z)
&=\frac{1}{2i\pi}\oint_{\S^1}\frac{\del_\zeta\bP_\D(\phi\circ h^{-1})(\zeta)}{\zeta-z}\d\zeta
=\frac{1}{2i\pi}\oint_{\S^1}\frac{\phi\circ h^{-1}(\zeta)}{(\zeta-z)^2}\d\zeta\\
&=\frac{1}{2i\pi}\oint_{\S^1}\frac{\phi(\zeta)}{(h(\zeta)-z)^2}h'(\zeta)\d\zeta
=\frac{1}{2i\pi}\oint_{\S^1}\frac{\d\phi(\zeta)}{h(\zeta)-z}.
\end{align*}
Hence,
\begin{align*}
\frac{1}{i\pi}\oint_{\S^1}\phi\circ h^{-1}\del\bP_\D(\phi\circ h^{-1})
&=\frac{1}{2\pi^2}\oint_{\S^1}\oint_{\S^1}\frac{\phi(h^{-1}(z))}{z-h(\zeta)}\d\phi(\zeta)\d z\\
&=\frac{1}{2\pi^2}\oint_{\S^1}\oint_{\S^1}\frac{\phi(z)}{h(z)-h(\zeta)}h'(z)\d z\d\phi(\zeta)
\end{align*}
where the integral is understood in the principal value sense. It follows that
\begin{align*}
\frac{1}{i\pi}\oint_{\S^1}\phi\circ h^{-1}\del\bP_\D(\phi\circ h^{-1})-\frac{1}{i\pi}\oint\phi\del\bP_\D\phi
&=\frac{1}{2\pi^2}\oint\oint\phi(z)\left(\frac{h'(z)}{h(z)-h(\zeta)}-\frac{1}{z-\zeta}\right)\d z\d\phi(\zeta)\\
&=-\frac{1}{2\pi^2}\oint\oint\log\left(\frac{h(z)-h(\zeta)}{z-\zeta}\right)\d\phi(z)\d\phi(\zeta)
\end{align*}
concluding the proof.
\end{proof}
Now, we take $h=h_t$ as in the beginning of the section and differentiate at $t=0$.

\begin{corollary}\label{C:var_dirichlet}
We have
\begin{equation}\label{E:C_var_dirichlet}
\frac{1}{i\pi}\oint\phi\circ h_t^{-1}\del\bP_\D(\phi\circ h_t^{-1})
- \frac{1}{i\pi}\oint\phi\del\bP_\D\phi
= -\frac{2}{\pi}\Re\left(t\oint\del_z\bP_\D\phi(z)^2zw_\mu(z)\d z\right)+o(t).
\end{equation}
\end{corollary}
\begin{proof}
From Lemma \ref{L:diff_energies} and the expansion \eqref{E:Ahlfors1} of $h_t$, the left hand side of \eqref{E:C_var_dirichlet} equals
\begin{align*}
&-\Re\left(\frac{it}{\pi^2}\oint\oint\frac{zw_\mu(z)-\zeta w_\mu(\zeta)}{z-\zeta}\d\phi(z)\d\phi(\zeta)\right)+o(t)\\
&=-\frac2\pi\Re\left(\frac{it}{\pi}\oint\oint\frac{zw_\mu(z)}{z-\zeta}\del\bP_\D\phi(z)\del\bP_\D\phi(\zeta)\right)+o(t).
\end{align*}
Applying Cauchy's principal value theorem, we further have for all $z \in \S^1$,
\[
\frac{i}\pi \oint\frac{1}{z-\zeta}\del\bP_\D\phi(\zeta) = \partial_z\bP_\D\phi(z),
\]
which concludes the proof.
\end{proof}

Finally, we treat the cross--term \eqref{E:pf_P_varS2}.
\begin{lemma}\label{L:cross}
We have
\begin{equation}\label{E:L_cross}
2\Re\left(\frac{1}{i\pi}\oint\log\left(\frac{z(h_t^{-1})'}{h_t^{-1}}\right)\del\bP_\D(\phi\circ h_t^{-1})\right)=\frac{2}{\pi}\Re\left(t\oint zw_\mu(z)\del_{zz}^2\bP_\D\phi(z)+t\oint w_\mu\del\bP_\D\right)+o(t).
\end{equation}
\end{lemma}
\begin{proof}
Because the Beltrami $\mu$ is supported in $\D$, $zw_\mu'$ is holomorphic in $\D^*$ and, using the expansion \eqref{E:Ahlfors1}, we have $\log(\frac{z(h_t^{-1})'}{h_t^{-1}})=-2\Re(itzw_\mu'(z))+o(t)$. Hence, the left hand side of \eqref{E:L_cross} equals
\begin{align*}
&-\Re\left(\frac{4}{i\pi}\oint\Re(itzw_\mu'(z))\del\bP_\D\phi(z)\right)+o(t)
=-\frac{2}{\pi}\Re\left(t\oint zw_\mu'(z)\del\bP_\D\phi(z)\right)+o(t)\\
& \qquad =\frac{2}{\pi}\Re\left(t\oint zw_\mu(z)\del_{zz}^2\bP_\D\phi(z)\right)+\frac{2}{\pi}\Re\left(t\oint w_\mu(z)\del\bP_\D(z)\right)+o(t),
\end{align*}
where we performed an integration by parts in the last equality.
\end{proof}

Putting the previous lemmas together and going back to \eqref{E:pf_P_varS1}, we find that
\begin{align}\nonumber
\bS_\D(\phi\cdot h^{-1}_t)-\bS_\D(\phi)
&=\frac{2}{\pi}\Re\left(t\oint\left(-\del\bP_\D\phi(z)^2+Q\del_{zz}\bP_\D\phi(z)\right)zw_\mu(z)\d z\right) + o(t) \\
& =\frac{2}{\pi}\Re\left(t\oint \bT_\D \phi(z) zw_\mu(z)\d z\right) + o(t).\label{E:pf_prop33}
\end{align}
Because $\bT_\D \phi$ is conformal in $\D$, we have
\[ 
\d(\bT_\D \phi(z) zw_\mu(z)\d z) = 2i \bT_\D \phi(z) \del_{\bar z}(zw_\mu(z)) |\d z|^2.
\]
By Stokes' formula, we can thus rewrite
\[ 
\frac{2}{\pi}\Re\left(t\oint \bT_\D \phi(z) zw_\mu(z)\d z\right)
= \frac{4}{\pi} \Re \left( t \int_\D \bT_\D \phi(z) \del_{\bar z}(izw_\mu(z)) |\d z|^2 \right).
\]
By \eqref{E:wmu_del} in Lemma \ref{L:qc}, $\del_{\bar z}(izw_\mu(z)) = \del_{\bar z} v_\mu(z) = \mu(z)$. We have obtained:
\[ 
\bS_\D(\phi\cdot h^{-1}_t)-\bS_\D(\phi) =  \frac{4}{\pi} \Re \left( t \int_\D \bT_\D \phi(z) \mu(z) |\d z|^2 \right) + o(t) = 4 \Re \big( t (\bT_\D \phi, \mu) \big) + o(t).
\]
This gives the differentiability at $t=0$ stated in \eqref{E:P_var_Liouville_disc1}. Using the structure of group action $(\phi,h)\mapsto\phi\cdot h$ (see \eqref{E:cdot_group}), we immediately get that $t\mapsto\bS_\D(\phi\cdot h_t^{-1})$ is differentiable for all $t$, and the formula for the differential is obviously continuous. 

\smallskip

We now assume that $\mu$ is supported in $\D^*$ and prove \eqref{E:P_var_Liouville_disc2}. Since by definition $\bS_\D(\phi)=\bS_{\D^*}(\phi)$
\[ 
\bS_{\D^*}(\phi\cdot h^{-1}_t)-\bS_{\D^*}(\phi) = \bS_\D(\phi\cdot h^{-1}_t)-\bS_\D(\phi).
\]
Applying \eqref{E:P_var_Liouville_disc1} to the Beltrami $\iota^*\mu$ whose supported is compactly included in $\D$, and then using Lemma \ref{L:stress_iota}, we get that
\begin{align*}
    \bS_{\D^*}(\phi\cdot h^{-1}_t)-\bS_{\D^*}(\phi) = 4 \Re\big( \bar{t} (\bT_\D \phi, \iota^*\mu) \big) + o(t) 
    = 4 \Re\big( t (\bT_{\D^*}\varphi + 2Q \bJ_{\D^*}\varphi,\mu) \big) + o(t).
\end{align*}
This proves \eqref{E:P_var_Liouville_disc2}.
\end{proof}

    \subsection{Variation of Liouville action for smooth curves}\label{SS:dirichlet}
Let $\eta$ be an analytic Jordan curve. Given $\varphi\in\cC^\infty(\eta)$, and recalling the definition \eqref{E:def_Dirichlet_energy} of the Dirichlet energy $\cE_\eta(\varphi)$ of $\varphi$, we define the \emph{Liouville action}:
\begin{equation}\label{eq:liouville_action}
\bS_\eta(\varphi):=\cE_\eta(\varphi)+4Q\bP_\eta\varphi(\infty).
\end{equation}
According to \cite[Theorem 3.6]{ViklundWang19}, this Liouville action is related to the Liouville actions $\bS_\D$ and $\bS_{\D^*}$ considered in Section \ref{SS:liouville_disc} by\footnote{In \cite{ViklundWang19}, the Liouville action in $\D^*$ is defined to be $\bS_\D(\phi)-4Q\bP\phi(0)$, while in our convention $\bS_{\D^*}(\phi)=\bS_\D(\phi)$. This is because they work with the Euclidean metric in $\D^*$ (which has negative geodesic curvature on $\S^1=\del\D^*$), while we choose the metric $|\frac{\d z}{z^2}|^2$ identifying $\D^*$ with a copy of $(\D,|\d z|^2)$.}
\begin{equation}
    \label{E:VW20}
    \bS_\eta(\varphi)=\bS_\D(\varphi\cdot f)+\bS_{\D^*}(\varphi\cdot g)-\frac{Q^2}{2\pi}\bS_1(\eta),
\end{equation}
where $f: \D \to \rmint(\eta)$ and $g:\D^* \to \rmext(\eta)$ are uniformising maps fixing $0$ and $\infty$ respectively and $\bS_1(\eta)$ is the universal Liouville action \eqref{E:def_K_S1}.

\begin{proposition}\label{P:var_S}
Let $\mu$ be a Beltrami differential compactly supported in $\mathrm{int}(\eta)$ (resp. $\mathrm{ext}(\eta)$), and $\Phi_t$ be the solution to the Beltrami equation with coefficient $t\mu$, normalised such that $\Phi_t(0)=0$, $\Phi_t(\infty)=\infty$, and $\Phi_t'(\infty)=1$ (resp. $\Phi_t'(0)=1$). For all $\varphi\in\cC^\infty(\eta)$, the map $t\mapsto\bS_{\Phi_t(\eta)}(\varphi)$ is of class $\cC^1$ in a complex neighbourhood $\cU$ of $t=0$ and for all $u \in \cU$:
\begin{align*}
&\bS_{\Phi_t(\eta)}(\varphi\cdot\Phi_t^{-1})-\bS_{\Phi_u(\eta)}(\varphi\cdot\Phi_u^{-1})=4\Re\big((t-u)(\bT_{\Phi_u(\eta)} (\varphi\cdot \Phi_u^{-1}),\mu_u)\big)+o(t-u),
\end{align*}
where $o(t-u) / |t-u| \to 0$ as $t \to u$ and where $\mu_u$ is defined in \eqref{E:Beltrami_mus}.
\end{proposition}

In the above setting where $\mu$ is compactly supported in $\rmint(\eta)$,
$\bP_{\Phi_t(\eta)} (\varphi\cdot \Phi_t^{-1}) = \bP_{\Phi_t(\eta)} (\varphi\circ \Phi_t^{-1}) + Q \log |\Phi_t'(\infty)| = \bP_\eta \varphi(\infty)$. So we are effectively computing the differential of the Dirichlet energy $\cE_\eta$.

\begin{proof}
We only treat the case where $\mu$ is compactly supported in $D=\mathrm{int}(\eta)$, since the other one is identical. Let $\Phi_1,\Phi_2$ quasiconformal on $\hat{\C}$ with $\Phi_1$ (resp. $\Phi_2$) conformal on a neighbourhood of $\overline{D^*}$ (resp. $\Phi_1(\overline{D^*})$.
By the structure of group action \eqref{E:cdot_group}, we have $\bS_{\Phi_2\circ\Phi_1(\eta)}(\varphi\cdot(\Phi_2\circ\Phi_1)^{-1})=\bS_{\Phi_2(\Phi_1(\eta))}((\varphi\cdot\Phi_1^{-1})\cdot\Phi_2^{-1})$. Together with Lemma \ref{L:Beltrami_mus}, we see that it suffices to show that $t\mapsto\bS_{\Phi_t(\eta)}(\varphi\cdot\Phi_t^{-1})$ is differentiable at $t=0$.

Writing $f_t,g_t$ for the welding maps of the curve $\Phi_t(\eta)$, we have $f_t=\Phi_t\circ f\circ\tilde{\Phi}_t^{-1}$ and $g_t=\Phi_t\circ g$, where $\tilde{\Phi}_t$ is a solution to the Beltrami equation with coefficient $t f^*\mu + t\iota^*(f^*\mu)$ \eqref{E:f*mu}, fixing 0 and $\infty$.
By \eqref{E:VW20}, we have
\begin{align*}
\bS_{\Phi_t(\eta)}(\varphi\cdot\Phi_t^{-1})-\bS_\eta(\varphi)
& =\bS_\D\left((\varphi\cdot \Phi_t^{-1})\cdot f_t\right)-\bS_\D\left(\varphi\cdot f\right)
+\bS_{\D^*}\left((\varphi\cdot \Phi_t^{-1})\cdot g_t\right)-\bS_{\D^*}\left(\varphi\cdot g\right) \\
& \quad -\frac{Q^2}{2\pi}\left(\bS_1(\Phi_t(\eta))-\bS_1(\eta)\right).
\end{align*}
By \eqref{E:cdot_group}, $(\varphi\cdot \Phi_t^{-1})\cdot g_t = (\varphi\cdot \Phi_t^{-1})\cdot (\Phi_t \circ g) = \varphi \cdot g$ and so the variation of the Liouville action in $\D^*$ simply vanishes.
On the other hand, we have $(\varphi\cdot \Phi_t^{-1})\cdot f_t = (\varphi\cdot f)\cdot\tilde{\Phi}_t^{-1}$ and so by Proposition~\ref{P:var_liouville} the variation of the Liouville action in $\D$ equals
\[ 
4\Re\big(t(\bT_\D(\varphi\cdot f),f^*\mu)\big) + o(t).
\]
Using \cite[Chapter~2, Theorem~3.8]{TakhtajanTeo06} for the variation of $\bS_1$ and then the chain rule \eqref{E:chain_rule_stress_mu} for the stress--energy tensor, we deduce that
\begin{align*}
\bS_{\Phi_t(\eta)}(\varphi\cdot\Phi_t^{-1})-\bS_\eta(\varphi)
&=4\Re\left(t\left(\bT_\D(\varphi\cdot f),f^*\mu\right)\right)-2Q^2\Re\left(t\left(\cS f,f^*\mu\right)\right)+o(t)\\
&=4\Re\left(t\left(\bT_D\varphi,\mu\right)\right)+o(t).
\end{align*}
This proves that $t\mapsto\bS_{\Phi_t(\eta)}(\varphi\cdot\Phi_t^{-1})$ is differentiable at $t=0$, and hence in a neighbourhood of $t=0$. The formula for the differential is clearly continuous, so the map is indeed continuously differentiable.
\end{proof}

\begin{corollary}\label{C:sobolev}
    Consider the setting of Proposition \ref{P:var_S} except that the field $\varphi$ on $\eta$ is only assumed to belong to some Sobolev space $H^s(\eta,\mathfrak{m}_\eta)$ for some $s \in \R$. Then the difference $\bS_{\Phi_t(\eta)}(\varphi) - \bS_\eta(\varphi)$ is well defined and the same conclusions hold.
\end{corollary}

\begin{proof}
    Consider the case where $\mu$ is supported in $\rmint(\eta)$.
    Let $(\varphi_\eps)_{\eps > 0}$ be an approximation of $\varphi$ by smooth fields and let $t \in \C$ be small. Assume without loss of generality that $t \in \R$.
    Then by Proposition \ref{P:var_S}, for all $\eps>0$,
    \[
        \bS_{\Phi_t(\eta)}(\varphi_\eps\cdot\Phi_t^{-1})-\bS_\eta(\varphi_\eps)=4\int_0^t \Re\big((\bT_{\rmint(\Phi_u(\eta))} (\varphi_\eps \cdot \Phi_u^{-1}),\mu_u)\big) \d u,
    \]
    where $\mu_u$ is the Beltrami \eqref{E:Beltrami_mus}.
    The stress-energy tensor in the above integral is smooth in $\rmint(\Phi_u(\eta))$. Since $\mu_u$ is compactly supported in this domain the right hand side term clearly converges and we get that
    \[
        \bS_{\Phi_t(\eta)}(\varphi\cdot\Phi_t^{-1})-\bS_\eta(\varphi)=4\int_0^t \Re\big((\bT_{\rmint(\Phi_u(\eta))} (\varphi \cdot \Phi_u^{-1}),\mu_s)\big) \d u,
    \]
    where the LHS is defined by this identity.
\end{proof}

\subsection{Proof of Theorem \ref{T:GFF_SLE}}\label{S:gff_sle}
Theorem \ref{T:GFF_SLE} assumes a particular normalisation of the solution to the Beltrami equation. Different choices of normalisation differ only by scaling, so we need to study the variation of the measure $e^{-2Qa}\d a\,\P_\eta$ under rescaling. To this end, let $\lambda\in\C^*$. We have for all test functions $F$
\begin{align*}
\int_\R\E_{\eta,\lambda\cdot}[F(\varphi+a)]e^{-2Qa}\d a
&=\int_\R \E_{\lambda\eta}[F(\varphi(\lambda^{-1}\cdot\,)-Q\log|\lambda|+a)]e^{-2Qa}\d a\\
&=|\lambda|^{-2Q^2}\int_\R\E_\eta[F(\varphi+a)]e^{-2Qa}\d a.
\end{align*}
In the last equality, we have used the change of variable $a\mapsto a+Q\log|\lambda|$ and the fact that $\varphi(\lambda^{-1}\cdot)$ has the law $\P_{\lambda\eta}$ if $\varphi$ is sampled from $\P_\eta$ (by scale invariance of the Dirichlet energy). 

Assume for a moment that $\eta$ is an analytic Jordan curve, instead of a sample from the SLE loop measure. Let $\Phi$ be a quasiconformal transformation of $\hat\C$, assumed to be conformal in a neighbourhood of $\eta$.
Recall that $\P_{\eta,\Phi}$ is the law of $\varphi\cdot\Phi$ where $(a,\varphi)$ is sampled from $e^{-2Qa}\d a \d \P_{\Phi(\eta)}$.
By \cite[Lemma 5.3]{GKRV21_Segal}, $\P_{\eta,\Phi}$ is absolutely continuous with respect to $\P_\eta$, and the Radon--Nikodym derivative $\mathrm{RN}_{\eta,\Phi}=\frac{\d\P_{\eta,\Phi}}{\d\P_\eta}$ takes the following explicit form:
\[\mathrm{RN}_{\eta,\Phi}(\varphi)=\left(\frac{|\mathfrak{m}_{\Phi(\eta)}|}{|\mathfrak{m}_\eta|}\det_\mathrm{Fr}\left(\bD_{\Phi(\eta)}\bD_\eta^{-1}\right)\right)^{1/2}\exp\left(-\frac{1}{2}\left(\cE_{\Phi(\eta)}(\varphi\cdot\Phi^{-1})-\cE_\eta(\varphi)\right)\right), \quad
\varphi\in H^s(\eta).
\]
Here $s<0$ is fixed, $\cE_\eta$ is the Dirichlet energy \eqref{E:def_Dirichlet_energy}, $\bD_\eta$ is the jump operator \eqref{E:Deta}, $\det_\mathrm{Fr}$ denotes the Fredholm determinant
and $|\mathfrak{m}_\eta|$ denotes the total mass of the one-dimensional Hausdorff measure of $\eta$.

In the rest of this section, we consider the following setup. Let $\eta$ be a Jordan curve and $\mu$ be a Beltrami differential compactly supported in $\rmint(\eta)$.
For small $t \in \C$, let $\Phi_t$ be the unique solution to the Beltrami equation with coefficient $t\mu$, normalised so that $\Phi_t(0)=0$, $\Phi_t(\infty)=\infty$, and $\Phi_t'(\infty)=1$.

In Proposition \ref{P:var_S}, we described precisely the infinitesimal perturbation of the Dirichlet energy. The following lemma takes care of the partition function:

\begin{lemma}\label{L:D_partition_f}
Assume that $\eta$ is an analytic Jordan curve. Then,
\begin{align*}
    \left(\frac{|\mathfrak{m}_{\Phi_{t}(\eta)}|}{|\mathfrak{m}_{\eta}|}\det_\mathrm{Fr}(\bD_{\Phi_{t}(\eta)}\bD^{-1}_{\eta})\right)^{1/2}
    = 1 & - \frac{1}{6} \Re (t (\cS f^{-1}, \mu)) + o(t).
\end{align*}
\end{lemma}

\begin{proof}[Proof of Lemma \ref{L:D_partition_f}]
According to \cite[Theorems 1.3-4]{Wang19}, $\log(\frac{\det(\bD_\eta)}{|\mathfrak{m}_\eta|})$ equals $\frac{1}{12\pi}$ times the universal Liouville action defined in \cite[(0.1)]{TakhtajanTeo06} (see also \cite[(17)]{Wang19}. The variation of the universal Liouville action is computed in \cite[Chapter 2, Theorem 3.8]{TakhtajanTeo06} and gives the statement of the lemma.
\end{proof}

Combining Corollary \ref{C:sobolev} and Lemma \ref{L:D_partition_f} and assuming that $\eta$ is analytic, we get that
\begin{equation}\label{E:RN_analytic}
\RN_{\eta,\Phi_t}(\varphi) = 1 - 2 \Re \big( t \big( \frac{1}{12} \cS f^{-1} + \bT_\eta \varphi,\mu \big) \big) + o(t),
\quad \varphi \in H^{s}(\eta,\mathfrak{m}_\eta).
\end{equation}
Integrating, we deduce that for all $F\in\cC^0(H^{s}(\eta,\mathfrak{m}_\eta))$,
\begin{align*}
&\E_{\eta,\Phi_t}[F(\varphi)]-\E_\eta[F(\varphi)]
=-2\Re\Big(t\E_\eta\Big[\Big(\frac{1}{12}\cS f^{-1}+\bT_\eta\varphi,\mu\Big)F(\varphi)\Big]\Big)+o(t).
\end{align*}
More generally, by Lemma \ref{L:Beltrami_mus} we can get an expression of the derivative at any fixed small $u \in \C$. Recalling that $\mu_u$ is defined in \eqref{E:Beltrami_mus} and letting $f_u : \D \to \rmint(\Phi_s(\eta))$ be a uniformising map, we get: for all $t \in \R$ small
\begin{align}\label{E:analytic1}
&\E_{\eta,\Phi_t}[F(\varphi)]-\E_\eta[F(\varphi)]
=-2\Re\Big( \int_0^t \E_{\Phi_u(\eta)}\Big[\Big(\frac{1}{12}\cS f_u^{-1}+\bT_{\Phi_u(\eta)}\varphi,\mu_u\Big)F(\varphi \cdot \Phi_u)\Big] \d u\Big).
\end{align}

We now want to transfer this result to samples $\eta$ of the SLE loop measure.
The main point is that, since $\mu$ is compactly supported in $\rmint(\eta)$, the integrals appearing in the right hand side only feature well behaved terms, far from the irregular boundary.
Assume now that $\eta$ is sampled from $\nu^\#$. All the statements below are valid for $\nu^\#$-a.e. $\eta$.

For $\eps \ge 0$ and $t \in \C$ small, let $\eta_\eps = g(e^\eps \S^1)$ and $f_{t,\eps} : \D \to \rmint(\Phi_t(\eta_\eps))$ be the normalised Riemann mapping. We will simply write $f_\eps$ instead of $f_{0,\eps}$.

\begin{lemma}\label{L:limit_analytic}
    For all $t \in \C$ small, $(f_\eps \circ f^{-1})^* \P_{\eta_\eps,\Phi_t} \to \P_{\eta,\Phi_t}$ weakly as $\eps \to 0$ in $H^{s}(\eta,\mathfrak{m}_\eta)$ for any $s < -1/2.$
\end{lemma}
\begin{proof}
We first recall a tightness criterion from \cite[Example 3.8.13]{Bogachev}. Let $\Gamma$ be a family of centred Gaussian measures on some separable Hilbert space $X$. Then $\Gamma$ is tight with respect to the weak topology as soon as
\[ 
\sup_{\gamma \in \Gamma} \int \| x\|^2 \gamma(\d x) < \infty.
\]
We apply this to the space $X = H^s(\eta,\mathfrak{m}_\eta)$ for some $s<-1/2$. It is then an elementary computation to show that
$\E_{\eta_\eps,\Phi_t}[\| \varphi \circ f_\eps \circ f^{-1}\|_{H^s(\eta,\mathfrak{m}_\eta)}^2]$ is uniformly bounded.
The family $((f_\eps \circ f^{-1})^* \P_{\eta_\eps,\Phi_t})_{\eps>0}$ is thus tight in the space of probability measures on $H^s(\eta,\mathfrak{m}_\eta)$. Moreover, any subsequential limit is still Gaussian and has the correct covariance kernel.
\end{proof}

Let $F \in C^0(H^s(\eta,\mathfrak{m}_\eta))$ and $t \in \R$ small.
By Lemma \ref{L:limit_analytic} combined with \eqref{E:analytic1} applied to the analytic curve $\eta_\eps$, we have
\begin{align*}
    & \E_{\eta,\Phi_t}[F(\varphi)] - \E_\eta[F(\varphi)] = \lim_{\eps \to 0} \big( \E_{\eta_\eps,\Phi_t}[F(\varphi \circ (f_\eps\circ f^{-1}))] - \E_\eta[F(\varphi \circ (f_\eps\circ f^{-1}))] \big) \\
    & = -2 \lim_{\eps \to 0} \Re \Big( \int_0^t  \E_{\Phi_u(\eta_\eps)}\Big[\Big(\frac{1}{12}\cS f_{u,\eps}^{-1}+\bT_{\rmint(\Phi_u(\eta_\eps))} \varphi ,\mu_u\Big)F((\varphi \cdot \Phi_u) \circ (f_\eps\circ f^{-1}))\Big] \d u \Big).
\end{align*}
As already alluded to, the support of the Beltrami $\mu_u$ is assumed to be at a macroscopic distance to the rough curve $\Phi_u(\eta)$, so the above expectation converges and we get
\begin{align*}
    & \E_{\eta,\Phi_t}[F(\varphi)] - \E_\eta[F(\varphi)] \\
    & = -2 \Re \Big( \int_0^t  \E_{\Phi_u(\eta)}\Big[\Big(\frac{1}{12}\cS f_{u,0}^{-1}+\bT_{\rmint(\Phi_u(\eta))}\varphi,\mu_u\Big)F(\varphi \cdot \Phi_u)\Big] \d u \Big)
\end{align*}
which concludes the proof of Theorem \ref{T:GFF_SLE}.
\qed

 \section{Characterisation of the Neumann GFF: proof of Theorem \ref{T:Neumann}}\label{S:neumann_gff}

 In this section, we apply the results of Section \ref{S:liouville} to the case $\eta=\S^1$ in order to derive an integration by parts formula and a characterisation of the boundary Neumann GFF as stated in Theorem \ref{T:Neumann}. In Section~\ref{SS:FF_modules}, we recall some facts about Feigin--Fuchs modules (based on \cite[Lectures 2 \& 3]{KacRaina_Bombay}) and prove a characterisation of the $\log$-correlated field on the unit circle (Theorem \ref{T:FF}). In Section \ref{SS:witt}, we define a representation of the Witt algebra on the $L^2$-space of the Neumann GFF. We combine all these ingredients in Section \ref{SS:characterisation} and prove Theorem \ref{T:Neumann}.

    \subsection{Feigin--Fuchs modules}\label{SS:FF_modules}
We denote by $\P_{\S^1}$ the probability measure on the Sobolev space $\dot H^{-s}(\S^1)=\{\phi\in H^{-s}(\S^1)|\,\int_0^{2\pi}\phi(e^{i\theta})\d\theta=0\}$ whose samples are centred Gaussian fields on $\S^1$ with covariance
\begin{equation}\label{eq:covar_cicle}
\E_{\S^1}[\phi(z)\phi(\zeta)]=-\log|z-\zeta|,\qquad \forall z,\zeta\in\S^1.
\end{equation}
This differs from the boundary Neumann GFF on $\S^1$ by a factor of 2; see Section \ref{SS:witt}.
We may write such a field in Fourier series $\phi(z)=2\Re(\sum_{m=1}^\infty\phi_mz^m)$, and under $\P_{\S^1}$, the modes $(\phi_m)_{m\geq1}$ are independent complex Gaussians $\cN(0,\frac{1}{2m})$. Let $L^2_\mathrm{hol}(\P_{\S^1})$ be the closure in $L^2(\P_{\S^1})$ of the space $\C[(\phi_m)_{m\geq1}]$ of holomorphic polynomials. 

We can represent an integer partition by a sequence $\bk=(k_m)_{m\in\N^*}\in\N^{\N^*}$ containing only finitely many non-zero terms, where the number partitioned by $\bk$ is $|\bk|=\sum_{m=1}^\infty mk_m$.
We denote by $\cT$ the set of all integer partitions.

Fix $\alpha\in\C$ and set $\Delta_\alpha:=\frac{\alpha}{2}(Q-\frac{\alpha}{2})$. For $n\in\N_{>0}$, we define the following endomorphisms of $\C[(\phi_m)_{m\geq1}]$:
\begin{equation}
\bA_n:=\frac{i}{\sqrt{2}}\del_{\phi_n},\qquad\bA_{-n}:=\frac{\sqrt{2}}{i}n\phi_n,\qquad\bA_{0,\alpha}:=\frac{i}{\sqrt{2}}(Q-\alpha).
\end{equation}
They satisfy the Heisenberg algebra $[\bA_n,\bA_m]=n\delta_{n,-m}$ for all $n,m\in\Z$. Moreover, $\bA_n^*=\bA_{-n}$ on $L^2_\mathrm{hol}(\P_{\S^1})$ for all $n\neq0$, and $\bA_{0,\alpha}^*=\bA_{0,2Q-\bar\alpha}$. 
For $n \ne 0$, we define the operators
\begin{equation}\label{eq:def_ff}
\bL_{n,\alpha}^\mathrm{FF}:=\frac{i}{\sqrt{2}}Qn\bA_n+\frac{1}{2}\sum_{m\in\Z}\bA_{n-m}\bA_m,\qquad\bL_{0,\alpha}^\mathrm{FF}:=\Delta_\alpha+\sum_{m=1}^\infty\bA_{-m}\bA_m,
\end{equation}
and for any integer partition $\bk = (k_m)_{m \in \N^*} \in \cT$, denote by
\[
\bL_{-\bk,\alpha}^\mathrm{FF}:=\cdots(\bL_{-3,\alpha}^\mathrm{FF})^{k_3}(\bL_{-2,\alpha}^\mathrm{FF})^{k_2}(\bL_{-1,\alpha}^\mathrm{FF})^{k_1},\qquad\bL_{\bk,\alpha}^\mathrm{FF}:=(\bL_{1,\alpha}^\mathrm{FF})^{k_1}(\bL_{2,\alpha}^\mathrm{FF})^{k_2}(\bL_{3,\alpha}^\mathrm{FF})^{k_3}\cdots.
\]
Note the reversal of the order in which we read the partition.
From \cite[Section 3.4]{KacRaina_Bombay}, these operators satisfy the Virasoro algebra with central charge $c_\rL=1+6Q^2$, namely
\[[\bL_{n,\alpha}^\mathrm{FF},\bL_{m,\alpha}^\mathrm{FF}]=(n-m)\bL_{n+m,\alpha}^\mathrm{FF}+\frac{c_\rL}{12}(n^3-n)\delta_{n,-m}.\]
Moreover, for all $\bk \in \cT$, $(\bL_{\bk,\alpha}^\mathrm{FF})^*=\bL_{-\bk,2Q-\bar\alpha}^\mathrm{FF}$ on $L^2_\mathrm{hol}(\P_{\S^1})$. 
The module
\begin{equation}
    \label{E:moduleV}
\cV_\alpha:=\mathrm{span}\{\Psi_{\alpha,\bk}:=\bL_{-\bk,\alpha}^\mathrm{FF}\ind|\,\bk\in\cT\}
\end{equation}
is a highest-weight representation with weight $\Delta_\alpha$. If $\alpha\not\in kac^-$ is not on the Kac table \eqref{E:kac}, we have $\cV_\alpha=\C[(\phi_m)_{m\geq1}]$ and the states $\Psi_{\alpha,\bk}$ are linearly independent and form a basis of $\C[(\phi_m)_{m\geq1}]$ \cite{Frenkel92_determinant}.


\begin{theorem}\label{T:FF}
Let $\alpha\in\C\setminus(kac^+\cup kac^-)$. Let $\Q$ be a Borel probability measure on $\C^{\N^*}$ (equipped with the cylinder topology) such that the adjoint relations $\bL_{-n,\alpha}^\mathrm{FF}=(\bL_{n,2Q-\bar\alpha}^\mathrm{FF})^*$ hold on $L^2_\mathrm{hol}(\Q)$ for all $n>0$. Then, $\Q$-a.s., the series $z\mapsto2\Re(\sum_{m=1}^\infty \phi_mz^n)$ converges in $H^{-s}(\S^1)$ for all $s>0$, and $\Q=\P_{\S^1}$ as a Borel probability measure on $H^{-s}(\S^1)$.
\end{theorem}
\begin{proof}
For all integer partitions $\bk,\bk'$ with $|\bk'| \ge |\bk|$ say, we have from the adjoint relations that
\[\int\Psi_{\alpha,\bk}\overline{\Psi}_{2Q-\bar\alpha,\bk'}\d\Q=\langle\Psi_{\alpha,\bk},\Psi_{2Q-\bar\alpha,\bk'}\rangle_{L^2_\mathrm{hol}(\Q)}=
\langle \bL_{\bk',\alpha}^\mathrm{FF} \bL_{-\bk,\alpha}^\mathrm{FF} \ind , \ind \rangle_{L^2_\mathrm{hol}(\Q)}.
\]
Since $\cV_\alpha$ is a highest-weight representation, $\bL_{\bk',\alpha}^\mathrm{FF} \bL_{-\bk,\alpha}^\mathrm{FF} \ind = \mathbf{B}_\alpha(\bk,\bk')\ind$ is a constant multiple of $\ind$ and thus
\[
\int\Psi_{\alpha,\bk}\overline{\Psi}_{2Q-\bar\alpha,\bk'}\d\Q = \mathbf{B}_\alpha(\bk,\bk')\langle\ind,\ind\rangle_{L^2_\mathrm{hol}(\Q)} = \mathbf{B}_\alpha(\bk,\bk')
\]
because $\Q$ is a probability measure. The same is true for $\P_{\S^1}$ so that, for all $\bk,\bk' \in \cT$,
\[ 
\int\Psi_{\alpha,\bk}\overline{\Psi}_{2Q-\bar\alpha,\bk'}\d\Q = \int\Psi_{\alpha,\bk}\overline{\Psi}_{2Q-\bar\alpha,\bk'}\d\P_{\S^1}.
\]

Since $\alpha\not\in(kac^+\cup kac^-)$, we have $\cV_\alpha\otimes\overline{\cV}_{2Q-\bar\alpha}=\C[(\phi_m,\bar\phi_m)_{m\geq1}]$, hence $\cV_\alpha\otimes\overline{\cV}_{2Q-\bar\alpha}$ is dense in $L^1(\P_{\S^1})$ (here the overline denotes complex conjugation). Thus, $\Q$ extends (uniquely) to a continuous linear form on $L^1(\P_{\S^1})$, and it coincides with $\P_{\S^1}$ as such. In particular, it gives full mass to the event that $\sum_{m=1}^\infty|\phi_m|^2m^{-s}<\infty$ for all $s>0$, i.e. the field $z\mapsto2\Re(\sum_{m=1}^\infty\phi_mz^m)$ converges in $\dot H^{-s}(\S^1)$ with $\Q$-probability 1 and we can view $\Q$ as a Borel probability measure on $\dot H^{-s}(\S^1)$. Moreover, for all Borel sets $E\subset\dot{H}^{-s}(\S^1)$, we have $\ind_E\in L^1(\P_{\S^1})$, so $\Q(E)=\P_{\S^1}(E)$.
\end{proof}

    \subsection{Witt representation}\label{SS:witt}
Recall that the \emph{boundary Neumann GFF with zero average on $\S^1$} is the centred Gaussian field on $\S^1$ with covariance 
\[\dot\E[\phi(z)\phi(\zeta)]=-2\log|z-\zeta|,\qquad\forall z,\zeta\in\S^1.\]
The corresponding probability measure is denoted $\dot\P$ and gives full mass to $\dot H^{-s}(\S^1)$ for all $s>0$. Note the difference by a factor 2 with \eqref{eq:covar_cicle}: this may look innocent, but it is the source of all the differences in how the symmetry is encoded in the system.

Samples of $\dot\P$ have a Fourier expansion $\dot\phi(z)=2\Re(\sum_{m=1}^\infty\phi_mz^m)$, where the modes $(\phi_m)_{m\geq1}$ are independent complex Gaussians $\cN_\C(0,\frac{1}{m})$. The \emph{boundary Liouville field} is the measure $e^{-Qc}\d c\otimes\dot\P$, i.e. we independently sample the zero mode from $e^{-Qc}\d c$. The space of polynomials $\C[(\phi_m,\bar\phi_m)_{m\geq1}]$ is dense in $L^2(\dot\P)$, and the space 
\begin{equation}
    \label{E:cC}
\cC:=\cS(\R)\otimes\C[(\phi_m,\bar\phi_m)_{m\geq1}]
\end{equation}
is dense in $L^2(\d c\otimes\dot\P)$, where $\cS(\R)$ denotes the Schwartz class in the variable $c$.
    
Let us consider the group $\mathrm{Diff}^\omega(\S^1)$ of analytic diffeomorphisms of $\S^1$ on $L^2(\dot{\P})$. Its Lie algebra is the space $\mathrm{Vect}^\omega_\R(\S^1)=iz\cC^\omega(\S^1)\del_z$ of (real) analytic vector fields on $\S^1$ (where $\cC^\omega(\S^1)$ denotes the space of real-valued functions on $\S^1$ with an analytic extension to a complex neighbourhood of $\S^1$). It is generated by the vector fields $((\ra_n)_{n\geq0},(\rb_n)_{n\geq1})$ with $\ra_n:=\frac{i}{2}z(z^n+z^{-n})\del_z$ and $\rb_n:=\frac{1}{2}z(z^n-z^{-n})\del_z$. These generators satisfy the commutation relations
\begin{equation}\label{eq:commutations_vect}
\begin{aligned}
&[\ra_n,\ra_m]=\frac{n-m}{2}\rb_{n+m}+\frac{n+m}{2}\rb_{n-m};\qquad[\rb_n,\rb_m]=-\frac{n-m}{2}\rb_{n+m}+\frac{n+m}{2}\rb_{n-m};\\
&[\ra_n,\rb_m]=-\frac{n-m}{2}\ra_{n+m}+\frac{n+m}{2}\ra_{n+m}.
\end{aligned}
\end{equation}
In the sequel, we will use the convention $\rb_0=0$.


Let $\ru=izu(z)\del_z\in\mathrm{Vect}^\omega_\R(\S^1)$. Its exponential $h_t=e^{t\ru}$ is defined in $\mathrm{Diff}^\omega(\S^1)$ for real $t$, and satisfies $h_t(z)=z+itzu(z)+o(t)$ as $t\to0$. We say that a function $F\in L^2(\d c\otimes\dot\P)$ is differentiable in direction $\ru$ if the function $t\mapsto F(c+\phi\cdot h_t)$ is differentiable at $t=0$, for all $c+\phi\in H^{-s}(\S^1)$, where $\phi\cdot h_t$ is the action \eqref{E:doth}. In this case, we write
\[\cD_\ru^\R F(c+\phi):=\frac{\d}{\d t}_{|t=0}F(c+\phi\cdot(e^{-t\ru})).\]
Moreover, for all $n\in\N_{>0}$, we define
\begin{equation}
    \label{E:def_cDn}
\cD_n:=i\cD_{\mathrm{a}_n}^\R-\cD_{\mathrm{b}_n}^\R\qquad\text{and}\qquad\cD_{-n}:=i\cD_{\mathrm{a}_n}^\R+\cD_{\mathrm{b}_n}^\R.
\end{equation}

\begin{lemma}
For all $\ru\in\mathrm{Vect}^\omega_\R(\S^1)$, $\cD_\ru^\R$ defines an endomorphism of $\cC$. Moreover, for all $n\in\N_{>0}$, we have:
\begin{align*}
&\cD_n=nQ\del_{\phi_n}-n\bar\phi_n\del_c+\sum_{m=1}^\infty(m-n)\phi_{m-n}\del_{\phi_m}-(m+n)\bar\phi_{m+n}\del_{\bar\phi_m};\\
&\cD_{-n}=-nQ\del_{\bar\phi_n}+n\phi_n\del_c+\sum_{m=1}^\infty(m+n)\phi_{m+n}\del_{\phi_m}-(m-n)\bar\phi_{m-n}\del_{\bar\phi_m};\\
&\cD_0=\sum_{m=1}^\infty m\phi_m\del_{\phi_m}-m\bar\phi_m\del_{\bar\phi_m}.
\end{align*}
\end{lemma}
\begin{proof}
To ease notations, we will write $\phi_0$ instead of $c$ in this proof.
It suffices to compute $\cD_{\ra_n}^\R$ and $\cD_{\rb_n}^\R$. For $\ra_n$, we have
\begin{align*}
(\phi\cdot e^{t\ra_n})(z)
&=\sum_{m\in\Z}\phi_m(z+\frac{it}{2}z(z^n+z^{-n})^m)+Q\log\Big(1+\frac{i}{2}tz\del_z(z^n+z^{-n}) \Big)+o(t)\\
&=\frac{iQnt}{2}(z^n-z^{-n})+\sum_{m\in\Z}\left(\phi_m+\frac{it}{2}((m-n)\phi_{m-n}+(m+n)\phi_{m+n})\right)z^m+o(t).
\end{align*}
From this, we get (using the notation $\bar\phi_m=\phi_{-m}$)
\begin{align*}
\cD_{\ra_n}^\R
=-\frac{inQ}{2}(\del_{\phi_n}-\del_{\bar\phi_n})-\frac{i}{2}\sum_{m=0}^\infty&(m-n)\phi_{m-n}\del_{\phi_m}-(m-n)\bar\phi_{m-n}\del_{\bar\phi_m}\\
+&(m+n)\phi_{m+n}\del_{\phi_m}-(m+n)\bar\phi_{m+n}\del_{\bar\phi_m}.
\end{align*}
We get a similar expression for $\cD_{\rb_n}^\R$, and the result follows from the definition of $\cD_n$.
\end{proof}

The next lemma states that the family $(\cD_n)_{n\in\Z}$ forms a representation of the complexified Witt algebra.
\begin{lemma}\label{lem:witt}
As endomorphisms of $\C[(\phi_m,\bar{\phi}_m)_{m\geq1}]$, we have the commutation relations
\[[\cD_n,\cD_m]=(n-m)\cD_{n+m},\qquad\forall n,m\in\Z.\]
\end{lemma}
\begin{proof}
The proof is identical to \cite[Proposition 2.6]{BJ24}; we will only sketch it and leave the details to the reader. By construction, for all $\ru\in\mathrm{Vect}(\T)$, $\cD_\ru^\R$ is the Lie derivative along (minus) the fundamental vector field induced by $\ru$, which gives $[\cD^\R_\ru,\cD^\R_\rv]=\cD^\R_{[\ru,\rv]}$ for all $\ru,\rv\in\mathrm{Vect}^\omega(\S^1)$. The general formula follows from the $\C$-bilinearity of the commutator in $\mathrm{End}(\cC)$ and the commutation relations \eqref{eq:commutations_vect}.
\end{proof}

    \subsection{Proof of Theorem \ref{T:Neumann}}\label{SS:characterisation}
From here, the proof of Theorem \ref{T:Neumann} proceeds by computing the adjoints of the operators $\cD_n$ and ``Laplace transforming" in the zero mode, in order to make contact with the Feigin--Fuchs representation.

Denoting by
\begin{equation}
    \label{E:vector_field_vn}
    \rv_n = -z^{n+1}\del_z, \quad n \in \Z,
\end{equation}
we have:

\begin{proposition}\label{P:adjoints}
On $L^2(e^{-Qc}\d c\otimes\dot\P)$, we have for all $n\geq0$,
\[\cD_n^*=\cD_{-n}-(\bT_\D \phi,\rv_{-n}).\]
In particular, $\cD_n^*$ is densely defined for all $n\in\Z$ and $\cD_n$ is closable on $L^2(e^{-Qc}\d c\otimes\dot\P)$. 
\end{proposition}
\begin{proof}
Let us consider the family of diffeomorphisms $h_t=e^{t\ra_n}$ for small real $t$.
By definition \eqref{eq:commutations_vect} of $\ra_n$, we have on $\S^1$
\[
h_t(z) = z + 2itz \Re(w(z)) + o(t) \qquad \text{where} \qquad w(z) = z^{-n}/2.
\]
This expansion takes a similar form to the one we used in Section \ref{SS:liouville_disc}; see in particular \eqref{E:Ahlfors1}.
From Section \ref{S:gff_sle}, we know that the law of $\phi\cdot h_t^{-1}$ is absolutely continuous with respect to  $e^{-Qc}\d c\otimes\dot\P$, and the Radon--Nikodym derivative equals $e^{-\frac{1}{2}(\bS_\D(\phi\cdot h_t)-\bS_\D(\phi))}(1+O(t^2))$, with $1+O(t^2)$ corresponding to the contribution of the determinant which vanishes at order 2 and where $\bS_\D$ is the Liouville action \eqref{E:def_LiouvilleS1}. Hence, $(\cD_{\ra_n}^\R)^*=-\cD_{\ra_n}^\R+\frac{1}{2}\cD_{\ra_n}^\R\bS_\D$, and Proposition \ref{P:var_liouville} yields (or more precisely \eqref{E:pf_prop33})
\begin{align*}
    (\cD_{\ra_n}^\R)^*=-\cD_{\ra_n}^\R+ \frac{1}{\pi} \Re \Big( \oint \bT_\D \phi(z) zw(z) \d z \Big) =-\cD_{\ra_n}^\R+\Im(\bT_\D \phi,\rv_{-n}).
\end{align*}
Similarly, we find $(\cD_{\rb_n}^\R)^*=-\cD_{\rb_n}^\R+\Re(\bT_\D \phi,\rv_{-n})$. Hence,
\begin{align*}
\cD_n^*=-i(\cD_{\ra_n}^\R)^*-(\cD_{\rb_n}^\R)^*=i\cD_{\ra_n}^\R+\cD_{\rb_n}^\R-\Re(\bT_\D \phi,\rv_{-n})-i\Im(\bT_\D \phi,\rv_{-n})=\cD_{-n}-(\bT_\D \phi,\rv_{-n}),
\end{align*}
concluding the proof.
\end{proof}
    
Now, we ``Laplace transform" the operators $\cD_n,\cD_n^*$ in the zero mode $c$ and make contact with Feigin--Fuchs modules in order to characterise the measure $\dot\P$ using Theorem \ref{T:FF}. We set for $n>0$ and $\alpha\in\C$,
\begin{equation}\label{eq:def_D_alpha}
\begin{aligned}
&\cD_{n,\alpha}:=nQ\del_{\phi_n}+\frac\alpha2 n\bar\phi_n+\sum_{m=1}^\infty(m-n)\phi_{m-n}\del_{\phi_m}-(m+n)\bar\phi_{m+n}\del_{\bar\phi_m};\\
&\cD_{-n,\alpha}:=-nQ\del_{\bar\phi_n}-\frac\alpha2 n\phi_n+\sum_{m=1}^\infty(m+n)\phi_{m+n}\del_{\phi_m}-(m-n)\bar\phi_{m-n}\del_{\bar\phi_m}.
\end{aligned}
\end{equation}
They are obtained from the operators $\cD_n$ by replacing $\del_c$ with the multiplication by $\frac\alpha2$, and act (as densely defined operators) on $L^2(\dot\P)$. 

\begin{lemma}\label{L:adjoint_D_alpha}
For all $n>0$ and all $\alpha\in\C$, we have $\cD_{n,\bar\alpha}^*=\cD_{-n,2Q-\alpha}-(\bT_\D \phi,\rv_{-n})$ on $L^2(\dot\P)$. 
Moreover, $\cD_{n,2\alpha}^*$ preserves $\C[(\phi_m)_{m\geq1}]$ and coincides with $\bL_{-n,\alpha}^\mathrm{FF}$ there.
\end{lemma}
\begin{proof}
This can be checked by direct computation, but we propose a more conceptual argument. Consider the operator of multiplication by $e^{\frac{\alpha}{2}c}$ from $\cS'(\R)\otimes\C[(\phi_m,\bar\phi_m)_{m\geq1}]$ to itself. It is plain to check that operator $e^{-\frac\alpha2 c}\circ\cD_n\circ e^{\frac\alpha2 c}$ preserves $\C[(\phi_m,\bar\phi_m)_{m\geq1}]$ (viewed as a subspace of $\cS'(\R)\otimes\C[(\phi_m,\bar\phi_m)_{m\geq1}]$), and it coincides with $\cD_{n,\alpha}$ there. For $p\in\R$, we have $\cD_{n,Q-ip}^*=e^{-\frac{1}{2}(Q+ip)c}\circ\cD_n^*\circ e^{\frac{1}{2}(Q+ip)c}=\cD_{-n,Q+ip}-(\bT_\D \phi,\rv_{-n})$, where the adjoint in the LHS (resp. RHS) is taken in $L^2(\dot\P)$ (resp. $L^2(e^{-Qc}\d c\otimes\dot\P)$. The formula extends to all $\alpha\in\C$ by analytic continuation.
The last statement of the lemma is obvious from \eqref{eq:def_D_alpha} and \eqref{eq:def_ff}.
\end{proof}

\begin{proof}[Proof of Theorem \ref{T:Neumann}]
Let $\Q$ be a probability measure on $\C^{\N^*}$ be such that the adjoint relations of Lemma \ref{L:adjoint_D_alpha} hold on $L^2(\Q)$ for some $\alpha \in \C \setminus (kac^+ \cup kac^-)$.
We introduce the unitary map $U:L^2(\P_{\S^1})\to L^2(\dot\P)$ defined by $UF(\phi):=F(\frac{1}{\sqrt{2}}\phi)$.
For $n>0$, Lemma \ref{L:adjoint_D_alpha} determines the adjoint of $\bL_{-n,\alpha}^\mathrm{FF}$ on $L^2_\mathrm{hol}(\Q)$, hence on $L^2_\mathrm{hol}(U^*\Q)$, where it coincides with $\bL_{n,2Q-\bar\alpha}^\mathrm{FF}$. We can conclude from Theorem \ref{T:FF} that $U^*\Q=\P_{\S^1}$, i.e. $\Q=\dot\P$. 
\end{proof}

\section{Welding homeomophism of SLE: proof of Theorem \ref{T:ghost}}\label{S:ghost}
\subsection{Setup and preliminaries}
Following \cite{BJ24}, we denote by $\cE$ the space
\begin{equation}\label{eq:def_E}
\cE:=\{f:\D\to\C\text{ conformal},\,f\text{ continuous \& injective on }\bar\D,\,f(0)=0,\,f'(0)=1\}.
\end{equation}
This space comes equipped with the uniform topology, and we recall that the shape measure $\nu^\#$ is a Borel measure on $\cE$ \cite[Proposition B.1]{BJ24}.

Let $f\in\cE$ and consider the Riemann mapping $g:\D^*\to\mathrm{ext}(f(\S^1))$ normalised such that $g(\infty)=\infty$ and $g(1)=f(1)$. Then, $g^{-1}\circ f|_{\S^1}\in\mathrm{Homeo}_1(\S^1)$, and this gives a function $\cE\to\mathrm{Homeo}_1(\S^1)$ (which is injective on the preimage of $\mathrm{Homeo}_1(\S^1)\cap\mathrm{Homeo}^\star(\S^1)$). We equip $\mathrm{Homeo}_1(\S^1)$ with the smallest topology which makes this function continuous. Note that $\mathrm{Homeo}_1(\S^1)$ has two connected components in this topology: the image of $\cE$ and its complement (the last set comes with the trivial topology). We have an identification $\mathrm{Homeo}(\S^1)\simeq\S^1\times\mathrm{Homeo}_1(\S^1),\,h\mapsto(h(1),h(1)^{-1}h)$.

Recall Definition \ref{def:tilde_nu} and the fact that $\tilde{\nu}$ can also be viewed as the (Borel) measure $\frac{\d\alpha}{2\pi}\otimes\cE$ on $\S^1\times\cE$, with $e^{i\alpha}$ parametrising the uniform rotation. We will write the samples of $\tilde{\nu}$ either as $h$ (when viewed as a homeomorphism of $\S^1$) or as pairs $(e^{i\alpha},f)$ with $\alpha\in[0,2\pi)$ and $f\in\cE$. The latter description gives us a natural space of test functions: denote by $f(z)=z(1+\sum_{m=1}^\infty a_mz^m)$ the expansion of $f$ around $0$, and define the space of test functions as $\cC^\infty(\S^1)\otimes\C[(a_m,\bar a_m)_{m\geq1}]$. The latter space is the space of polynomials in the coefficients $(a_m)_{m\geq1}$ and is dense in $L^2(\nu^\#)$ \cite[Section 2.2]{BJ24}. 

To summarise the setup, we can view $\nu_\kappa$ as a scale invariant measure on $\R\times\cE\simeq\cJ_{0,\infty}$, and $\tilde{\nu}_\kappa$ as a rotationally invariant measure on $\S^1\times\cE$. As far as deformations are concerned, we can either \emph{postcompose} $f\in\cE$ (viewed as an element of $\cJ_{0,\infty}$) by a map which is conformal in a neighbourhood of $f(\S^1)$ (as done in \cite{BJ24}); or we can \emph{precompose} $f$ (viewed as an element of $\mathrm{Homeo}(\S^1)$) with an analytic diffeomorphism of $\S^1$. These two symmetries induce two distinct sets of vector fields on $\cE$. The main result of this section can be understood as the computation of the determinant of the change of basis between these two sets of vector fields. In order to make this rigorous, we will study the effect of differentiating along these vector fields. Hence, we introduce two sets of differential operators corresponding to the Lie derivatives along these vector fields. The latter family was introduced in \cite{BJ24}. The main goal of this section will be to relate these two families, i.e. relate small perturbations in $\mathrm{Homeo}(\S^1)$ to small perturbations in $\cJ_{0,\infty}$, and rely on the integration by parts formula from \cite{BJ24}.

$\bullet$ Operators on $L^2(\tilde{\nu})$.
Let $\mu \in L^\infty(\hat\C)$ be a Beltrami differential compactly supported in $\D$ (resp. $\D^*$), and let $\tilde{\Phi}_t$ be the solution to the Beltrami equation with coefficient $t\mu+\bar t\iota^*\mu$, normalised to fix 0, 1, $\infty$ (defined in a complex neighbourhood of $t=0$). We say that $F:\mathrm{Homeo}(\S^1)\to\R$ is right (resp. left) differentiable along $\mu$ at $h\in\mathrm{Homeo}(\S^1)$ if it admits a first order expansion
\begin{equation}
    \label{E:def_sL_sR}
\begin{aligned}
&F(h\circ\tilde{\Phi}_t)=F(h)+t\sR_\mu F(h)+\bar t\bar\sR_\mu F(h)+o(t) \\
\text{resp.}\qquad&F(\tilde{\Phi}_t\circ h)=F(h)+t\sL_\mu F(h)+\bar t\bar\sL_\mu F(h)+o(t).
\end{aligned}
\end{equation}
We introduce the operator $\Theta$ acting on $L^2(\tilde{\nu})$ by:
\begin{equation}\label{E:Theta}
    \Theta F : h \in \mathrm{Homeo}(\S^1) \mapsto F(h^{-1}) \in \R, \qquad
    F \in L^2(\tilde\nu).
\end{equation}
First, note that the law of $h^{-1}$ under $\tilde\nu$ is invariant under left/right rotations. Moreover, writing a conformal welding $h=g^{-1}\circ f$ and observing that $\iota$ fixes $\S^1$ setwise, we have the identity $h^{-1}=(\iota\circ f\circ\iota)^{-1}\circ(\iota\circ g\circ\iota)$. Hence, viewed as an operator on $L^2(\nu^\#)$, the restriction of $\Theta$ to $\mathrm{Homeo}(\S^1)$ coincides with the operator $\Theta$ from \cite{BJ24}. In particular, $\Theta$ is unitary and self-adjoint on $L^2(\tilde{\nu})$.

$\bullet$ Operators on $L^2(\nu^\#)$.
In \cite{BJ24}, we introduced and studied in great detail analogous differential operators acting on $L^2(\nu^\#)$. See also \cite{GQW25}. For any Laurent polynomial vector field $\rv=v(z) \del_z \in \C(z)\del_z$, consider the associated flow $\Phi_t(z) = z + tv(z) + o(t)$. We say that a function $F \in L^2(\nu^\#)$ is differentiable at $\eta \in \cE$ in direction $\rv$ if
\begin{equation}\label{E:cL}
    F(\Phi_t(\eta)) = F(\eta) + t \cL_\rv F(\eta) + \bar t \bar\cL_{\rv} F(\eta) + o(t)
\end{equation}
as $t \to 0$ where we have identified the uniformising map $f$ with the Jordan curve $\eta$.
As shown in \cite{BJ24}, the differential operators $\cL_\rv$, $\bar\cL_\rv$ are densely defined operators on $L^2(\nu^\#)$. We will often consider the vector fields $\rv_n = -z^{n+1}\del_z$ \eqref{E:vector_field_vn}.

Now, we record two preparatory lemmas. Lemma \ref{L:Theta} relates $\sR_\mu$ to $\sL_\mu$, and Lemma \ref{L:mobius} studies the action of M\"obius transformations on $\tilde{\nu}$.

\begin{lemma}\label{L:Theta}
For all Beltrami differentials $\mu\in L^\infty(\D)$, we have $\sR_\mu=-\Theta\circ\sL_{\iota^*\mu}\circ\Theta$.
\end{lemma}
\begin{proof}
For all $\tilde{\Phi}\in\mathrm{Diff}^\omega(\S^1)$, and all test functions $F$, we have $F(h\circ\tilde{\Phi}^{-1})=F((\tilde{\Phi}\circ h^{-1})^{-1})$. Now, let $\mu$ be a Beltrami differential compactly supported in $\D$ and let $\tilde{\Phi}_t$ be the normalised solution to the Beltrami equation with coefficient $t\mu + \bar t \iota^*\mu$. To first order, $\tilde\Phi_t^{-1}$ is the normalised solution of the Beltrami equation with coefficient $-t \mu - \bar t \iota^*\mu$. So the right composition by $\tilde{\Phi}_t^{-1}$ produces $-\sR_\mu$, while the left composition by $\tilde{\Phi}_t$ produces $\sL_{\iota^*\mu}$. Hence, taking a complex derivative at $t=0$ gives $-\sR_\mu F=\Theta\circ\sL_{\iota^*\mu}\circ\Theta$ as required.
\end{proof}

Recall by definition that $\tilde{\nu}$ is invariant under left and right composition by rotations, which amounts to choosing the Riemann maps $f:\D\to\mathrm{int}(\eta)$ and $g:\D^*\to\mathrm{ext}(\eta)$ uniformly at random among those fixing $0$ and $\infty$ respectively. Of course, we could lift $\tilde{\nu}$ to an invariant measure under the left and right action of the \emph{whole} M\"obius group $\mathrm{PSL}_2(\R)$ by choosing $f$ and $g$ uniformly among \emph{all} such maps (without the requirement that they fix $0$ and $\infty$). This amounts to precomposing the normalised Riemann mapping by a Haar-distributed M\"obius transformation of the disc (independent of everything). The next lemma gives a more compact description of this invariant measure.

\begin{lemma}\label{L:mobius}
The measure $e^{2\bK}\tilde{\nu}$ is invariant by left and right composition by the group of M\"obius transformations of $\S^1$.
\end{lemma}
In \cite{AngCaiSunWu_1}, the Velling--Kirillov potential $\bK$ runs by the name \emph{electrical thickness}, and the authors find an exact formula for its generating function $\E^\#[e^{\lambda\bK}]$, which converges for $\lambda<1-\frac{\kappa}{8}$. In particular, the measure $e^{2\bK}\tilde{\nu}$ is infinite, which is consistent with the fact that it is M\"obius-invariant (Haar measure is infinite).

\begin{proof}[Proof of Lemma \ref{L:mobius}]
 We have a parametrisation of the group of M\"obius transformations of $\D$ by $z\mapsto e^{i\alpha}\frac{z-a}{1-\bar az}$ with $\alpha\in[0,2\pi)$ and $a\in\D$. In this proof and this proof only, we will write $\tilde{\Phi}_a(z):=\frac{z-a}{1-\bar az}$ and $\Phi_a(z):=z-a$. Note that $\tilde{\Phi}_a^{-1}=\tilde{\Phi}_{-a}$ and $\Phi_a(\infty)=\infty$, $\Phi_a'(\infty)=1$. By rotation invariance, it suffices to prove that $\tilde{\nu}$ is invariant under left/right composition by all $\tilde{\Phi}_a$, $a\in\D$.

We will only prove right invariance, since left invariance is identical. It will also be more convenient to normalise the SLE shape measure to have unit conformal radius \emph{from $\infty$}, i.e. $g'(\infty)=1$ and the Velling--Kirillov potential is $\bK=-\log|f'(0)|$. Let $a\in\D$ close to 0 and let us consider the small motion $h\circ\tilde{\Phi}_{-a}=g^{-1}\circ f\circ\tilde{\Phi}_{-a}$. We have $f(\tilde{\Phi}_{-a}(0))=af'(0)+o(a)$ and so, after compensating the $af'(0)$, we have $\Phi_{af'(0)}\circ f\circ\tilde{\Phi}_a(0)=o(a)$. Writing $h\circ\tilde{\Phi}_{-a}=(\Phi_{af'(0)}\circ g)^{-1}\circ(\Phi_{af'(0)}\circ f\circ\tilde{\Phi}_{-a})$, we see that the small motion of $g$ is $\Phi_{af'(0)}\circ g$ at order 1 in $a$. Hence, using the M\"obius invariance of the SLE loop measure, we have for all test functions $F$,
\begin{align*}
\int F(h\circ\tilde{\Phi}_{-a})e^{2\bK}\d\tilde{\nu}
&=\int F(\Phi_{af'(0)}\circ g\circ e^{i\alpha})|f'(0)|^{-2}\d\nu^\#(g)\frac{\d\alpha}{2\pi}+o(a).
\end{align*}
The curve $\Phi_{af'(0)} \circ g(\S^1)$ is the translated curve $g(\S^1) - af'(0)$ and the operators $\cL_{\rv_{-1}}$ and $\bar\cL_{\rv_{-1}}$ \eqref{E:cL} are the generators of such translations: 
\[ 
F(\Phi_{af'(0)} \circ g(\S^1)) = F(g) + a f'(0) \cL_{\rv_{-1}}F(g) + \overline{af'(0)} \bar\cL_{\rv_{-1}}F(g) + o(a).
\]
Hence,
\[ 
\int F(h\circ\tilde{\Phi}_{-a})e^{2\bK}\d\tilde{\nu} = \int Fe^{2\bK}\d\tilde\nu-2\Re\left(\bar a\int \bar\cL_{\rv_{-1}}\left(f'(0)\right)F\d\tilde\nu\right)+o(a).
\]
To conclude, we need to prove that $\bar\cL_{\rv_{-1}}(f'(0))=0$. Consider the motion $\Phi_a\circ g$ for $a$ close to 0. The motion of $f$ is $\Phi_a\circ f\circ\tilde{\Phi}_{\frac{-a}{f'(0)}}$ at order 1 in $a$. Since $\tilde{\Phi}_a'(0)=1-|a|^2=1+o(a)$, we have $(\Phi_a\circ f\circ\tilde{\Phi}_{\frac{-a}{f'(0)}})'(0)=f'(0)+a\frac{f''(0)}{f'(0)}+o(a)$. This first order expansion has no $\bar a$-term, so $\bar\cL_{\rv_{-1}}(f'(0))=0$ as claimed. This proves that the Lie derivative of $e^{2\bK}\tilde{\nu}$ vanishes along vector fields generating the M\"obius transformations $\tilde{\Phi}_a$. To conclude for all fixed $a\in\D$, we just integrate this identity along the straight line from 0 to $a$ in $\D$.
\end{proof}

    \subsection{Proof of Theorem \ref{T:ghost}}
The proof of Theorem \ref{T:ghost} relies on three computational Lemmas~\ref{L:weight}, \ref{L:ghost} and \ref{lem:expand_K} below, whose statements and proofs are postponed to the end of this section in order to ease the reading of the proof.

Let $\mu \in L^\infty(\hat\C)$ be compactly supported in $\D^*$, and let us prove the formula for $\sL_\mu^*$. We note that the operator $\sL_\mu$ depends only on the restriction of the vector field $w_\mu$ to $\S^1$. This vector field is the Cauchy transform of $\mu$ up to the M\"obius gauge. Moreover, Lemma \ref{L:mobius} takes care of those vector fields generating M\"obius transformations of $\S^1$, so we may well restrict to a convenient set of representatives: for $n\geq2$, the vector field $\rv_n = -z^{n+1}\del_z$ is the Cauchy transform of the Beltrami differential $\mu_n(\zeta):=-n4^n\zeta^{-n-1}\zeta\ind_{|\zeta|>2}$.

Let $\tilde{\Phi}_t$ be the solution to the Beltrami equation with coefficient $t\mu +\bar t \iota^* \mu$ fixing $0$, $1$ and $\infty$. We wish to express the motion $\tilde{\Phi}_t \circ h$ in the space of homeomorphisms as a motion in the space $\cE$ \eqref{eq:def_E}.
Let $\Phi_t$ be the solution to the Beltrami equation with coefficient $g_* \mu$ normalised so that $\Phi_t(0)=0$, $\Phi_t'(0)=1$ and $\Phi_t(\infty) = \infty$. 
By the chain rule, $g_t:=\Phi_t\circ g\circ\tilde{\Phi}_t^{-1}$ has a vanishing $\del_{\bar z}$ derivative in $\D^*$ and is thus conformal by Weyl's lemma. Moreover, since $\mu$ is supported in $\D^*$, $f_t:=\Phi_t\circ f$ is conformal in $\D$. The normalisation of $\Phi_t$ then guarantees that $f_t$ and $g_t$ are properly normalised at $0$ and $\infty$. We have thus described the motion $\tilde{\Phi}_t\circ h=g_t^{-1}\circ f_t$ in terms of a motion in $\cE$.

Now, due to the normalisation of $\Phi_t$, we have by \eqref{E:qc_vendredi} the expansion
$\Phi_t(z)=z+t(v_\mu(z)-v_\mu(0)-zv_\mu'(0))+o(t)$ where
\begin{equation}
v_\mu(\xi)
=-\frac{1}{\pi}\int_\C\frac{\mu\circ g^{-1}(z)}{\xi-z}\frac{\overline{(g^{-1})'(z)}}{(g^{-1})'(z)}|\d z|^2
=-\frac{1}{\pi}\int_{\D^*}\frac{\mu(z)}{\xi-g(z)}g'(z)^2|\d z|^2.
\end{equation}
In the neighbourhood of $\xi=0$, we have the power series expansion
\begin{equation}
v_\mu(\xi)=-\sum_{n=-1}^\infty\beta_n^\mu\xi^{n+1}\qquad\text{with}\qquad\beta_n^\mu=-\frac{1}{\pi}\int_{\D^*}\mu(z)g(z)^{-n-2}g'(z)^2|\d z|^2.
\end{equation}

 By rotational invariance of $\tilde{\nu}$, it is enough to restrict our attention to test functions, which are independent of the rotation parameter $\alpha$, i.e. we will pick test functions in $\C[(a_m,\bar a_m)_{m\geq1}]$. For $\tilde{\nu}$-a.e. $h=g^{-1}\circ f\in\mathrm{Homeo}(\S^1)$ and all Beltrami differentials $\mu$ compactly supported in a neighbourhood of $\infty$, any test function $F$ is differentiable at $f$ in direction $g_*\mu$ \cite[Proposition 2.6]{BJ24}. Moreover, we have $\cL_nF=0$ for $n$ large enough. Hence, using \cite[Theorem 4.1]{BJ24}, we get
\begin{align*}
\int\sL_\mu F\d\tilde{\nu}
=\int\sum_{n=1}^\infty\beta_n^\mu\cL_n(F)\d\nu^\#
&=-\int\sum_{n=1}^\infty\left(\frac{c_\rM}{12}\beta_n^\mu(\cS g^{-1},\rv_n)+\cL_n(\beta_n^\mu)\right)F\d\nu^\#.
\end{align*}
Combining Lemmas \ref{L:weight} and \ref{L:ghost} below, we have that $\sum_{n=1}^\infty \cL_n(\beta_n^\mu) = \frac{26}{12}(\cS g,\mu) - \beta_0^\mu$ in $L^1(\tilde\nu)$. Since $F$ is a test function, we can then apply dominated convergence to get
\begin{align*}
\int\sL_\mu F\d\tilde{\nu} = -\int\left(\frac{c_\rM}{12}(\cS g^{-1},g_*\mu)+\frac{26}{12}(\cS g,\mu)-\beta_0^\mu\right)F\d\nu^\#.
\end{align*}
By \eqref{E:schwarzian_mu}, $(\cS g^{-1},g_*\mu) = -(\cS g,\mu)$.
Altogether, we have obtained that
\[ 
\int\sL_\mu F\d\tilde{\nu}
= \int\left(\frac{c_\rM-26}{12}(\cS g,\mu)+\beta_0^\mu\right)F\d\tilde{\nu}.
\]
This shows that $\sL_\mu^*$ coincides with the operator $-\bar\sL_\mu+\frac{c_\rM-26}{12}\overline{(\cS g,\mu)}\overline{\beta}_0^\mu$ on a dense subspace of $L^2(\tilde\nu)$, which proves the claim. In particular, both $\sL_\mu$ and $\sL_\mu^*$ are closable. The formula for $\sR_\mu^*$ is an immediate consequence of Lemma~\ref{L:Theta} and \cite[Lemma 4.4]{BJ24}. \qed

\begin{lemma}\label{L:weight}
We have $\cL_{\rv_0}(\beta_0^\mu)=0$ and $\cL_{\rv_{-1}}(\beta_{-1}^\mu)=\beta_0^\mu$.
\end{lemma}

\begin{proof}
In this proof, we make the dependence $\beta_n^\mu=\beta_n^\mu(\eta)$ explicit.
The vector field $\rv_0$ is the generator of dilations $z\mapsto e^{-t}z$, and we have 
\[\beta_0^\mu(e^{-t}\eta)=-\frac{1}{\pi}\int_{\D^*}\mu(z)\frac{(e^{-t}g)'(z)^2}{(e^{-t}g)(z)^2}|\d z|^2=\beta_0^\mu(\eta).\]
Hence, $\cL_{\rv_0}(\beta_0^\mu)=0$.
The vector field $\rv_{-1}$ is the generator of translations $z\mapsto z-t$. The map $-t+g$ uniformises $-t+\eta$ and satisfies $-t+g(\infty)=\infty$. Thus,
\begin{align*}
\beta_{-1}^\mu(-t+\eta)
&=-\frac{1}{\pi}\int_{\D^*}\mu(z)\frac{g'(z)^2}{g(z)-t}|\d z|^2\\
&=-\frac{1}{\pi}\int_{\D^*}\mu(z)\frac{g'(z)^2}{g(z)}\big(1+\frac{t}{g(z)}\big)|\d z|^2+o(t)=\beta_{-1}^\mu(\eta)+t\beta_0^\mu(\eta)+o(t).
\end{align*}
\end{proof}

\begin{lemma}\label{L:ghost}
We have $\sum_{n=-1}^\infty\cL_{\rv_n}(\beta_n^\mu)=\frac{26}{12}(\cS g,\mu)$ in $L^1(\tilde{\nu})$.
\end{lemma}
\begin{proof}
Recall that $w_\mu$ (resp. $v_\mu$) is the Cauchy transform of $\mu$ (resp. $g_*\mu$), and $\tilde{\Phi}_t(z)=z+tw_\mu(z)+\bar t\tilde{w}_\mu(z)+o(t)$. Let $\Phi_t=z+tv_\mu+o(t)$ be the flow of $v_\mu\del_z$ and $g_t:=\Phi_t\circ g\circ\tilde{\Phi}_t^{-1}$. At first order in $t$, we have $g_t=g+t(v_\mu\circ g-w_\mu g')+\bar tO(1)+o(t)$. By Cauchy's integral formula, we can write
\[v_\mu\circ g(z)-w_\mu(z)g'(z)=\frac{1}{2i\pi}\oint K(g(z),\zeta)v_\mu(\zeta)\d\zeta,\]
with the kernel $K$ defined on $g(\D^*)\times g(\D^*)$ by
\[K(z,\zeta):=\frac{(g^{-1})'(\zeta)^2}{(g^{-1})'(z)(g^{-1}(z)-g^{-1}(\zeta))}-\frac{1}{z-\zeta}.\]
It is easy to see that $K$ is holomorphic in both variables, with the apparent pole on the diagonal being removable. By the Leibniz rule, we have
\begin{align*}
\cL_{\rv_n}(\beta_n^\mu)
&=-\frac{1}{\pi}\int_{\D^*}\mu(z)\left(2g'(z)g(z)^{-n-2}\cL_{\rv_n}(g'(z))+(n+2)g'(z)^2g(z)^{-n-3}\cL_{\rv_n}(g(z))\right)|\d z|^2\\
&=\frac{1}{i\pi^2}\int_{\D^*}\oint\mu(z)g'(z)g(z)^{-n-2}\zeta^{n+1}\del_z(K(g(z),\zeta))\d\zeta|\d z|^2\\
&\quad-\frac{n+2}{2i\pi^2}\int_{\D^*}\oint\mu(z)g'(z)^2g(z)^{-n-3}\zeta^{n+1}K(g(z),\zeta)\d\zeta|\d z|^2.
\end{align*}
Up to choosing an equivalent Beltrami differential supported in a smaller neighbourhood of $\infty$, we can assume that $|\zeta/g(z)|<1$ and $|g'(z)/g(z)|<1$, for all $z$ in the support of $\mu$ and all $\zeta$ on the integration contour. In this case, summing over $n\geq-1$ yields a geometric series converging in $L^1(\tilde{\nu})$. Thanks to the residue theorem, the sum evaluates to
\begin{equation}\label{eq:end_ghost}
\begin{aligned}
& \sum_{n=-1}^\infty\cL_{\rv_n}\beta_n^\mu
=\frac{1}{i\pi^2}\int_{\D^*}\oint\mu(z)\frac{g'(z)}{\zeta-g(z)}\del_z(K(g(z),\zeta))\d\zeta|\d z|^2\\
&\hspace{80pt} - \frac{1}{2i\pi^2}\int_{\D^*}\oint\mu(z)\frac{g'(z)^2}{(\zeta-g(z))^2}K(g(z),\zeta)\d\zeta|\d z|^2\\
&\quad=\frac{2}{\pi}\int_{\D^*}\mu(z)g'(z)\del_z(K(g(z),\zeta))_{|\zeta=g(z)}|\d z|^2+\frac{1}{\pi}\int_{\D^*}\mu(z)g'(z)^2\del_\zeta K(g(z),\zeta)_{|\zeta=g(z)}|\d z|^2\\
&\quad=\frac{1}{\pi}\int_{\D^*}\mu(z)g'(z)^2\left(2\del_zK(g(z),\zeta)_{|\zeta=g(z)}+\del_\zeta K(g(z),\zeta)_{|\zeta=g(z)}\right)|\d z|^2.
\end{aligned}
\end{equation}
By Lemma \ref{lem:expand_K} below applied to $\psi=g^{-1}$ and the chain rule for the Schwarzian derivative,
\[ 
2\del_zK(g(z),\zeta)_{|\zeta=g(z)}+\del_\zeta K(g(z),\zeta)_{|\zeta=g(z)} = -\frac{13}{6}\cS g^{-1}(g(z)) = \frac{13}{6} \cS g(z)/g'(z)^2.
\]
Hence,
\[ 
\sum_{n=-1}^\infty\cL_{\rv_n}\beta_n^\mu
= \frac{13}{6\pi}\int_{\D^*}\mu(z)\cS g(z)|\d z|^2.
\]
\end{proof}

For the next lemma, we recall that $\cA f=\frac{f''}{f'}$ is the pre-Schwarzian derivative.

\begin{lemma}\label{lem:expand_K}
Let $\psi$ be conformal on some planar domain $D$, and define the holomorphic kernel $K_\psi(z,\zeta):=\frac{\psi'(\zeta)^2}{\psi'(z)(\psi(z)-\psi(\zeta))}-\frac{1}{z-\zeta}$ on $D\times D$.
We have
\begin{align*}
\del_zK_\psi(z,\zeta)_{|z=\zeta}=-\frac{2}{3}\cS\psi(z)+\frac{3}{4}\cA\psi(z)^2 \quad \text{and} \quad \del_\zeta K_\psi(z,\zeta)_{|\zeta=z}=-\frac{5}{6}\cS\psi(z)-\frac{3}{2}\cA\psi(z)^2.
\end{align*}
In particular, $(2\del_z+\del_\zeta)K_\psi(z,\zeta)_{|z=\zeta}=-\frac{13}{6}\cS\psi(z)$.
\end{lemma}

\begin{proof}
Fixing $\zeta$ and expanding in $z$ as $z\to\zeta$, we get that $K_\psi(z,\zeta)$ equals
\begin{align*}
&\frac{\psi'(\zeta)^2}{(\psi'(\zeta)+(z-\zeta)\psi''(\zeta)+\frac{1}{2}(z-\zeta)^2\psi'''(\zeta))((z-\zeta)\psi'(\zeta)+\frac{1}{2}\psi''(\zeta)(z-\zeta)
^2+\frac{1}{6}(z-\zeta)^3\psi'''(\zeta))}\\
&\quad-\frac{1}{z-\zeta}+o(z-\zeta)\\
&=\frac{1}{z-\zeta}\left(1-(z-\zeta)\cA \psi(\zeta)-\frac{1}{2}(z-\zeta)^2\frac{\psi'''(\zeta)}{\psi'(\zeta)}+(z-\zeta)^2\cA \psi(\zeta)^2\right)\\
&\qquad\times\left(1-\frac{1}{2}(z-\zeta)\cA \psi(\zeta)-\frac{1}{6}(z-\zeta)^2\frac{\psi'''(\zeta)}{\psi'(\zeta)}+\frac{1}{4}(z-\zeta)^2\cA \psi(\zeta)^2\right)-\frac{1}{z-\zeta}+o(z-\zeta)\\
&=\frac{3}{2}\cA\psi(\zeta)-\frac{2}{3}(z-\zeta)\frac{\psi'''(\zeta)}{\psi'(\zeta)}+\frac{7}{4}(z-\zeta)\cA \psi(\zeta)^2+o(z-\zeta).
\end{align*}
From this, we get $\del_zK_\psi(z,\zeta)_{|z=\zeta}=\frac{7}{4}\cA \psi(z)^2-\frac{2}{3}\frac{\psi'''(z)}{\psi'(z)} = \frac{3}{4} \cA \psi(z)^2 - \frac23 \cS \psi(z)$.
Now, fix $z$ and expand in $\zeta$ as $\zeta\to z$:
\begin{align*}
K_\psi(z,\zeta)
&=\frac{(\psi'(z)+(\zeta-z)\psi''(z)+\frac{1}{2}(\zeta-z)^2\psi'''(z))^2}{\psi'(z)(z-\zeta)(\psi'(z)+\frac{1}{2}(\zeta-z)\psi''(z)+\frac{1}{6}(\zeta-z)^2\psi'''(z))}-\frac{1}{z-\zeta}+o(\zeta-z)\\
&=\frac{1}{z-\zeta}\left(1+2(\zeta-z)\cA \psi(z)+(\zeta-z)^2\frac{\psi'''(z)}{\psi'(z)}+(\zeta-z)^2\cA \psi(z)^2\right)\\
&\qquad\times\left(1-\frac{1}{2}(\zeta-z)\cA\psi(z)-\frac{1}{6}(\zeta-z)^2\frac{\psi'''(z)}{\psi'(z)}+\frac{1}{4}(\zeta-z)^2\cA \psi(z)^2\right)-\frac{1}{z-\zeta}+o(\zeta-z)\\
&=-\frac{3}{2}\cA\psi(z)-\frac{1}{4}(\zeta-z)\cA\psi(z)^2-\frac{5}{6}(\zeta-z)\frac{\psi'''(z)}{\psi'(z)}+o(\zeta-z).
\end{align*}
From this, we get $\del_\zeta K_\psi(z,\zeta)_{|\zeta=z}=-\frac{1}{4}\cA \psi(z)^2-\frac{5}{6}\frac{\psi'''(z)}{\psi'(z)} = -\frac32 \cA \psi(z)^2 - \frac56 \cS\psi(z)$.
\end{proof}

It is easy to deduce Corollary \ref{C:RN_welding} as an integral version of Theorem \ref{T:ghost}. 

\begin{proof}[Proof of Corollary \ref{C:RN_welding}]
Recall the definition \eqref{E:Omega} of $\Omega(h_2,h_1)$. 
Let $\mu\in L^\infty(\D^*)$ compactly supported with $\norm{\mu}_{L^\infty}<1$, and for each $t\in[0,1]$, let $\tilde{\Phi}_t$ be the normalised solution to the Beltrami equation with coefficient $t\mu+\bar t\iota^*\mu$. For all test functions $F$, and all $t\in[0,1]$, we have by the Leibniz rule
\begin{equation}\label{eq:diff_t}
\del_t\int F(\tilde{\Phi}_t\circ h)e^{\Omega(\tilde\Phi_t\circ h,h)}\d\tilde{\nu}_\kappa(h)=2\Re \Big(\int \big(\sL_\mu F(\tilde{\Phi}_t\circ h)+F(\tilde{\Phi}_t\circ h) \del_t \Omega(\tilde\Phi_t\circ h,h)\big)e^{\Omega(\tilde\Phi_t\circ h,h)}\d\tilde{\nu}_\kappa(h) \Big).
\end{equation}
By \eqref{E:TT06}, $\del_t \Omega(\tilde\Phi_t\circ h,h) = \frac{c_\rL}{12}\tilde{\vartheta}_{\tilde{\Phi}_t\circ h}(\mu)+\tilde{\varpi}_{\tilde\Phi_t\circ h}(\mu)$ and by Theorem \ref{T:ghost},
\begin{align*}
\int\sL_\mu F(\tilde\Phi_t\circ h)e^{\Omega(\tilde\Phi_t\circ h,h)}\d\tilde{\nu}_\kappa(h)
&=\int\left(\frac{c_\rL}{12}\tilde{\vartheta}_h(\mu)+\tilde{\varpi}_h(\mu)-\sL_\mu(\Omega(\tilde\Phi_t\circ h,h))\right)F(\tilde\Phi_t\circ h)e^{\Omega(\tilde\Phi_t\circ h,h)}\d\tilde{\nu}_\kappa(h)
\end{align*}
where we emphasise that, on the right hand side, $\cL_\mu$ acts on both coordinates of $\Omega$. In particular, by \eqref{E:TT06},
\[ 
\sL_\mu(\Omega(\tilde\Phi_t\circ h,h)) = \frac{c_\rL}{12}\tilde{\vartheta}_h(\mu)+\tilde{\varpi}_h(\mu) + \frac{c_\rL}{12}\tilde{\vartheta}_{\tilde{\Phi}_t\circ h}(\mu)+\tilde{\varpi}_{\tilde\Phi_t\circ h}(\mu).
\]
Putting things together, this shows that the $\del_t$-derivative in \eqref{eq:diff_t} vanishes.
Integrating \eqref{eq:diff_t} between 0 and $1$ then yields for all $\tilde{\Phi}\in\mathrm{Diff}^\omega(\S^1)$:
\[\int F(\tilde\Phi\circ h)e^{\Omega(\tilde{\Phi}\circ h,h)}\d\tilde{\nu}_\kappa(h)=\int F(h)\d\tilde{\nu}_\kappa(h),\]
which is what we wanted to prove.
\end{proof}

\subsubsection{Link with the $bc$-system and Faddeev--Popov ghosts}\label{par:bc}
To end this section, we draw an analogy between the differential of the welding map and the ``$bc$-ghost system" from bosonic string theory, following the account given in \cite[Section 5.3.3]{dFMS_BigYellowBook}. This system was introduced in Polyakov's seminal paper \cite{Polyakov81} and were always thought to encapsulate an infinite dimensional change of variable (and its Jacobian).

The $bc$-ghost system has two ``fields" $b(\zeta)$ and $c(z)$ whose two-point correlation function is the kernel of the Cauchy transform $\langle b(\zeta)c(z)\rangle=\frac{1}{\pi}\frac{1}{\zeta-z}$ \cite[(5.108)]{dFMS_BigYellowBook}. The field $c$ has conformal weight $-1$ \cite[(5.115)]{dFMS_BigYellowBook}, while the field $b$ has conformal weight $2$ \cite[(5.116)]{dFMS_BigYellowBook}. In other words, $c(z)$ transforms as a vector field (or $(-1,0)$-tensor) while $b(\zeta)$ transforms as a quadratic differential (or $(2,0)$-tensor), consistently with the kernel $K_\psi$ considered in Lemma \ref{lem:expand_K}. Indeed, the Cauchy transform sends Beltrami differentials (or $(-1,1)$-tensors) to vector fields, so its kernel is a vector field in the first variable, and a quadratic differential in the second one.

The $bc$-system has central charge $-26$ \cite[(5.117)]{dFMS_BigYellowBook} and the stress-energy tensor of the theory is $\pi:\!2b\,\del c+\del b\, c\!:$ \cite[(5.114)]{dFMS_BigYellowBook}, in striking analogy with $(2\del_z+\del_\zeta)K_\psi(z,\zeta)|_{z=\zeta}=-\frac{13}{6}\cS\psi(z)$ of Lemma \ref{lem:expand_K}. In \eqref{eq:end_ghost}, this holomorphic quadratic differential computes the ``divergence" of vector fields on $\mathrm{Homeo}(\S^1)$ with respect to the SLE loop measure. In other words, it plays the role of the differential of the (would-be) Jacobian determinant of the conformal welding map. Hence, we have a perfect interpretation of the stress-energy tensor of the $bc$-system as the differential of the Jacobian of conformal welding.

In forthcoming works, we will extend this computation to non simply connected surfaces. This adds a finite dimensional problem on top of the infinite dimensional studied in this paper, which consists in computing the law of the modulus of the surface. Combining the techniques developed here with results from \cite{McIntyre2004HolomorphicFO}, we will relate this law to the Belavin--Knizhnik measure from bosonic string theory \cite{BelavinKnizhnik}. In particular, this law features the ``Faddeev--Popov" ghost, which is the regularised determinant of the Laplace operator acting on $(1,0)$-vector fields.

\section{SLE/GFF coupling: proof of Theorem \ref{T:zipper}}\label{S:welding}
We recall the hypotheses of Theorem \ref{T:zipper}: $\eta$ is sampled from the weighted SLE shape measure $e^{2\bK}\nu^\#_\kappa$ with welding maps $f$ and $g$. Conditionally on $\eta$, $\varphi+a$ is sampled independently from $e^{-2Qa}\d a\otimes\P_\eta$. We also have an auxiliary parameter $b$ sampled independently from Lebesgue measure on $\R$. We then define the fields 
\[\phi+c:=\varphi\cdot f+a+b\qquad\text{and}\qquad\phi^*+c^*:=\varphi\cdot g+a-b,
\]
and denote their (infinite) law by $\Q$. Our goal is to prove that $\Q$ is the law of independent Liouville fields in $\D$ and $\D^*$ respectively. The extra parameter $b$ allows us to decouple the two fields, and can be understood as a choice of embedding of the curve (i.e. rescale the curve by $e^b$).

\begin{lemma}\label{L:zero_modes}
There exist $C_\diamond>0$ and a probability measure $\Q^\#$ on $(\C^{\N^*})^2$ such that $\Q=C_\diamond^2 e^{-Qc}\d c\otimes e^{-Qc^*}\d c^*\otimes\Q^\#$.
\end{lemma}

As usual, given the coefficients $( (\phi_m)_{m \ge 1}, (\phi^*_m)_{m \ge 1}) \in (\C^{\N^*})^2$, the associated pair of fields on $\S^1$ will eventually correspond to $(2\Re(\sum_{m=1}^\infty \phi_m z^m), 2\Re(\sum_{m=1}^\infty \phi^*_m z^m))$.

\begin{proof}
For all test functions $F$ and $x,y\in\R$, the change of variable $(a,b)\mapsto(a+\frac{x+y}{2},b+\frac{x-y}{2})$ gives
\begin{align*}
&\int F(\phi+c+x,\phi^*+c^*+y)\d\Q(\phi+c,\phi^*+c^*)\\
&=\int F(\varphi\cdot f+a+b +x,\varphi\cdot g+a-b+y)e^{-2Qa}\d a\,\d\P_\eta(\varphi)\d b\,\d\nu^\#(\eta)\\
&=\int F(\varphi\cdot f+a+b,\varphi\cdot g+a-b)e^{-2Q(a-\frac{x+y}{2})}\d a\,\d\P_\eta(\varphi)\d b\,\d\nu^\#(\eta)\\
&=e^{Q(x+y)}\int F(\phi+c,\phi^*+c^*)\d\Q(\phi+c,\phi^*+c^*).
\end{align*}
This proves the existence of the factorisation for some measure $\Q^\#$ (two $\sigma$-finite measures which agree on a $\pi$-system that generates the $\sigma$-algebra are in fact equal).
Moreover, it is clear that the event that $(c,c^*)\in A\times B$ has finite $\Q$-measure for all compact sets $A,B\subset\R$, which proves that $\Q^\#$ has finite total mass. It can thus be renormalised by a global multiplicative constant to get a probability measure.
\end{proof}

Theorem \ref{T:zipper} is then equivalent to the next statement, which will be proved using the characterisation of Theorem \ref{T:Neumann}.
\begin{proposition}
For $\Q^\#$-almost every $((\phi_m)_{m\geq1}, (\phi^*_m)_{m\geq1})$, the series $z\mapsto2\Re(\sum_{m=1}^\infty\phi_mz^m)$ and $z\mapsto2\Re(\sum_{m=1}^\infty\phi^*_mz^m)$ converge in $\dot H^{-s}(\S^1)$ for all $s>0$ and $\Q^\#=\dot\P^{\otimes2}$ as Borel probability measures on $\dot H^{-s}(\S^1)^2$.
\end{proposition}

\begin{proof}
In order to apply Theorem \ref{T:Neumann}, we need to generate some small deformations of the fields $(\phi,\phi^*)$ using analytic diffeomorphisms of $\S^1$, and check how the measure reacts to these deformations. We will denote by $\cD_n^{(1)}$ (resp. $\cD_n^{(2)}$) the operators \eqref{E:def_cDn} of Section \ref{S:neumann_gff} acting on the first (resp. second variable). We also generalise the notation by writing $\cD_\mu$ for the operator obtained when generating a diffeomorphism of the unit circle by a Beltrami differential $\mu$. We will first compute the adjoint of $\cD_n^{(j)}$ ($j=1,2$) on $L^2(\Q)$ and ``Laplace transform" it to find the adjoint of $\cD_{n,\alpha}^{(j)}$ on $L^2(\Q^\#)$.

Let $\mu\in L^\infty(\D)$ compactly supported, and let $\tilde{\Phi}_t$ (resp. $\Phi_t$) be the solution to the Beltrami equation with coefficient $t\mu+\bar t\iota^*\mu$ (resp. $f_*\mu$), normalised to fix 0, 1, $\infty$ (resp. such that $\Phi_t(0)=0$, $\Phi_t(\infty)=\infty$, $\Phi_t'(\infty)=1$). Then, $g_t:=\Phi_t\circ g$ (resp. $f_t:=\Phi_t\circ f\circ\tilde{\Phi}_t^{-1}$) is a Riemann mapping for $D_t^*:=\mathrm{ext}(\Phi_t(\eta))$ and $D_t:=\mathrm{int}(\Phi_t(\eta))$. Indeed, by the chain rule, both maps are holomorphic in their respective domains and are normalised appropriately. Note that we do not have to normalise $g_t$ with $g_t'(\infty)\in\R_+$, since we only care about the distribution of the map $g_t\circ e^{i\alpha}$ (with $e^{i\alpha}$ Haar-distributed in $\S^1$).

We are now going to study the variation $((\phi+c)\cdot\tilde{\Phi}_t^{-1},\phi^*+c^*)$. Let $F$ be a test function. By definition,
\[F\left((\phi+c)\cdot\tilde{\Phi}_t^{-1},\phi^*+c^*\right)=F(\phi+c,\phi^*+c^*)+t\cD_\mu^{(1)}(F)+\bar t\cD_{\iota^*\mu}^{(1)}(F)+o(t).\]
On the other hand, using $g=\Phi_t^{-1}\circ g_t$ and $f\circ\tilde{\Phi}_t^{-1}=\Phi_t^{-1}\circ f_t$, we have 
\[F\left((\phi+c)\cdot\tilde{\Phi}_t^{-1},\phi^*+c^*\right)=F\left((\varphi\cdot\Phi_t^{-1})\cdot f_t+a+b,(\varphi\cdot\Phi_t^{-1})\cdot g_t+a-b\right).\]
Moreover, conditionally on $\eta$, if $\varphi+a$ is a sample from $e^{-2Qa}\d a\otimes\P_\eta$, the field $a+\varphi\cdot\Phi_t^{-1}$ is a field on $\Phi_t(\eta)$ which is absolutely continuous with respect to $e^{-2Qa}\d a\otimes\P_{\Phi_t(\eta)}$. In fact, the law of $a+\varphi\cdot\Phi_t^{-1}$ is just $e^{-2Qa}\d a\otimes\P_{\Phi_t(\eta),\Phi_t^{-1}}$ by definition \eqref{E:def_P_eta_Phi}. Hence,
\begin{equation}\label{eq:var_Q}
\begin{aligned}
&\int F\left((\phi+c)\cdot\tilde{\Phi}_t^{-1},\phi^*+c^*\right)\d\Q(\phi+c,\phi^*+c^*)\\
&\qquad= \int F(\varphi\cdot f_t+a+b,\varphi\cdot(g_t\circ e^{i\alpha})+a-b)\d\P_{\Phi_t(\eta),\Phi_t^{-1}}(\varphi)e^{2\bK}\d\nu^\#(\eta)\frac{\d\alpha}{2\pi}e^{-2Qa}\d a\,\d b\\
&\qquad=\int G_t(\Phi_t(\eta),b)e^{2\bK}\d\nu^\#(\eta)\frac{\d\alpha}{2\pi}\d b,
\end{aligned}
\end{equation}
where we have defined $G_t(\eta,\alpha,b):=\int F(\varphi\cdot f+a+b,\varphi\cdot(g\circ e^{i\alpha})+a-b)\d\P_{\eta,\Phi_t^{-1}}(\varphi)e^{-2Qa}\d a$.
By the chain rule, \eqref{eq:var_Q} is further equal to
\begin{equation}\label{E:chain_rule_leibniz}
    \int G_t(\eta,\alpha,b)e^{2\bK}\d\nu^\#(\eta)\frac{\d\alpha}{2\pi}\d b + 2 \Re \Big( t \int\cL_{f_*\mu}G(\eta,\alpha,b) e^{2\bK}\d\nu^\#(\eta)\frac{\d\alpha}{2\pi}\d b \Big) + o(t).
\end{equation}
From Proposition \ref{P:var_liouville}, we have
\begin{align*}
G_t(\eta,\alpha,b)
&=G(\eta,\alpha,b)+2\Re\left(t\int\left(\frac{1}{12}(\cS f^{-1},f_*\mu)+(\bT_D \varphi,f_*\mu)\right)F\d\P_\eta(\varphi)e^{-2Qa}\d a\right)+o(t).
\end{align*}
By \eqref{E:schwarzian_mu}, $(\cS f^{-1},f_*\mu) = -(\cS f,\mu)$ and by \eqref{E:chain_rule_stress_mu} we have $(\bT_D \varphi,f_*\mu) = (\bT_\D \phi,\mu) - \frac{Q^2}{2} (\cS f,\mu)$ (recall $\phi=\varphi\cdot f+a+b$), so
\begin{align*}
G_t(\eta,\alpha,b)
&=G(\eta,\alpha,b)+2\Re\left(t\int\left((\bT_\D \phi,\mu)-\frac{c_\rL}{12}(\cS f,\mu)\right)F\d\P_\eta(\varphi)e^{-2Qa}\d a\right)+o(t)\\
&=G(\eta,\alpha,b)-\frac{c_\rL}{6}\Re\left(t\left(\cS f,\mu\right)\right)G(\eta,\alpha,b)+2\Re\left(t\int\left(\bT_\D \phi,\mu\right)F\d\P_\eta(\varphi)e^{-2Qa}\d a\right)+o(t).
\end{align*}
Now, since the welding homeomorphism of $\Phi_t(\eta)$ is $g_t^{-1}\circ f_t=g^{-1}\circ f\circ\tilde{\Phi}_t^{-1}$, by Theorem \ref{T:ghost}, we have
\begin{align*}
\int\cL_{f_*\mu}(G)e^{2\bK}\d\nu^\#\frac{\d\alpha}{2\pi}\d b=\int\sR_\mu(G)e^{2\bK}\d\tilde{\nu}\d b=\frac{c_\rL}{12}\int\left(\cS f,\mu\right)Ge^{2\bK}\d\tilde{\nu}\d b
\end{align*}
Summing up the two contributions in \eqref{E:chain_rule_leibniz}, we see that the Schwarzian derivatives compensate each other. Together with dominated convergence theorem, this yields:
\[\int F((\phi+c)\cdot\tilde{\Phi}_t^{-1},\phi^*+c^*)\d\Q(\phi + c,\phi^*+c^*)=2\Re\Big( t\int (\bT_\D \phi,\mu)F(\phi + c,\phi^*+c^*)\d\Q(\phi + c,\phi^*+c^*)\Big)+o(t),\] 
Specialising to the canonical basis, and by taking the linear and antilinear part in $t$ (as in the proof of Proposition \ref{P:adjoints}), we get for all $n>0$,
\[(\cD_n^{(1)})^*=\cD_{-n}^{(1)}-(\bT_\D \phi,\mu)\text{ on }L^2(\Q).\]

Finally, we Laplace transform in the $c$-variable, i.e. we consider the operator $\cD_{n,\alpha}^{(1)}=e^{-\frac\alpha2c}\circ\cD_n^{(1)}\circ e^{\frac\alpha2c}$ acting on $L^2(\Q^\#)$. Due to the factorisation of Lemma \ref{L:zero_modes}, a straightforward calculation as in the proof of Lemma \ref{L:adjoint_D_alpha} gives
\[(\cD_{n,2Q-\bar\alpha}^{(1)})^*=\cD_{-n,\alpha}^{(1)}-\left(\bT_\D \phi,\mu\right)\text{ on }L^2(\Q^\#).\]
We get similarly that $(\cD_{n,2Q-\bar\alpha}^{(2)})^*=\cD_{-n,\alpha}^{(2)}-(\bT_{\phi^*}^\D,\mu)$ on $L^2(\Q^\#)$, and an application of Theorem~\ref{T:Neumann} ends the proof.
\end{proof}

\appendix

\section{Brownian local time of SLE}

In this section, we construct and study the Brownian local time of an SLE loop. This technical result is used in Section \ref{SS:jump} in order to show that the jump operator $\bD_\eta$ and its Dirichlet form $(\cE_\eta,\cF_\eta)$ have nice properties. We will work in the chordal case, which is sufficient by local absolute continuity. Specifically, consider a chordal SLE $\eta$ from 0 to $\infty$ in $\H$ with law $\P^{\rm SLE}$ and an independent Brownian motion $(B_t)_{t \ge 0}$ starting at $x$ with law $\P_x^{\rm BM}$ for some $x \in \C$. Let $\alpha = 1 + \kappa/8$ be the fractal exponent of $\eta$ and
\[ 
I_{T,\eps} := \eps^{-2+\alpha} \int_0^T \indic{\{\dist(B_t,\eta)<\eps\}} \d t, \qquad T \ge 0,\quad \eps >0.
\]

\begin{lemma}\label{L:local_time}
For all $T>0$ and starting points $x \in \H$,
$(I_{T,\eps})_{\eps >0}$ converges in $L^1(\P^{\rm SLE} \otimes \P_x^{\rm BM})$ to some random variable $I_{T}$: the ``local time'' accumulated by $(B_t)_{0 \le t\le T}$ on $\eta$. Moreover, for all $T>0$ and $\P^{\rm SLE}$-almost every $\eta$, for every starting point $x \in \eta$, $\P_x^{\rm BM}(I_T>0)=1$.
\end{lemma}

As already mentioned, the article \cite{LawlerRezaei15_Minkowski} constructs the Minkowski content of $\eta$, giving a notion of ``size'' for $\eta$. Lemma~\ref{L:local_time} provides a random notion of size of $\eta$, based on Brownian motion. These two results are closely related and, as we are about to see, Lemma \ref{L:local_time} is a quick consequence of estimates derived in \cite{LawlerRezaei15_Minkowski}. In particular, we recall from \cite{LawlerRezaei15_Minkowski} the existence of ``one-point and two-point Green's functions'' for SLE:
\begin{gather}
\label{E:GreenSLE1}
G^{\rm SLE}(z) = \lim_{\eps \to 0} \eps^{-2+\alpha} \P^{\rm SLE}(\dist(z,\eta)<\eps), \quad z \in \H; \\
\label{E:GreenSLE2}
G^{\rm SLE}(z,w) = \lim_{\eps,\delta \to 0} \eps^{-2+\alpha} \delta^{-2+\alpha} \P^{\rm SLE}(\dist(z,\eta)<\eps, \dist(w,\eta)<\delta), \quad z,w \in \H.
\end{gather}
Moreover, there exists a constant $c = c(\kappa)$ such that
\begin{equation}
\label{E:GreenSLE3}
G^{\rm SLE}(z) = c ~ \Im(z)^{\alpha-2} (\sin \arg(z))^{8/\kappa-1}, \quad z \in \H.
\end{equation}

\begin{proof}[Proof of Lemma \ref{L:local_time}]
Let $T>0$ and $x \in \C$.
We start by introducing an intermediate approximation of $I_{T,\eps}$. For $R \ge 1$, let $\H_R = \{ z \in \C: \Im(z) \ge 1/R, |z| \le R\}$ and
\[ 
I_{T,R,\eps} := \eps^{-2+\alpha} \int_0^T \indic{\{\dist(B_t,\eta)<\eps, B_t \in \H_R\}} \d t, \quad T \ge 0, R \ge 1, \eps >0.
\]
This cutoff in $R$ will give us some room in the two-point computation below.
By the triangle inequality, for any $R \ge 1$,
\begin{align}
\label{E:flight4}
\limsup_{\eps,\delta \to 0} \E^{\rm SLE} \otimes\E_x^{\rm BM}[|I_{T,\eps}-I_{T,\delta}|] & \le \limsup_{\eps, \delta \to 0}
\E^{\rm SLE} \otimes\E_x^{\rm BM}[|I_{T,R,\eps}-I_{T,R,\delta}|] \\
\notag
& + 2 \limsup_{\eps \to 0} \E^{\rm SLE} \otimes\E_x^{\rm BM}[|I_{T,\eps}-I_{T,R,\eps}|].
\end{align}
Let us start by showing that, $R \ge 1$ being fixed, the first term in the RHS vanishes. By Cauchy--Schwarz, it is enough to show that $(I_{T,R,\eps})_{\eps >0}$ is Cauchy in $L^2(\P^{\rm SLE}\otimes\P^{\rm BM})$. To this end, we will establish the following identity:
\begin{equation}
\label{E:flight2}
\lim_{\eps, \delta \to 0} \E^{\rm SLE} \otimes\E_x^{\rm BM}[I_{T,R,\eps} I_{T,R,\delta} ] = 2 \int_0^T \hspace{-3pt} \frac{\d s}{2\pi s} \int_s^T \hspace{-6pt} \frac{\d t}{2\pi(t-s)} \hspace{-2pt}\int_{\H_R} \hspace{-6pt} |\d z|^2 e^{-\frac{|z-x|^2}{2s}} \hspace{-2pt} \int_{\H_R} \hspace{-6pt} |\d w|^2 e^{-\frac{|w-z|^2}{2(t-s)}} G^{\rm SLE}(z,w).
\end{equation}
Using the explicit law of $(B_s,B_t)$ for $0 < s < t < T$, we have
\begin{align}
\label{E:flight1}
\E^{\rm SLE}\otimes\E_x^{\rm BM}[I_{T,R,\eps}I_{T,R,\delta}] & =
2 \int_0^T \frac{\d s}{2\pi s} \int_s^T \frac{\d t}{2\pi(t-s)} \int_{\H_R} |\d z|^2 e^{-\frac{|z-x|^2}{2s}} \int_{\H_R} |\d w|^2 e^{-\frac{|w-z|^2}{2(t-s)}} \times \\
\notag
& \hspace{40pt} \times \eps^{-2+\alpha}\delta^{-2+\alpha} \P^{\rm SLE}(\dist(z,\eta)<\eps, \dist(w,\eta)<\delta).
\end{align}
By \cite[Lemma 2.9]{LawlerRezaei15_Minkowski} (see also Equation (19) therein), there exists $c=c(R)>0$ such that for all $z, w \in \H_R$ and $\eps,\delta \in (0,1)$,
\[
\eps^{-2+\alpha} \delta^{-2+\alpha} \P^{\rm SLE}(\dist(z,\eta)<\eps, \dist(w,\eta)<\delta) \le c |z-w|^{\alpha-2}.
\]
It is then an elementary computation to check that, if one injects the following upper bound in \eqref{E:flight1}, the resulting integral is finite.
Together with \eqref{E:GreenSLE2}, we can thus apply dominated convergence theorem to get \eqref{E:flight2}.

Going back to \eqref{E:flight4}, we have shown that, for all $R \ge 1$,
\begin{align*}
\limsup_{\eps,\delta \to 0} \E^{\rm SLE} \otimes\E_x^{\rm BM}[|I_{T,\eps}-I_{T,\delta}|] & \le 2 \limsup_{\eps \to 0} \E^{\rm SLE} \otimes\E_x^{\rm BM}[|I_{T,\eps}-I_{T,R,\eps}|].
\end{align*}
A similar argument based on dominated convergence theorem (see \cite[Theorem 2.4]{LawlerRezaei15_Minkowski} for a bound on the probability that $\eta$ gets $\eps$-close to a given point) shows that the right hand side term of the above display equals
\[
2 \int_0^T \frac{\d t}{2\pi t} \int_{\H \setminus \H_R} |\d z|^2 e^{-\frac{|z-x|^2}{2t}} G^{\rm SLE}(z).
\]
Using the explicit expression \eqref{E:GreenSLE3} of $G^{\rm SLE}(z)$ and because $\bigcup_{R \ge 1} \H_R = \H$, 
we see that the above integral vanishes as $R \to \infty$. This proves that $(I_{T,\eps})_{\eps>0}$ is Cauchy in $L^1(\P^{\rm SLE} \otimes \P^{\rm BM})$ as desired.

It remains to check that, for $\P^{\rm SLE}$-almost every $\eta$ and for every $x \in \eta$, $\P_x^{\rm BM}(I_T>0)=1$.
Since
\[
\E^{\mathrm{SLE}}\otimes\E_x^{\mathrm{BM}}[I_T] = \int_0^T \frac{\d t}{2\pi t} \int_{\H} |\d z|^2 e^{-\frac{|z-x|^2}{2t}} G^{\rm SLE}(z)
\]
is positive, the probability that $I_T>0$ is positive.
We start by considering the case $x=0$ for which a scaling covariance property holds.
For all $\lambda >0$, the joint law of $(\lambda \eta,(\lambda B_{t/\lambda^2})_{t \ge 0})$ is still given by $\P^{\rm SLE} \otimes \P_0^{\rm BM}$. This scaling property implies that for all $T>0$, $I_T\laweq T^{\alpha/2}I_1$ under the product measure $\P^{\mathrm{SLE}}\otimes\P_0^{\mathrm{BM}}$. In particular, $\P^{\mathrm{SLE}}\otimes\P_0^{\mathrm{BM}}(I_T>0)$ does not depend on $T$, and
\[
\P^{\mathrm{SLE}}\otimes\P_0^{\mathrm{BM}}(\forall T \in \Q \cap (0,1), I_T > 0 ) = \lim_{\substack{T \to 0\\T \in \Q \cap (0,1)}} \P^{\mathrm{SLE}}\otimes\P_0^{\mathrm{BM}}(I_T >0) = \P^{\mathrm{SLE}}\otimes\P_0^{\mathrm{BM}}(I_1>0) >0.
\]
The event $\{\forall T \in \Q \cap (0,1), I_T > 0\}$ being a tail event for $(B_t)_{t \ge 0}$ and the Brownian motion driving the Loewner equation, we infer that its probability equals 1.
Applying the Markov property of SLE, we deduce that for $\P^{\rm SLE}$-almost every $\eta$,
$\P_x^{\rm BM}(I_T>0)=1$ holds for a countable dense subset of starting points $x \in \eta$ (corresponding e.g. to rational times for the capacity parametrisation of SLE). To conclude that the same statement holds simultaneously for all $x \in \eta$, it is enough to check that, conditionally on $\eta$,
\begin{equation}
    \label{E:LA13}
\H \ni x \mapsto \P_x^{\rm BM}(I_T>0)
\quad \text{is continuous.}
\end{equation}

Let $x \in \H$ be some starting point and $\delta>0$. 
Let $\eps>0$ and $y \in \H$ be at distance at most $\eps$ from $x$. For $A>0$, the laws of $(B_t)_{t \ge A\eps^2}$ under $\P_x^{\rm BM}$ and under $\P_y^{\rm BM}$ are mutually absolutely continuous with Radon--Nikodym derivative equal to $1+o(1)$ where $o(1) \to 0$ as $A \to \infty$, uniformly in $\eps$.
Fixing $A>0$ large enough so that this Radon--Nikodym derivative is at least $1-\delta/2$, we have for all $\eps >0$,
\[
\P_y^{\rm BM}(I_T>0) \ge \P_y^{\rm BM}(I_T-I_{A\eps^2}>0) \ge (1-\delta/2) \P_x^{\rm BM}(I_T-I_{A\eps^2}>0).
\]
Now, $\P_x^{\rm BM}$-almost surely $I_{A\eps^2} \to 0$ as $\eps\to 0$. This shows that for all $y$ close enough to $x$ we have
\[
\P_y^{\rm BM}(I_T>0) \ge (1-\delta) \P_x^{\rm BM}(I_T>0).
\]
This is the desired continuity statement \eqref{E:LA13} which concludes the proof of the lemma.
\end{proof}

\section{Comparison with other conformal welding results}\label{Appendix:comparison}
This appendix grew out of what was initially planned to be a short paragraph on the comparison of the results of Section \ref{ss:zipper} with the conformal welding of quantum discs. Since quantum surfaces are a delicate notion, we include a review, which we aim both concise and accessible. It is based on the recent comparison of Liouville CFT with quantum surfaces via the so-called uniform embedding \cite{AHS_SLE}, and subsequent conformal welding results \cite{AHS20,AHS23_loop,AngCaiSunWu_1}.

    \subsection{Quantum surfaces}
Let $D\subset\hat\C$ be a domain, $\bz=z_1,...,z_N\in D$ distinct marked points, and $\mathrm{Aut}(D,\mathbf{z})$ be the group of conformal automorphisms of $D$ fixing $\bz$ pointwise. One can also add marked points on the boundary, but we will ignore this for the purpose of this review. We have the usual action of $\mathrm{Aut}(D,\bz)$ on the space of distributions $\cD'(D)$ by $X\cdot\Phi=X\circ\Phi+Q\log|\Phi'|$. We denote the equivalence relation by $\sim_\gamma$.

We introduce the following definitions based on the typical LQG/LCFT literature.
\begin{definition}

\begin{itemize}
A \emph{Liouville field} in $D$ with $\alpha_j$-insertions at $z_j$ is a Borel measure $\Q$ on $\cD'(D)$ such that for all $\Phi\in\mathrm{Aut}(D,\bz)$, 
\begin{equation}\label{eq:weyl}
\Phi^*\Q=e^{\frac{c_\rL}{12}\bS_D(\log|\Phi'|)}\prod_{j=1}^N|\Phi'(z_j)|^{2\Delta_j}\Q,
\end{equation}
where $\bS_D(\omega)=\frac{i}{\pi}\int_D\del\omega\wedge\delbar\omega+Q\int_{\del D}\omega k_{|\d z|^2}$ is the Liouville action in $D$ expressed in the Euclidean metric ($k_{|\d z|^2}$ is the geodesic curvature of the boundary), and $\Delta_j=\frac{\alpha_j}{2}(Q-\frac{\alpha_j}{2})$. The condition \eqref{eq:weyl} is the \emph{Weyl anomaly}.
\item\textit{(Tentative)}: A \emph{quantum surface} (embedded in $D$, with marked points $\bz$) is a Borel measure on $\cD'(D)/_{\sim\gamma}\mathrm{Aut}(D,\bz)$.\footnote{Some sources call quantum surface an equivalence class in $\cD'(D)/_{\sim_\gamma}\mathrm{Aut}(D,\bz)$, rather than a measure on this space. We will adopt the other convention in order to avoid the terminology ``random quantum surface". We thank Xin Sun for pointing this out to us.}
\end{itemize}
\end{definition}
The two notions look equivalent, since one could in principle pass any Liouville field to the quotient, and conversely lift any quantum surface to a Liouville field. However, the Liouville carries extra information via the explicit form of the conformal covariance. Moreover, since working with quotients ranks among the most perilous mathematical activities, we will insist on distinguishing the two notions. 

The Liouville field is a perfectly valid notion (leaving aside the question of its construction), but we need to be extra careful when it comes to the meaning attached to a quantum surface. Indeed, in many cases, $\mathrm{Aut}(D)$ is not compact, and the naive process of integrating along the fibres of a Liouville field yields a trivial measure (Haar measure being infinite). The usual strategy is to identify the quotient $\cD'(D)/_{\sim\gamma}\mathrm{Aut}(D)$ with a \emph{global section}.\footnote{i.e. a (smooth) map $\sigma:\cD'(D)/\mathrm{Aut}(D,\bz)\to\cD'(D)$ such that $\pi\circ\sigma=\mathrm{id}$, with $\pi$ the canonical projection. A choice of section identifies the quotient with its image $\sigma(\cD'(D)/_{\sim\gamma}\mathrm{Aut}(D,\bz))\subset\cD'(D)$.} Then, a measure on a section (viewed as a subset of $\cD'(D)$) extends to a Liouville field by imposing the Weyl anomaly \eqref{eq:weyl}. Conversely, given a Liouville field, we can (try to) condition it on this section. There is not always a preferred choice of section, and the process introduces some arbitrariness, especially if $\mathrm{Aut}(D)$ is large (e.g. $D=\hat\C$). To sum up, the only canonical description of a quantum surface is its lift to a Liouville field. That being said, and as we shall see below, there are some quantum surfaces admitting particularly nice embeddings. As a historical note, quantum surfaces come from the approach due to Duplantier--Miller--Sheffield \cite{DuplantierSheffield11,MatingOfTrees}, while Liouville fields were considered by David--Kupiainen--Rhodes--Vargas \cite{DKRV16}. The existence of Gaussian fields satisfying the Weyl anomaly is the consequence of two properties of the Liouville action: it is quadratic and satisfies a cocycle property.

As an example, consider the punctured plane $\C^\times:=\C\setminus\{0\}$ (viewed as the sphere $\hat\C$ with two marked points $(0,\infty)$) with $\mathrm{Aut}(\C^\times)\simeq(\C^\times,\times)$. One way to fix the scaling parameter is to impose $\mathrm{argmax}_{t\in\R}\,\lbrace Y_t:=\int_0^{2\pi}X(e^{-t+i\theta})\d\theta=0\rbrace$. We will only consider rotationally invariant measures for which the above condition makes sense with full mass. The precise form of these measures is not so relevant here, but they are locally absolutely continuous with respect to the Gaussian free field, and $Y_t$ (the ``radial process") is typically a (conditioned) two-sided drifted Brownian motion such that $\lim_{t\to\pm\infty}\,Y_t=-\infty$ with full measure (see \cite[Proposition 2.14]{AHS_SLE} for an example of the type of process considered). The \emph{two-pointed quantum sphere} $\mathrm{QS}_2$ is initially defined on the section just above (hence it induces a quantum surface), but the important thing for us is that its lift to a Liouville field is (up to a multiplicative constant) $\d\mathrm{LF}_\C^{(\gamma,0),(\gamma,\infty)}(c+X)=e^{2(\gamma-Q)c}\d c\,e^{\gamma X(0)}e^{\gamma X(\infty)}\d\P_\C(X)$ \cite[(1.4)]{AHS_SLE}. Here, $\P_\C$ is the GFF in a certain metric, and the exponential $e^{\gamma X(0)}$ is only well defined after Wick renormalisation (and the use of Cameron--Martin's theorem). See \cite{AHS_SLE} for precise definitions. There are variants of $\mathrm{QS}_2$ for generic $\alpha$-insertions, and its partition function is the reflection coefficient of LCFT \cite[Section 1.2]{KRV_DOZZ}.

Now, we want to leverage the existence of $\mathrm{QS}_2$ to get quantum surfaces embedded in $\hat\C$ (no marked points). Again, there are two possibilities: the first is to find a map directly on the quotient $\cD'(\C^\times)/_{\sim\gamma}\C^\times\to\cD'(\hat\C)/_{\sim\gamma}\mathrm{Aut}(\hat\C)$, and pushforward $\mathrm{QS}_2$. The second (and easier) way is to consider the lift of $\mathrm{QS}_2$ to $\mathrm{LF}_\C^{(\gamma,0),(\gamma,\infty)}$, and map it to a M\"obius-covariant field on the sphere. In both cases, the invariance property must be proved in some way.

An instance of such operation is \emph{forgetting the marked points}. Using the properties of Haar measure on the M\"obius group and a suitable reweighting by the total mass of GMC, it is shown in \cite[Theorem 1.1]{AHS_SLE} that one can integrate out the position of the two marked points of $\mathrm{QS}_2$ in a way that preserves the invariance under LQG coordinate change. The resulting quantum surface is denoted $\mathrm{QS}$, and the corresponding Liouville field is $\d\mathrm{LF}_\C(c+X)=e^{-2Qc}\d c\,\d\P_\C(X)$, where $\P_\C$ is the GFF in a certain metric (see \cite{AHS_SLE} for the explicit expression). Contrary to the punctured plane, there is no privileged choice of section $\cD'(\hat\C)/\mathrm{Aut}(\hat\C)\to\cD'(\C)$, and we are not aware of a nice embedding that would represent the quotient measure. In LQG literature, the terminology \emph{quantum sphere} $\mathrm{QS}$ is used as an umbrella term to describe the would-be quotient measure, without reference to a particular choice of embedding.

There is a similar process defining the \emph{quantum disc} $\mathrm{QD}$ by forgetting the (boundary) marked points of the two-pointed quantum disc $\mathrm{QD}_{0,2}$. Again, while $\mathrm{QD}_{0,2}$ can be described by a particular embedding, there is no nice embedding of $\mathrm{QD}$ and the simplest way to understand this measure is as a Liouville field: as a Liouville field $\mathrm{QD}$ is $\mathrm{LF}_\D=e^{-Qc}\d c\otimes\P_\D$, where $\P_\D$ is just the Neumann GFF in $\D$. We know that this field has a $\frac{\gamma}{2}$-GMC measure on the boundary, and we can disintegrate along this random variable, as in the paragraph preceding Corollary \ref{cor:conditioned}: $\mathrm{LF}_\D=\int_0^\infty\mathrm{LF}_\D(\ell)\ell^{-\frac{2}{\gamma}Q}\frac{\d\ell}{\ell}$, where $\mathrm{LF}_\D^\#(\ell)$ is the measure with expectation $\dot\E_\ell$ defined in \eqref{eq:condition_ell}. Note that the trace of $\mathrm{LF}_\D$ on $\S^1$ is nothing but the measure $e^{-Qc}\d c\otimes\dot\P$ considered in this paper (see in particular the statements of Section \ref{ss:zipper}).

    \subsection{Conformal welding of quantum discs}
With this terminology, we can now state the conformal welding results of \cite[Theorem 1.1]{AHS23_loop} ($C$ is a positive constant): 
\begin{equation}\label{eq:welding_ahs}
\mathrm{QS}\otimes\mathrm{SLE}^\mathrm{loop}_\kappa\simeq C\int_0^\infty\mathrm{Weld}(\mathrm{QD}(\ell),\mathrm{QD}(\ell))\ell\d\ell.
\end{equation}
Thanks to the material from the previous section, we have all the tools to unpack this formula. Although \cite{AHS23_loop} writes an equality, we prefer to use $\simeq$ since this should really be understood as an isomorphism of $L^2$-spaces. In the right-hand-side, the measure $\mathrm{QD}(\ell)$ corresponds to the quotient measure of $\ell^{2-\frac{2Q}{\gamma}}\mathrm{LF}_\D(\ell)$, so it coincides with the right-hand side in Corollary \ref{cor:conditioned}. Moreover, the welding operation in \cite{AHS23_loop} is done so as to get a rotationally invariant welding homeomorphism. In fact, the homeomorphism is invariant under the full M\"obius group (both under left and right composition) by the M\"obius-invariance of $\mathrm{LF}_\D$. Hence, the law of this homeomorphism is $e^{2\bK}\tilde{\nu}$. In the left-hand-side of \eqref{eq:welding_ahs}, let us choose the embedding such that the loop disconnects 0 from $\infty$ and has unit conformal radius viewed from 0. Since the welding homeomorphism of the loop comes from the M\"obius-invariant measure from the right-hand-side, the law of the curve is $e^{2\bK}\nu^\#$. Moreover, to do the cutting operation, we need the additional uniform rotation from the left-hand-side in Corollary \eqref{cor:conditioned}. In \cite{AHS23_loop}, this extra randomness takes the form of a sample from the GMC-measure and is also here to restore the rotation invariance. Finally, the law of the M\"obius-invariant field on top of $\eta$ is $\mathrm{LF}_\C$, whose trace on $\eta$ is the measure $e^{-2Qa}\d a\otimes\P_{\eta}$. It follows that the left-hand-side of \eqref{eq:welding_ahs} coincides with the measure $e^{-2Qa}\d a\otimes\P_{\eta}\otimes\frac{\d\alpha}{2\pi}\otimes\nu^\#$ as required. 

\section{A formal computation}\label{appendix:formal}

In this section, we make a formal computation explaining how one could derive Theorem~\ref{T:ghost} from the path integral
\[ 
``\;\d\tilde{\nu}_\kappa(h)=\exp\left(-\frac{c_\rL}{24\pi}\bS_1-2\bK\right)Dh\;",
\]
where $Dh$ is the non-existent Haar measure on Homeo$(\S^1)$ and $\bS_1$ and $\bK$ are defined in \eqref{E:def_K_S1}. 
The fact that this non-rigorous computation yields the correct integration by parts formula gives an \textit{a posteriori} justification to the above path integral.
Let $F, G \in L^2(\tilde\nu_\kappa)$ be test functions and $\mu \in L^\infty(\hat\C)$ be a Beltrami differential compactly supported in $\D^*$. For all $t \in \C$ small, let $\tilde\Phi_t$ be the solution to the Beltrami equation with coefficient $t \mu + \bar t \iota^*\mu$, normalised to fix $0,1,\infty$. Recall that for all welding homeomorphism $h$,
\[ 
F(\tilde\Phi_t \circ h) = F(h) + t\sL_\mu F(h) + \bar t \bar\sL_\mu F(h) + o(t).
\]
We have
\begin{align*}
    \int F(\tilde\Phi_t\circ h) \overline{G(h)} \d\tilde{\nu}_\kappa(h)
    & = \int F(\tilde\Phi_t\circ h) \overline{G(h)} \exp\left(-\frac{c_\rL}{24\pi}\bS_1(h)-2\bK(h)\right)Dh \\
    & = \int F(h) \overline{G(\tilde{\Phi}_t^{-1}\circ h)} \exp\left(-\frac{c_\rL}{24\pi}\bS_1(\tilde{\Phi}_t^{-1}\circ h)-2\bK(\tilde{\Phi}_t^{-1}\circ h)\right)Dh.
\end{align*}
In the last equality, we made a change of variables and used the invariance of the ``Haar measure'' on Homeo$(\S^1)$. To first order, $\tilde\Phi_t^{-1}$ is the normalised solution of the Beltrami equation with coefficient $-t \mu - \bar t \iota^*\mu$. Identifying the $t$-coefficients in the above display, we thus get that
\begin{align*}
    & \int \sL_\mu F(h) \overline{G(h)} \d\tilde{\nu}_\kappa(h) \\
    & = \int F(h) \Big( -\overline{\bar\sL_\mu G(h)} + \overline{G(h)} \cL_\mu \big( \frac{c_\rL}{24\pi} \bS_1 + 2\bK \big)(h) \Big)\exp\left(-\frac{c_\rL}{24\pi}\bS_1(h)-2\bK(h)\right)Dh\\
    & = \int F(h) \Big( -\overline{\bar\sL_\mu G(h)} + \overline{G(h)} \cL_\mu \big( \frac{c_\rL}{24\pi} \bS_1 + 2\bK \big)(h) \Big) \d \tilde{\nu}_\kappa(h).
\end{align*}
As recalled in \eqref{E:TT06}, Takhtajan and Teo \cite{TakhtajanTeo06} proved that
$\sL_\mu\left(\frac{c_\rL}{24\pi}\bS_1+2\bK\right) = \frac{c_\rL}{12}\tilde\vartheta(\mu)+\tilde\varpi(\mu)$. Altogether,
\begin{align*}
    \int \sL_\mu F(h) \overline{G(h)} \d\tilde{\nu}_\kappa(h)
    = \int F(h) \Big( -\overline{\bar\sL_\mu G(h)} + \overline{G(h)} \frac{c_\rL}{12}\tilde\vartheta(\mu)+\tilde\varpi(\mu) \Big) \d\tilde\nu_\kappa(h)
\end{align*}
which coincides with the first identity in \eqref{E:T_ghost}.

\bibliographystyle{alpha}
\bibliography{bpz}
\end{document}